\documentclass[10pt,a4paper]{article}
\usepackage{typearea}
\usepackage{amsmath}
\usepackage{amsfonts}
\usepackage{amssymb}
\usepackage{amsthm}
\usepackage{amscd}
\typearea{11}
\author{Andreas Klein}
\title{The Eta invariant on the Milnor fibration of a quasihomogeneous polynomial}
\newcommand{\Gr}{{\rm Gr}}

\newcommand{\etab}{\widetilde{\eta}}

\newcommand{\R}{\mathbb{R}}
\newcommand{\del}{\partial}

\newcommand{\C}{\mathbb{C}}

\newcommand{\tr}{\mathrm{trace}}

\newcommand{\im}{\mathrm{im}}
\newcommand{\spec}{\mathrm{spec}}
\renewcommand{\ker}{\mathrm{ker}}

\newtheorem{theorem}{Theorem}[section]
\newtheorem{Def}[theorem]{Definition}
\newtheorem{prop}[theorem]{Proposition}
\newtheorem{asser}[theorem]{Assertion}
\newtheorem{lemma}[theorem]{Lemma}
\newtheorem{folg}[theorem]{Corollary}

\begin{document}
\maketitle
\begin{abstract}We calculate the eta-invariant for the odd signature operator relative to a specific submersion metric on the Milnor fibration of a quasihomogeneous hypersurface singularity using certain global boundary conditions in terms of the data of its fibre intersection form, monodromy and variation mapping resp. the monomial data of its Milnor algebra. This is done by representing this eta-invariant as the eta-invariant of the odd signature operator on a certain closed fibrewise double of the original bundle and expressing the latter as the mapping cylinder of a specific fibrewise isometry. In this situation, well-known cutting and pasting-laws for the Eta-invariant apply and give equality (modulo the integers) to a certain real-valued Maslov-type number, first introduced by Lesch and Wojciechowski, whose value in this case is a topological invariant of the isolated singularity. We finally give an explicit formula for the eta-invariant in the case of Brieskorn polynomials in terms of combinatorial data.
\end{abstract}

\section{Introduction}
Let $f \in \mathbb{C}[z_0,\dots,z_n]$, $n\geq 1$ be a quasihomogeneous polynomial with isolated singularity in $0 \in \mathbb{C}^{n+1}$, that is there are integers $\beta_0, \dots\beta_n,\beta > 0$ such that $f(t^{\beta_0} z_0,\dots,t^{\beta_n} z_n)=t^\beta f(z_0,\dots,z_n)$ for any $t\in \mathbb{C}^*$. In \cite{klein3} resp. \cite{klein1}, we observed that a certain set of differences of eta-invariants $\eta_\Delta(i)$ resp. a certain set of spectral flows ${\rm SF}(\alpha(i)),\ i=1,\dots, \mu$ for the odd signature operstor on the Milnor fibration of $f$ are, being determined by the spectrum of $f$ (cf. Definition \ref{spectrumdef}), topological invariants, provided $n \leq 2$. Here $\mu$ is the Milnor number of $f$ and the $\alpha(i) \in \Lambda \subset \mathbb{N}^{n+1}, |\Lambda|=\mu$ determine a monomial basis $z^{\alpha(1)},\dots, z^{\alpha(\mu)}$ of its Milnor algebra
\begin{equation}\label{mf}
M(f)=\mathcal{O}_{\mathbb{C}^{n+1},0}/(\frac{\partial f}{\partial z_0},\dots,\frac{\partial f}{\partial z_n})\mathcal{O}_{\mathbb{C}^{n+1},0}.
\end{equation}
Now, while currently it is unknown if these invariants, i.e. the spectrum of an isolated quasihomogeneous singularity, are topological invariants for $n \geq 3$, it is well-known by results of Le-Ramanujam and Varchenko (see again Saeki \cite{saeki} and references therein) that if two quasihomogeneous polynomials $f$ and $g$ with an isolated singularity can be connected by a $\mu$-constant deformation (cf. Varchenko \cite{varchenko2}), then they can be connected by a deformation of constant topological type and $f$ and $g$ have the same spectrum and the same weights and these four conditions are in fact equivalent (see Theorem \ref{foureq}, Section \ref{relcohom}). In especially the sets $\eta_\Delta(i)$ resp. ${\rm SF}(\alpha(i))$ for $i =1,\dots,\mu$, the latter being equivalent to the pectrum of $f$, are invariant under $\mu$-constant deformation for all $n\geq 1$. One could now pose the question if such an invariance is manifest not only on the level of 'differences of eta-invariants' and spectral flows, but \\

{\it Given a $\mu$-constant deformation connecting quasihomogeneous polynomials $f_0$ and $f_1$ with an isolated singularity at the origin, are there 'well-posed' boundary projections $P_{0,1} \in {\rm Gr}(A)$ so that for the associated eta-invariants on the Milnor bundles $Y_0, Y_1$, equipped with appropriate submersion metrics, we have $\eta (D_0,P_0)=\eta (D_1,P_1) \in \mathbb{R}$? Furthermore, is $\eta (D,P)$ determined by the topological type of $f$ for appropriate $P \in {\rm Gr}(A)$?}\\

Recall that for a closed manifold $Y$ of dimension $2n+1$ and an elliptic self-adjoint differential operator of first order with compact resolvent $D$ on $Y$ its {\it eta-function} as introduced by Atiyah et al. (\cite{Atiyah}) with the objective to generalize Hirzebruch's result on the signature defect of Hilbert modular cusps (\cite{atiyahdonelly}) is given for ${\rm Re}(s)>>0$ by the holomorphic function on the half-plane
\[
\eta(D)(s):= Tr(D|D|^{s-1})=\sum_{i \in \mathbb{Z}}{\rm sign}(\lambda_i)|\lambda_i|^{-s},
\]
where $\{\lambda_i\}_{i\in \mathbb{Z}}$ are the (non-zero) eigenvalues of $D$ (counting multiplicity), and is seen to have a meromorphic continuation to $\mathbb{C}$ with at most a simple pole at the origin (see Section \ref{begriffe1}). Now in the compact case its residue at zero is shown to be a (locally computable) homotopy invariant of $D$, hence a function of the stable class defined by the symbol of $D$ in $K^1(TY)$. This combined with the signature Theorem in \cite{Atiyah}) already implies the vanishing of the residue of $\eta(s)$ at zero for quite a general class of (pseudo)differential-operators (see \cite{atiyah3} and the remarks in Section \ref{begriffe1}), i.e. all Dirac operators in odd dimensions. Note that from its representation as a 'boundary correction' term in the Index Theorem it also follows using simple examples that $\eta(0)$ is not multiplicative under finite coverings, hence a {\it global} invariant of $Y$ (\cite{Atiyah}). Note further that if $\partial X=Y$, the odd signature operator $D$ on $Y$ appears as the tangential operator, restricted to {\it even} forms, of the signature operator $\tilde D^+=d +d^*:\Omega^+(X)\rightarrow \Omega^-(X)$, where $\pm$ denote the $\pm 1$-eigenspaces of the involution $\tau=(-1)^r*$, where $r$ depends on $n$ and the degree of the form. In our case, returning to the Milnor bundle $Y$ of a quasihomogeneous polynomial $f:\mathbb{C}^{n+1}\rightarrow \mathbb{C}$, for $P\in {\rm Gr}(A)$ it follows that $\eta(D_P)(s)$ has in general a pole of order $2$, which is why we will restrict ourselves to the class ${\rm Gr}_\infty(A)\subset {\rm Gr}(A) $ of self-adjoint boundary conditions which differ from $P_{>0}$, the projection onto the positive eigenvalues of $A$, by a smoothing operator (cf. Section \ref{begriffe1}). Now assume for the moment $Y$ had closed fibres and consider the family of signature operators $D_y, \ y \in S^1$ along the fibres.  Then following a scheme of Quillen (\cite{quillen}) one can make sense of a complex line bundle $\lambda_Y$ over $S^1$ whose fibres are canonically isomorphic
to
\[
\lambda_{Y_y}={\rm det}({\rm ker}\ D_y)^*\otimes {\rm det}({\rm coker}\ D_y), \ y \in S^1,
\]
and which comes equipped with a distinguished norm involving a zeta-regularized determinant, the Quillen norm. Now let $c:[0,1]\rightarrow S^1$ and $\tilde Y=c^*Y$. Then $\partial \tilde Y= Y_{c(0)}\sqcup Y_{c(1)}$, the index of the tangential operator $D_{\partial \tilde Y}$ vanishes and one has an involution $\tau=(-1)^r*$ on ${\rm ker}\ D_{\partial \tilde Y}$ whose $\pm 1$-eigenspaces we denote by $K^\pm$. By choosing an isometry $T:K^+\rightarrow K^-$ w.r.t. to the hermitian structure on $K^\pm$ given by $i\tau$ we have an associated element $P(T)\in Gr_\infty(D_{\partial \tilde Y})$ given by the orthogonal projection onto ${\rm im}(P_{>0})\oplus {\rm graph}(T)$ and an associated eta-invariant $\eta(D, P(T))$, where we assume $Y$ is equipped with a submersion metric $g$. Now the intruiging property of Quillen's determinant line is that setting $\tilde \eta(D,P(T))=(\eta(D, P(T))+ {\rm dim \ ker} D_{P(T)})/2$ we have a representative (cf. Dai and Freed \cite{daifreed}, note this uses essentially Scott-Wojchiechowski's formula \cite{scottwoj} alluded to above)
\[
\tau_{\tilde Y} :=e^{2\pi i\tilde \eta(D,P(T))}\left(\prod_{\lambda_i >0}\lambda_i\right)^{1/2}({\rm Det}\ T)^{-1} \in \lambda_{\partial \tilde Y}^{-1}
\]
where $\lambda_{\partial \tilde Y}=(c|_{\{0,1\}})^*\lambda_Y$ and the product is defined via the zeta-determinant ${\rm exp}(-\zeta'(0))$ (\cite{quillen}). Then since $\lambda_{\partial \tilde Y}=\lambda_{Y_{c(0)}}\otimes \lambda_{-Y_{c(1)}}$ and $\lambda_{-Y_{c(1)}}=\lambda_{Y_{c(1)}}^{-1}$ this amounts to a map
\[
\tau_{\tilde Y}: \lambda_{Y_{c(0)}} \rightarrow \lambda_{Y_{c(1)}}.
\]
Now 'blowing up' the base part of the submersion metric by a factor $1/\epsilon^2$ for small $\epsilon >0$, calling it $g(\epsilon)$ and taking the (existing) limit ${\rm lim}_{\epsilon \rightarrow 0} (\tau_{\tilde Y}(\epsilon))$ gives exactly the parallel transport induced by the Quillen connection on $\lambda$, as is shown by Dai and Freed (\cite{daifreed}). Consequently we have for $c(0)=c(1)$ that (cf. \cite{daifreed}, eq. (5.6), using the 'bounding' spin-structure)
\[
{\rm lim}_{\epsilon \rightarrow 0} (\tau_{Y}(\epsilon))={\rm Tr}_{Y_{c(0)}}({\rm lim}_{\epsilon \rightarrow 0} \tau_{\tilde Y}(\epsilon))= {\rm holonomy \ around}\ c,
\]
where $\tau_{Y}(\epsilon)$ denotes the exponentiated (reduced) eta-invariant for $g(\epsilon)$ and ${\rm Tr}_{Y_{c(0)}}$ is a certain supertrace (see \cite{daifreed}, eq. (2.6)). Note that here the essential ingredient is a glueing law for the exponentiated eta-invariant, that is $\tau_{Y}={\rm Tr}_{Y_{c(0)}}(\tau_{\tilde Y})$, where the isometry $T$ which is implicit on the right hand side can be fixed to be induced by an identifying isometry $\Phi:Y_{c(0)}\simeq  Y_{c(1)}$, noting that $K^+\oplus K^-\simeq H^*(\partial \tilde Y, \mathbb{C})\simeq H^*(Y_{c(0)}\sqcup -Y_{c(0)}, \mathbb{C})$. Now observe that the weighted circle action $\sigma$ on $Y$ identifies the fibres isometrically, hence choosing the submersion metric on $Y$ which is induced by the horizontal distribution spanned by the Killing vector field of $\sigma$ on $Y$ (the Euler vector field) and by restriction of the euclidean metric on $\mathbb{C}^{n+1}$ to the fibres, the corresponding isometry $T$ can be chosen to be induced by $id \sqcup \sigma(1/\beta)$, acting on 'the diagonal' in $K^+\oplus K^-\simeq H^*(Y_{c(0)}\sqcup -Y_{c(0)}, \mathbb{C})$. But then it simply follows that (for details see Section \ref{etaquas})
\begin{equation}\label{taueta}
{\rm lim}_{\epsilon \rightarrow 0} (\tau_{Y}(\epsilon))={\rm det}(T).
\end{equation}
Indeed, even more is true: since the fibres with respect to the submersion metric induced by the Euler vector fields are totally geodesic, the connection induced on $\lambda_Y$ by the Levi-Civita-connection of $g(\epsilon)$ is essentially the Gauss-Manin connection, furthermore since the second fundamental form of the fibres vanishes we have by a formula of Bismut and Freed (\cite{bisfreed}, (3.49))
\begin{equation}\label{etaconst}
\frac{\partial}{\partial \epsilon}\tilde \eta(D)(\epsilon)=\int_Y\hat A\left(\frac{R^{T^vY}}{2\pi}\right),
\end{equation}
where $\hat A$ is a certain $O(n)$-invariant polynomial and $R^Z$ is the curvature of the vertical tangent bundle $T^vY$ of $Y$, which is independent of $\epsilon$. Since the limit of $\tilde \eta(D)(\epsilon)$ exists at least in $\mathbb{R}/\mathbb{Z}$ due to the Index theorem and the convergence of the curvature (in fact, it exists in $\mathbb{R}$, since the set ${\rm ker}(D_y), \ y \in S^1$ forms a vector bundle on $S^1$, the $D_y$ being the signature operators on the fibres and the kernel of $D(\epsilon)$ is of constant dimension (\cite{daieta})), (\ref{etaconst}) simply says that $\eta(D)(\epsilon)$ is independent of $\epsilon$, hence the adiabatic limit in (\ref{taueta}) can be replaced by the exponentiated eta-invariant $\tau_Y$ itself. Now in the above discussion we ignored the non-empty boundaries of the fibres of $Y$. To overcome this problem, one can modify the metric in a boundary neighborhood $U$ of $Y$ slightly (Lemma \ref{product123}), so that this neighborhood becomes isometric to the metric product $U_y \times S^1$ for some $y$, where $U_y =Y_y\cap U$. Then by glueing the metric product $Y_0:=Y_y\times S^1$ to $Y$ along their common boundaries with opposite orientation, we get a bundle $Y^e$ with closed fibres $F^e$ and submersion metric whose fibres are still totally geodesic and whose algebraic monodromy, that is the action of the time-one flow of the horizontal lift of $\partial/\partial t$ to $Y^e$ on $U^e=H^*(F^e, \mathbb{C})\simeq U\oplus U^*$, where $U=H^*(Y_y, \mathbb{C})$, is given with respect to the latter splitting by (cf. Lemma \ref{betaisom})
\begin{equation}\label{var}
\rho^e=\begin{pmatrix} I & 0 \\ V & \rho \end{pmatrix},
\end{equation}
where $\rho$ and $V$ are algebraic monodromy resp. variation mapping of $Y$ (see Appendix A). Note that replacing $Y$ by $Y^e$ corresponds in some sense to a Theorem by Nemethi (\cite{nem1}), which states that any $\epsilon$-hermitian variation structure $(U, b, h^*, V)$ can be represented by an $\epsilon$-hermitian isometric structure $(U^e, b^e, \rho^e)$ ($\epsilon=\pm 1$), that is, $b^e$ is an $\epsilon$-hermitian {\it nondegenerate} form. Now the eta-invariant on the trivial part of $Y^e$, $\eta(D_0; P_0)$ vanishes when using the Calderon projection $P_0$ of $Y_0$ and applying a well-known glueing law for the eta-invariant (see Bruening/Lesch \cite{brules2}) yields the exact equality
\begin{equation}\label{etabruening}
\tilde \eta(D_{Y^e})=\tilde \eta(D,Id-P_0)-\tilde \eta(D_0, P_0),
\end{equation}
so from the above discussion we already have the following result which we here state as an assertion, since the proof of our formulas on the eta invariant follows a slightly different strategy:
\begin{asser}
The reduced eta-invariant $\tilde \eta(D,I-P_0)$ of the odd signature operator $D$ on the Milnor bundle $Y$ of a quasihomogeneous polynomial $f:\mathbb{C}^{n+1}\rightarrow \mathbb{C}$, with respect to the submersion metric given by the Euler vector field, $P_0$ being the Calderon projector of the odd signature operator on the metric product (identifying boundaries with opposite orientation) depends modulo the integers only on the variation structure of $Y$ and is given by
\[
e^{2\pi i\tilde \eta(D,I-P_0)}= {\rm det}(\tilde T)\ {\rm mod}\ \mathbb{Z},
\]
where there graph of $\tilde T:K^+\rightarrow K^-$ in $K^+\oplus K^- \simeq H^*(F^e\bigcup -F^e,\mathbb{C})=:U^e$ is given by the image of the diagonal in $U^e\oplus U^e$ under $Id\oplus\rho^e$, $\rho^e$ as in (\ref{var}). Consequently, for $n \geq 2$, $[\tilde \eta(D, I-P_0)] \in \mathbb{R}/\mathbb{Z}$ is determined by the topological type of $f$.
\end{asser}
In fact, instead of using the holonomy theorem as outlined in the argumentation above, we will use in Section \ref{etaquas} refined glueing laws for the eta-invariant, as they were systematically approached by Lesch and Kirk in \cite{lesch}, based on the result (\ref{etabruening}), these glueing theorems (see Theorem \ref{invertiblecase}) allow for an explicit control of the $\mathbb{Z}$-part of the eta-invariant. As a consequence, it turns out that the above arguments remain valid in $\mathbb{R}$. Instead of listing the various formulas obtained in Sections \ref{etaquas} and \ref{brieskorn} giving $\tilde \eta(D,I- P_0)$ in terms of the combinatorial data of a monomial base of $M(f)$ (Theorem \ref{eta}) resp. in terms of the exponents of $f$ if $f$ is a Brieskorn polnomial (Theorem \ref{etab}), we will come back to our discussion on the spectrum of $f$
${\rm sp}(\alpha(i))=(l(\alpha(i))-1), \ i=1, \dots, \mu$ (cf. Def. \ref{spectrumdef}) and state that:
\begin{theorem}\label{etathm}
Let $\eta(D, I-P_0)$ be the eta-invariant of the odd signature operator $D$ on $Y$ (with the metric described in Lemma \ref{product123}) as calculated in Theorem \ref{eta}. We then have
\begin{equation}
\eta(D,I-P_0)=\sum_{\alpha \in \Lambda, {\rm sp}(\alpha) \notin \mathbb{Z}}(-1)^{[{\rm sp}(\alpha)]+n+1}\left(1-2\{{\rm sp}(\alpha)\}\right) +\frac{{\rm arg}(-1+\frac{4}{3}i)}{\pi}\sum_{\alpha \in \Lambda, {\rm sp}(\alpha)\in \mathbb{Z}}(-1)^{{\rm sp}(\alpha)+n+1},
\end{equation}
where $[\cdot]$ denotes the integer part, $\{\cdot\}$ the image in $\mathbb{R}/\mathbb{Z}$. Furthermore the above eta-invariant equals the eta-invariants given by the $APS$-boundary condition $P^+(\Lambda_Y)$ induced by the space of 'limiting values of extended $L^2$-solutions' of $D$ on $Y$, $\Lambda_Y\subset {\rm ker}(A)$ and relative to the adiabatic limit of the Calderon-projector $P_{Y}^{\infty}$ on $Y$:
\[
\eta(D, I-P_{0}) = \eta(D,P^+(\Lambda_Y))  = \eta(D,P_{Y}^{\infty})
\]
where both equalities are in fact valid in $\mathbb{R}$.
\end{theorem}
Note here, we set ${\rm arg}(re^{i\theta})=\theta \in[0,2\pi), r>0$. In Theorem \ref{eta2}, we give a formula for the eta-invariant relative to {\it arbitrary} $APS$-boundary conditions. Note that for $n=2k, k\in \mathbb{N}$, so ${\rm dim}\ Y=4k+1$, there is a skew-complex involution $C$ which anti-commutes with $D$, so $\eta(D_{Y^e})$ and hence $\eta(D, I-P_{0})$ vanishes. In fact the vanishing of the eta-invariant in this case corresponds to the fact that the spectral numbers ${\rm sp}(\alpha(i))$ are symmetric relative to the point $(n+1)/2-1$ by \cite{varchenko4}. Thus if $n$ is even, the signs in each of the two sums in the expression for $\eta(D,I-P_0)$ given in Theorem \ref{etathm} cancel. Further the above formula resp. Lemma \ref{glue} shows that $\eta(D,I-P_0)\in \mathbb{R}$ is determined completely by the variation structure of $f$, hence its complex Seifert form (\ref{seifertform}) and monodromy and is thus for $n\geq 1, \ n\in \mathbb{N}$ determined by the topological type of $f$. Further since by the above discussion the spectrum is constant under $\mu$-constant deformations, we can deduce from Theorem \ref{etathm} directly the answer to our above question:
\begin{folg}\label{deformation}
Let $n\geq 2 \in \mathbb{N}$ and let $f_0$ and $f_1$ be quasihomogeneous polynomials being connected by a $\mu$-constant deformation, then for the corresponding eta-invariants on $Y_0,\ Y_1 $ we have $\eta(D_0,P^+(\Lambda_{Y_0}))=\eta(D_1,P^+(\Lambda_{Y_1}))$ with respect to the APS-boundary-projectors $P^+(\Lambda_{Y_{0,1}})$.
\end{folg}
Certainly this is already implied by the topological invariance of $\eta(D,P^+(\Lambda_Y))$ and the fact that the topological type of $f$ is invariant under $\mu$-constant deformation (\cite{ramanujam}, also Theorem \ref{foureq}). On the other hand, the formula in Theorem \ref{etathm} expresses the eta-invariant in terms of the spectrum of $f$ and thus gives a direct proof of its invariance under $\mu$-constant deformation using Varchenko's Theorem \ref{definvspec}. We will state finally Theorem \ref{etab} which gives a beautiful formula for the eta-invariant of a Brieskorn singularity $f:(\mathbb{C}^{n+1},0)\rightarrow (\mathbb{C},0)$ in terms of lattice point counting. Thus let $f= \sum_{i=1}^{n+1}z_i^{a_i}$, where $a_i\in \mathbb{N}_+$ and assume as before $n=2k$ or $n=2k+1, \ k\in \mathbb{N}$. Set
\[
\Lambda:=\{{\bf k} \in \mathbb{Z}^{n+1}|1\leq k_j \leq a_j-1\},\quad \Lambda_0:= \{{\bf k}\ \in \Lambda| \sum_{j=1}^{n+1}k_j/a_j \in \mathbb{Z}\}.
\]
Writing for any subset $\Lambda' \subset \Lambda$ the symbol $\sum_{\Lambda'}$ as the sum over all $n+1$-tuples ${\bf k}\subset \mathbb{Z}^{n+1}$ so that ${\bf k} \in \Lambda'$. Then we have the following:
\begin{theorem}\label{etabthm}
Let $f$ be a Brieskorn polynomial and let $D$ be the signature operator on its Milnor bundle $Y$ w.r.t the APS-boundary condition $P^+(\Lambda_Y) \in {\rm Gr}_\infty(A)$  in anology to Theorem \ref{etathm} above. Then the eta-invariant $\eta(D, P^+(\Lambda_Y))$ for $n=2k$ or $n=2k+1, \ k\in \mathbb{N}$ equals:
\[
\begin{split}
\eta(D, P^+(\Lambda_Y))  &= (-1)^n\sum_{\Lambda\setminus\Lambda_0} {\rm  sign}\ ({\rm sin}(\pi\sum_{j=1}^{n+1}k_j/a_j))\cdot (1-2\{\sum_{j=1}^{n+1}k_j/a_j\}) \\
&-\frac{{\rm arg}(-1+\frac{4}{3}i)}{\pi}\sum _{\Lambda_0}(-1)^{{\sum_{j=1}^{n+1}k_j/a_j}+n}.
\end{split}
\]
\end{theorem}
\tableofcontents

\section{Preparations}\label{chapter1}
\subsection{Eta-invariants on manifolds with boundary}\label{begriffe1}
Let $(X,g)$ be a closed odd-dimensional Riemannian manifold, $E \rightarrow X$ a complex hermitian or real vector bundle. Let $D$ be a generalized Dirac operator with a compatible connection $\nabla^E$ on $X$, that is $D$ is a first order formally self-adjoint elliptic differential operator satisfying a generalized Weitzenboeck-formula of the type
\[
D^2=(\nabla^E)^*(\nabla^E)+\mathcal{R},
\]
where $\mathcal{R}$ is a bundle endomorphism of $E$. As is well-known (cf. Berline-Getzler-Vergne (\cite{getzler}), $D$ is essentially self-adjoint and has discrete spectrum $\dots \lambda_{-1}< 0\leq \lambda_0\leq \lambda _1\dots$ (counting as usual with multiplicity). Now its {\it Eta-function} is defined as the for $Re(s) > n$ absolutely converging series, $n ={\rm dim} \ X$
$$\label{etafunction}
\eta (D)(s):= \sum_{i \in \mathbb{Z}}{\rm sign}(\lambda_i)|\lambda_i|^{-s},
$$
where $\rm{sign}(x)$ equals ${\rm sign}(x)$ for $x \neq 0$ and equal to zero for $x=0$. It is a well-known result using the Mellin transform and heat-equation methods (see e.g. Gilkey \cite{gilkey}) that $\eta (D)(s)$ can be meromorphically extended to $\mathbb{C}$ having a priori a simple pole at zero. That $0$ is actually a regular value, was proven by Atiyah et al. (\cite{atiyah3}) by associating to the self-adjoint symbol of $D$ its stable class $K^1(TX)$. Furthermore the residue $R(A)$ turns out to be a (locally computable) homotopy invariant of $D$, hence depends only on its stable class, on the other hand, one can show the stable class of $D=d+d^*$ being a generator of $K^1(TX)$. Then the finiteness of $\eta(0)$ follows as a direct application of the Index Theorem by choosing a manifold $Y$ so that an appropriate multiple of $X$ bounds $Y$. Noting this one sets
\begin{Def}
The eta-invariant of $D$ is defined as
\[
\eta(D):=\eta(D)(0).
\]
\end{Def}
Note that formally, $\eta(D)=\sharp\{\lambda_i >0\}-\sharp\{\lambda_i>0\}$.\\
Now consider a Riemannian manifold $(X,g)$ with boundary $\partial X \neq \emptyset$ that is of odd dimension, that is ${\rm dim }X=2n+1$, a hermitian vecor bundle $E \rightarrow X$ and a symmetric Dirac operator $D:C^\infty(E)\to C^\infty (E)$ as above.
Note that the symmetry of $D$ is measured with respect to the $L^2$ inner product on $X$, so that if
$\phi_1,\phi_2\in C^\infty_0(E)$ are supported in the interior of $X$ then
\[
(\phi_1,\phi_2)_X:=\int_X(D\phi_1,\phi_2)_{E_x}dx=\int_X(\phi_1,D\phi_2)_{E_x}dx.
\]
Assume a 'metric collar' neighbourhood of $\partial X$, that is a neighbourhood $U\subset X$ of $\partial X$, so that on $U$
\begin{equation}\label{metricol234}
[0, \epsilon)\times \partial X \hookrightarrow  X, \quad g=dr^2+g_{\partial X},
\end{equation}
for some appropriate isometry, where $g_{\partial X}$ is a metric on $\partial X$ independent of $r$.\\
Then it is well-known (see for instance \cite{brules2}) that the restriction of $D$ to the collar takes the form
$D=\gamma(\frac{d}{dx}+A)$, where $\gamma:E_{|\partial X}\to E_{|\partial X}$ is a bundle endomorphism and $A: C^\infty(E_{|\partial X})\to C^\infty(E_{|\partial X})$ is a first-order self-adjoint
elliptic differential operator on the closed manifold $\partial X$ (called the {\it tangential operator}) satisfying
\begin{equation}\label{properties546}
\gamma^2=-I, \quad \gamma^*=-\gamma,\quad \text{and}\quad \gamma A=-A\gamma.
\end{equation}
Note that $A$ is independent of $x$ for $x\in [0,\epsilon)$ due to (\ref{metricol234}). Now let $<\cdot,\cdot>_{\partial X}$ be the $L^2$ inner product on $\partial X$ induced by $g_{\partial X}$ an define
\[
\Omega(\phi,\psi)=<\phi, \gamma\psi>_{\partial X}
\]
for $\phi,\psi\in L^2(E_{|\partial X})$, then $\Omega$ is a hermitian symplectic form on $L^2(E_{|\partial X})$. Now assume that $D$ is symmetric on a subspace $\mathcal{D}_0\subset C^\infty(E)$, then by Greens formula one has for $f,g \in \mathcal{D}_0$
\begin{equation}\label{symmetry12}
(Df,g)_X\ -\ (f,D^*g)_X =\Omega(f|_{\partial X},g|_{\partial X}),
\end{equation}
so defining a projection $P:L^2(E_{|\partial X})\rightarrow L^2(E_{|\partial X})$ so that $I-P$ is the orthogonal projection onto (the closure of) $\mathcal{D}_0|_{\partial X}\subset C^\infty(E_{|\partial X})$ one has
\[
\mathcal{D}_0\subset \mathcal{D}_P :=\{f \in C^\infty(E)|P(f_{|\partial X})=0\}
\]
hence $D_P:= D| \mathcal{D}_P$ is as symmetric extension of $D|\mathcal{D}_0$. From (\ref{symmetry12}) one sees that $\gamma({\rm ran} \ (I-P)) \subset {{\rm ran} \ (I-P)}^\perp$ which implies by direct calculation (see for instance \cite{loyapark}) that
\[
I-P\leq \gamma^*P\gamma.
\]
Now calling a closed subspace $L\subset L^2(E_{|\partial X})$ {\it Lagrangian} iff $\gamma(L)=L^\perp$,  so if $L={\rm ran}\ P$ for some projection $P$ one has that $\gamma:{\rm ran} \ P \rightarrow {{\rm ran} \ P}^\perp$ is an isomorphism, one calculates (see again \cite{loyapark}) that $L$ being Lagrangian in $L^2(E_{|\partial X})$ is equivalent to
\[
I-P=\gamma^*P\gamma
\]
and in fact one sees (as the next theorem) that the latter condition is necessary for the operator $D_P=D|\mathcal{D}_P$ to be essentially self-adjoint in $L^2(E)$, but before fomulating this, we need a definition. For this let
\[
P_{>0}:L^2(E_{|\partial X})\longrightarrow L^2(E_{|\partial X})
\]
denote  the positive spectral projection for the self--adjoint tangential operator $A:C^\infty(E_{|\partial X})\to  C^\infty(E_{|\partial X})$; thus if $\{  \psi_\lambda \}$ is a  basis of $L^2(E_{|\partial X})$ with $A\psi_\lambda=\lambda \psi_\lambda$, then $P_{>0}(\sum  a_\lambda \psi_\lambda)=\sum_{\lambda >0}a_\lambda \psi_\lambda$. Then one defines the following class of projections in $L^2(E_{|\partial X})$ (cf. \cite{lesch}).
\begin{Def} \label{fredholmgr} Define the {\it self-adjoint Fredholm Grassmannian} ${\rm Gr}(A)$ to be the set of maps
$P:L^2(E_{|\partial X})\to L^2(E_{|\partial X})$  so that
\begin{enumerate}
\item $P$ is pseudo--differential of order $0$,
\item $P=P^*, P^2=P$, i.e. $P$ is an orthogonal projection,
\item $\gamma P\gamma^*=I-P$,
\item $(P_{>0}, P)$ form a Fredholm pair, that is,
\[
P_{>0|{\rm im}\  P}:{\rm im}\ P\to {\rm im}\  P_{>0}
\]
is Fredholm.
\end{enumerate}
\end{Def}
Note that the Grassmannian ${\rm Gr}(A)$ is topologized using the norm topology on bounded
operators. We then have the following result proved for instance in \cite{brules} or \cite{boos}.
\begin{theorem}\label{bvp}
Let $P \in {\rm Gr}(A)$, then $D$ with domain
\[
\mathcal{D}_P :=\{f \in C^\infty(E)|P(f_{|\partial X})=0\}
\]
is a 'well-posed' boundary value problem in the sense of Seeley (\cite{seeley}), i.e.$D|\mathcal{D}_P$ is essentially self-adjoint in $L^2(E)$. If $D_P$ denotes the closure of $D| \mathcal{D}_P$, then its domain is given by
\[
\mathcal{D}(D_P)=\{f \in H_1(E)|P(f_{|\partial X})=0\}\subset L^2(E),
\]
furthermore $D_P$ is Fredholm and has compact resolvent, in particular its spectrum is discrete and each eigenvalue has finite multiplicity.
\end{theorem}
Now using the cobordism theorem (see Palais \cite{pal}) one concludes on  the validity of the hypothesis in the implication
\[
{\rm sign}(i\gamma)|_{{\rm ker}\ A}=0 \Rightarrow {\rm dim \  ker}(\gamma-iI)= {\rm dim \  ker}(\gamma+iI).
\]
It follows that $(ker A, \gamma)$ is a symplectic vectorspace and there exist {\it Lagrangian} subspaces
\[
L \subset {\rm ker}\ A \quad {\rm s.t.}\ \gamma(L)=L^\bot\cap {\rm ker} \ A.
\]
That the above hypothesis is true follows using the splitting $L^2(E_{|\partial X})=H^+ \oplus H^-$ into the $\pm i$-eigenspaces of $\gamma$. One has as a consequence of (\ref{properties546})
\[
A=\begin{pmatrix}0& A_{-} \\ A_{+} &0\end{pmatrix},
\]
and $A_+$ is a Fredholm operator with index
\[
{\rm ind } \ A_+={\rm ker}\ A\cap {\rm ker }(\gamma -i) - {\rm ker}\ A \cap {\rm ker }(\gamma +i),
\]
but the vanishing of ${\rm ind} \ A_+$ follows from the cobordism theorem. Given such a Lagrangian subspace $L\subset {\rm ker}\ A$ define an orthogonal projection $P_+(L):L^2(E_{|\partial X})\rightarrow L^2(E_{|\partial X})$ by
\[
P_+(L)={\rm proj}_L+P_{>0}.
\]
Then because of $L$ being Lagrangian, it follows that $P_+(L)$ satisfies the third condition in Definition \ref{fredholmgr}, since it differs from $P_{>0}$ by a finite dimensional spectral projection which is smoothing, $P_+(L)$ is as $P_{>0}$ pseudodifferential of order zero and the pair $(P_{>0}, P_+(L))$ is Fredholm, so $P_+(A)\in {\rm Gr}(A)$ and it is even more, it is an element of the set ${\rm Gr}_\infty(A)\subset {\rm Gr}(A)$
\[
{\rm Gr}_\infty(A)=\{P \in {\rm Gr}(A)| P-P_{>0} \ {\rm is \ a\ smoothing\ operator}\}.
\]
Note that here as in Definition \ref{fredholmgr}, one could replace $P_{>0}$ by any pseudo-differential projection $Q$ such that $P_{>0}-Q$ is smoothing since for $P,Q,R$ orthogonal projections in a Hilbert space $H$ and $Q-R$ compact, $(P,Q)$ is a Fredholm pair if and only if $(P,R)$ is a Fredholm pair. $P_+(L)$, which depends only on the choice of $L$ and on $A$, thus on boundary data, is called the {\it Atiyah-Patodi-Singer} (APS)-boundary projection associated to $L$, in contrary, the following projection $P_X \in Gr(A)$, called the Calderon projector associated to $(D,X)$ depends on all of $X$ and $D$. Before stating the definition we need a version of the 'unique continuation property' for Dirac operators (see \cite{boos}, Theorem 8.2).
\begin{theorem}\label{UCP}
let $X=X_+\cup X_-$ be a connected partitioned manifold with a hermitian bundle $E$ and Dirac-type operator $D$ as above so that $X_+\cap X_-=\partial X_{\pm}=Y$. Let $s \in C^\infty(E)$ satisfy $Ds=0$ and $s|_Y=0$. Then $s=0$ on $X$.
\end{theorem}
Now the {\it Calderon  projector} $P_X$ is defined as the orthogonal projection onto the {\it Cauchy data space}
\begin{equation}\label{calderon}
L_X:=r\left(\ker D : H_{1/2}(E)\longrightarrow
H_{-1/2}(E)\right)\subset L^2(E_{|\partial X}).
\end{equation} Here $r$ denotes the
restriction to the boundary. That the trace operator $r$ defines a bounded
map from the $H_{1/2}$--kernel of $D$ into $L^2(E_{|\partial X})$ is proved for instance in Boos' monograph ( \cite{boos}), furthermore, it is proved in \cite{scott} that the Calderon projector $P_{X}= {\rm proj}_{L_X}$ lies in ${\rm Gr}_\infty(A)$.
The above cited unique continuation property for $D$ implies that
\[
r:\left(\ker D :H_{1/2}(E)\longrightarrow H_{-1/2}(E)\right)\longrightarrow L^2(E_{|\partial X})
\]
is injective, so that to any  element $f$ in the image of $P_X$ we can assign a unique solution to $D\phi=0$ on $X$ with $\phi\in H_{1/2}$ and $r(\phi)=f$, this obervation is generalized by the following Lemma, taken from Lesch/Kirk (\cite{lesch}):
\begin{lemma}\label{kerdp} Let $P\in{\rm Gr}(A)$. Then
\[
{\rm ker} P_{|{\rm im} P_X}={\rm im} P_X\cap \ker P=\gamma({\rm ker} P_X)\cap  {\rm ker}\  P
\]
and this space is isomorphic to  the kernel of
$D_P$.  Thus $D_P$ is invertible if and only if ${\rm im} P_X\cap {\rm ker} P=0$. In
particular $D_{P_X}$ is invertible.
\end{lemma}
\begin{proof} The first two assertions follows directly from the fact that if $\phi\in{\rm ker} D_P$, the restriction of $\phi$ to the boundary  lies in the intersection of ${\rm ker} P$ and the image of $P_X$ and the unique continuation property for $D$. The third follows since $P_X$ is a self-adjoint projection.
 \end{proof}
Note that ${\rm Gr}(A)$ as well as ${\rm Gr}_\infty(A)$ are path-connected, more precisely one has (c.f. \cite{lesch},\cite{DouWoj}):
\begin{prop}\label{connected} The Grassmannians ${\rm Gr}(A), {\rm Gr}_\infty(A)$ are path
connected. For a fixed $P\in{\rm Gr}_\infty(A)$ (resp. ${\rm Gr}(A)$) the space
$\left\{Q\in{\rm Gr}_{(\infty)}(A)|\ {\rm ker} Q\cap{\rm im} P=0 \ \right \}$  is path connected.
\end{prop}
The next lemma (see again \cite{lesch}) that represents the Lagrangian Grassmannians in a certain space of unitary mappings will be crucial for calculations. For this consider the bundle endomorphism $\gamma:E_{|\partial X}\to E_{|\partial X}$ which induces a decomposition of
$E_{|\partial X}=E_i\oplus E_{-i}$ into the $\pm i$ eigenbundles and consequently we get a decomposition of
$L^2(E_{|\partial X})$ into the $\pm i$ eigenspaces,
\begin{equation}\label{decompL2}
L^2(E_{|\partial X})= L^2(E_i)\oplus L^2(E_{-i})=:\mathcal{E}_{i}\oplus \mathcal{E}_{-i}.
\end{equation}
For the following (cf. \cite{lesch}) a pair of projections $(P,Q)$ will be called invertible, if $P$ restricted to the image of $Q$ is an isomorphism onto the image of $P$.
\begin{lemma}\label{grassmannian2}
Let $\mathcal{U}(\mathcal{E}_i,\mathcal{E}_{-i})$ denote the set of 0th order pseudo-differential isometries from  $\mathcal{E}_i$ to $\mathcal{E}_{-i}$. Then the above decomposition of $L^2(E_{|\partial X})$ gives rise to a mapping
\begin{equation}\label{phimap}
\Phi: {\rm Gr(A)}\longrightarrow \mathcal{U}(\mathcal{E}_i,\mathcal{E}_{-i}),
\end{equation}
which is given by representing $P \in {\rm Gr}(A)$ as
\[
P=\frac{1}{2}\begin{pmatrix} I&\Phi(P)^*\\ \Phi(P)&I\end{pmatrix}.
\]
Conversely, given such an isometry $T\in \mathcal{U}(\mathcal{E}_i,\mathcal{E}_{-i})$,  then
\[
\frac{1}{2}\begin{pmatrix} I&T^*\\ T&I  \end{pmatrix}
\]
is a pseudo-differential projection satisfying the properties 1. - 3. of Definition \ref{fredholmgr}. Furthermore given projections $P,Q$ satisfying these properties one has
\begin{enumerate}
\item $(P,Q)$ form a Fredholm pair if and only if $-1\not\in
{\rm spec} _{\rm ess} \Phi(P)^*\Phi(Q)$,
\item $(P,Q)$ is invertible if and only if $-1\not\in
{\rm spec} \Phi(P)^*\Phi(Q)$,
\item ${\rm ker} P\cap {\rm im} Q$ is canonically isomorphic to ${\rm ker} (I+\Phi(P)^*\Phi(Q))$,
\item $P-Q$ is smoothing if and only if $\Phi(P)^*\Phi(Q)-I$ is smoothing.
\end{enumerate}
In particular, if $Q=P^+(L)$ for some Lagrangian $L\subset {\rm ker} A$, then
$P\in{\rm Gr}(A)$ if and only if $-1\not \in{\rm spec}_{\rm ess} \Phi(P)^*\Phi(Q)$.
Define now
\begin{equation}
\mathcal{U}_{\rm Fred}=\left\{U\in \mathcal{U}|-1\not\in{\rm spec}_{\rm ess} U \right\},
\end{equation}
and
\begin{equation}
\mathcal{U}_{\infty}=\left\{U\in \mathcal{U}_{{\rm Fred}}|U-I {\rm \  is\ a\ smoothing\
operator}\right\}.
\end{equation}
Then given any $P\in {\rm Gr}_\infty(A)$,the map
\begin{equation}
U\mapsto \frac{1}{2}\begin{pmatrix} I&(\Phi(P)U) ^*\\ \Phi(P)U &I \end{pmatrix}
\end{equation}
defines homeomorphisms
\[
\mathcal{U}_{\rm Fred}\longrightarrow {\rm Gr}(A),\quad \mathcal{U}_{\infty}\longrightarrow {\rm Gr}_\infty(A).
\]
\end{lemma}
Following the original picture of Lagrangian subspaces as images of orthogonal projections we define $\mathcal{L}$ to be the set of Lagrangian subspaces of $L^2(E_{|\partial X})$ whose associated projections are pseudo-differential of order $0$. Accordingly, the Cauchy data space, $L_X$, (the image  of the Calderon-projector) is a Lagrangian subspace of $L^2(E_{|\partial X})$. Then from the above considerations it follows immediately that if one defines
\begin{equation}\label{fredholmlag}
\mathcal{L}_{{\rm Fred}}=\left\{L\in \mathcal{L}|(L,\gamma(L_X)){\rm  is\ a\ Fredholm\ pair\ of\ subspaces}\right\},
\end{equation}
where here, $(L_1, L_2) \in \mathcal{L}^2$ is a Fredholm pair if and only if $L_1\cap L_2$ is of finite dimension and  $L_1+L_2$ is closed with finite codimension, and
\begin{equation}
\mathcal{L}_{\infty}=\{L\in \mathcal{L}_{\rm Fred}|{\rm proj}_L-{\rm proj}_{L_X}
{\rm {\ is\ a\ smoothing\ operator}}\}.
\end{equation}
that one has homeomorphisms
\[
\mathcal{L}_{{\rm Fred}}\longrightarrow {\rm Gr}(A)
\] and
\[
\mathcal{L}_{\infty}\longrightarrow {\rm Gr}_\infty(A)
\]
which are given by translating $L\in \mathcal{L}_{{\rm Fred}}$ into $P\in {\rm Gr}(A)$ so that $L={\rm im}\ P$, note on the other hand that ${\rm ker}\ P$ for $P\in {\rm Gr}(A)$ is a Lagrangian subspace since it is orthogonal to $\gamma({\rm im}P)$ and is given as the graph of the unitary mapping $-\Phi(P)$:
\[
{\rm ker}\ P=\{\begin{pmatrix}x\\-\Phi(P)x\end{pmatrix}|x\in\mathcal{E}_i\}\subset
L^2(E_{|\partial X}).
\]
As we saw above, for $P \in {\rm Gr}(A)$, $D_P$ is self-adjoint in $L^2(E)$ with compact resolvent, hence we can define $\eta(D_P;s)$ for ${\rm Re}(s)>>0$ as in \ref{etafunction}. Then (see \cite{DouWoj}, \cite{brules2} and \cite{Grubb}, \cite{Woj}):
\begin{theorem}\label{ML-S3.1} For $P\in{\rm Gr}(A)$ the function
$\eta(D_P;s)$ extends meromorphically to the whole complex plane with poles of order
at most $2$. If $P\in{\rm Gr}_\infty(A)$ then $\eta(D_P,s)$ is regular at $s=0$.
\end{theorem}
Hence one sets:
\begin{Def}\label{ML-S3.2} For $P\in {\rm Gr}(A)$, the $\eta$-invariant of $D_P$,
$\eta(D_P)$, is defined to be the constant term in the Laurent
expansion of $\eta(D_P;s)$ at $s=0$, i.e.
\[
\eta(D_P,s)= a s^{-2} + b s^{-1} + \eta(D_P) + O(s),
\]
while for further use, we denote
\[
\tilde \eta(D_P)=\left(\eta(D_P) + {\rm dim} \ker D_P\right)/2
\]
as the {\it reduced} eta- invariant of $D_P$.
\end{Def}
Note that the 'standard source' of such Laurent expansions are short time asymptotic expansions of ${\rm tr}(D_Pe^{-tD_P^2})$ (see for instance \cite{brules2}).\\
We will now see how the Scott-Wojciechowski theorem (see \cite{scottwoj}) relates the Fredholm determinant over the boundary Grassmannian to the dependence of $\eta$-invariants on the boundary condition. This together with a result of Bruening /Lesch (\cite{brules2}) gives the first version of 'glueing law' for eta-invariants used in later sections, the exposition is taken from \cite{lesch} and \cite{brules2}.\\
Let $H$ be a separable Hilbert space, we will say $T:H\rightarrow H$ is of {\it determinant class} if it is of the form $T=1+A$, where $A$ is of trace class. For then $A$ has as sequence of eigenvalues $\{\lambda_k\}$ (with multiplicities) so that $\sum_k\lambda_k <\infty$ and we can define
\[
{\rm det}_{\rm F}(T)={\rm det_F}(1+A)=\prod_k(1+\lambda_k).
\]
We cite the following Theorem which is taken from \cite{lesch} and is a direct consequence of the celebrated Scott-Wojciechowski-Theorem (see \cite{Woj}) which relates the ratio of $\zeta$-determinants (which we did not introduce here) to a Fredholm determinant on ${\rm Gr}_\infty(A)$. Note that for $P,Q\in {\rm Gr}(A)$, $\Phi(P)\Phi(Q)^* - 1$ is a smoothing operator (see Lemma \ref{grassmannian2}), hence it is of trace class, so $\Phi(P)\Phi(Q)^*$ is of determinant class. Consequently the determinant ${\rm det}_{\rm F}(\Phi(P)\Phi(Q)^*)$ is defined and lies in $U(1)$ since $\Phi(P)\Phi(Q)^*$ is unitary. Then one has the following result.
\begin{theorem}\label{scottwoj}  Let $P,Q\in {\rm Gr}_\infty(A)$. Then
\begin{equation} e^{2\pi i(\tilde \eta(D_P)-\tilde \eta(D_{Q}))}={\rm det}_{\rm F}(\Phi(P)\Phi(Q)^*).
\end{equation}
\end{theorem}
Note that considering the reals $\mathbb{R}$ as the universal cover of $U(1)$ via
the map $r\mapsto e^{2\pi i r}$, the Theorem can be interpreted as stating that for a smooth path $P_t \in {\rm Gr}(A), \  t\in [0,1]$, the map
\[
s\mapsto \frac{1}{2}\int_0^s \frac{d}{dt}(\eta(D_{P_t})) dt
\]
is the unique lift to $\mathbb{R}$ of the map $[0,1]\to U(1)$
\[
s\mapsto {\rm det}_{\rm F}\left(\Phi(P_s)\Phi(P_0)^*\right).
\]
One can improve Theorem \ref{scottwoj} by choosing a branch of the logarithm. For this, note that  if $P\in {\rm Gr}_\infty(A)$, then from Lemma \ref{kerdp} one knows that $D_P$ is invertible if and only if ${\rm ker} P_X\cap \gamma ({\rm ker} P)=0$ where $P_X$ denotes the Calderon projector, this is by Lemma \ref{grassmannian2} equivalent to $-1\not\in{\rm spec}(\Phi(P)\Phi(P_X)^*)$. On the other hand, since the pair $(P,P_X)$ is a Fredholm pair, by the same Lemma \ref{grassmannian2}, we know that $-1\not\in {\rm spec}_{\rm ess}(\Phi(P)\Phi(P_X)^*)$, so $-1$ is an isolated point in the spectrum of $\Phi(P)\Phi(P_X)^*$. So we can choose a holomorphic branch of the logarithm ${\rm log}:\mathbb{C}\setminus \{0\}\rightarrow \mathbb{C}$ as
\begin{equation}\label{wind12}
 {\rm log}(re^{it})={\rm ln} \ r+i t, \quad r>0, -\pi<t\leq \pi.
\end{equation}
and define the operator ${\rm log}(\Phi(P)\Phi(P_X)^*)$ by holomorphic functional calculus. The so defined ${\rm log}(\Phi(P)\Phi(P_X)^*)$ is then of trace class (see \cite{lesch}) and by construction
\begin{equation}\label{fredholmdet}
{\rm tr \ log}(\Phi(P)\Phi(P_X)^*)\equiv {\rm log\ det}_{\rm F}\left(\Phi(P)\Phi(P_X)^*\right){\rm mod}\ 2\pi i \mathbb{Z}.
\end{equation}
Using this one has the following, for a proof we refer to \cite{lesch}.
\begin{theorem}\label{invertible} Let $X$ and $D$ be as above. Then for
$P\in {\rm Gr}_\infty(A)$ we have
\[
\tilde \eta(D_P)-\tilde \eta(D_{P_X})=\frac{1}{2\pi i} {\rm tr\ log}
(\Phi(P)\Phi(P_X)^*).
\]
\end{theorem}
Now suppose we are given a closed manifold $M$ containing a separating hypersurface $N\subset M$. We consider only Dirac operators $D$ on $M$ so that in a collar neighborhood $[-\epsilon,\epsilon]\times N$ of $N$, $D$ has the form $D=\gamma(\frac{d}{dx}+A)$ as in (\ref{metricol234}). Then Kirk and Lesch \cite{lesch} prove the following, in the next section we will give refined versions for the case of signature operator.
\begin{theorem}\label{invertiblecase}Let $D$ be a Dirac operator on
$M$ and let $N\subset M$ split $M$ into $M^+$ and $M^-$.
Then for $P\in \Gr_\infty(A)$ and $Q\in \Gr_\infty(-A)$ and if $P_{M^+}, P_{M^-}$ denote the Calderon projectors for $M^+$ and $M^-$ respectively, then we have with $\Phi=\Phi_\gamma$,
\[
\begin{split}
\etab(&D,M)-\etab(D_P,M^+) - \etab(D_{Q},M^-)\\
&=-\frac{1}{2\pi i}{\rm tr}\ {\rm log}(\Phi(P)\Phi(P_{M^+})^*)-\frac{1}{2\pi i}{\rm tr}\ {\rm log}(\Phi(P_{M^-})\Phi(Q)^*)\\
&\quad +\frac{1}{2\pi i}{\rm tr}\ {\rm log}(\Phi(I-P_{M^-})\Phi(P_{M^+})^*).
\end{split}
\]
In particular,
 \begin{equation}\label{iceq1}\etab(D,M)=\etab(D_{P_{M^+}},M^+) +
\etab(D_{P_{M^-}},M^-)+\frac{1}{2\pi i}{\rm tr}\ {\rm log}(\Phi(I-P_{M^-})\Phi(P_{M^+})^*).
\end{equation}
Finally, from the last equation we have:
$$\etab(D,M)=\etab(D_{P_{M^+}},M^+) +
\etab(D_{I-P_{M^+}},M^-).$$
\end{theorem}
For later application, we will need a generalization of (\ref{iceq1}) to the case when we glue two manifolds with boundary $M^+, M^-$ with an orientation-preserving isometry $\rho:M^+ \rightarrow M^-$ along their common boundary $N$. The resulting manifold will be denoted $M_\rho$. Assume that $\rho$ is covered by a map $\rho^*:L^2(E)\rightarrow L^2(E)$ that preserves the $L^2$-inner product on sections in $E$ and commutes with $\gamma$. Let $M_{cut}=M^+\sqcup M^-$ so that $\partial M_{cut}=N\sqcup N$ where here, $N$ is oriented by the metric collar on $M^-$ and $M^+$ as left-boundary (reparametrizing the collar of $M^-$ by $N\times [0,\epsilon]$, cf. \cite{lesch}). Then $L^2(E_{|\partial M_{cut}})= L^2(E|_N)\oplus L^2(E|_N)$ and w.r.t. to that splitting the bundle endomorphism $\tilde \gamma: L^2(E_{|\partial M_{cut}})\rightarrow L^2(E_{|\partial M_{cut}})$ decomposes as $\tilde \gamma|\partial M_{cut}=-\gamma|-N\oplus \gamma|N$ (see again \cite{lesch}). Furthermore if $\tilde A$ denotes the tangential operator of $M_{cut}$, we have an isomomphism
\[
{\rm Gr}(\tilde A)\simeq {\rm Gr}(-A)\times {\rm Gr}(-A), \ (P,Q)\mapsto \begin{pmatrix}P&0\\0&Q\end{pmatrix}
\]
Recall the 'continuous transmission' projection acting on $L^2(E_{|\partial M_{cut}})$ in \cite{lesch},
\[
P_\Delta=1/2\begin{pmatrix}I&-I\\-I&I\end{pmatrix}
\]
and denote ${\rm ker}(\Delta)=L_{\Delta}$. Let
\[
\rho_{cut}: N\sqcup N\rightarrow N\sqcup N, \quad \rho_{cut}(x,y)=(x, \rho(y)),
\]
denote the corresponding covering homomorphism by $\rho_{cut}^*:L^2(E_{|\partial M_{cut}})\rightarrow L^2(E_{|\partial M_{cut}})$. We will assume in the following that $\rho_{cut}^*$ commutes wih $\tilde \gamma$. Denote $L_{\Delta_\rho}= \rho_{cut}^*(L_{\Delta})\subset L^2(E_{|\partial M_{cut}})$ and let $P_{\Delta,\rho}\in {\rm Gr}(\tilde A)$ denote the projection so that $L_{\Delta_\rho}={\rm ker}(\Delta_\rho)$ (this is a Fourier-integral operator, see the remark in \cite{lesch} under (5.3) loc. cit.). For any $P \in {\rm Gr}(A)$, let $L_P={\rm ker}(P)$ and denote $L_{P_\rho}=\rho^*(L_P)$. Let $P_\rho$ the associated projection, i.e. $L_{P_\rho}= {\rm ker}(P_\rho)$. We can now formulate
\begin{theorem}\label{generalglue}
Let $D$ be a Dirac operator on $M_\rho=M^+\cup_\rho M^-$ as above. Let $P_{M^+}, P_{M^-}$ denote the Calderon projectors for $M^+$ and $M^-$ respectively, then we have
\begin{equation}\label{iceq2}\etab(D,M_\rho)=\etab(D_{P_{M^+}},M^+) +
\etab(D_{P_{M^-}},M^-)+\frac{1}{2\pi i}{\rm tr}\ {\rm log}(\Phi(I-P_{M^-})\Phi(P_{M^+, \rho})^*).
\end{equation}
\end{theorem}
\begin{proof}
For $\theta \in [0,\pi/4]$ and $P\in Gr(A)$, let $P(\theta, P)$ be the path described in \cite{lesch}, Lemma 5.3 connecting $\Delta=P(\pi/4, P)$ and
\[
P(0,P)=\begin{pmatrix}P&0\\0&I-P\end{pmatrix} \in {\rm Gr}(\tilde A).
\]
Let $L_{\theta,\rho}=\rho_{cut}^*({\rm ker}(P(\theta,P)))$ and let $P(\theta, P,\rho)\in {\rm Gr}(\tilde A), \theta\in [0,\pi/4]$ be the corresponding path of projections. Then (compare (5.2) in \cite{lesch})
\[
\eta(D_{P(\pi/4, P,\rho)},M_{cut})= \eta(D_{P_{\Delta,\rho}},M_{cut})=\eta(D, M_\rho).
\]
On the other hand, by definition of $P(0,P,\rho)$:
\[
\eta(D_{P(0, P,\rho)},M_{cut})= \eta(D_{P},M^+)+\eta(D_{I-P_\rho},M^-).
\]
Thus we have
\begin{equation}\label{blaeta}
\begin{split}
\eta(D, M_\rho)=&\eta(D_{P},M^+)+\eta(D_{I-P_\rho},M^-)+\frac{1}{2}\int_0^s \frac{d}{dt}(\eta(D_{P(\theta,P,\rho)}, M_{cut})) dt\\
&+{\rm SF}(D_{P(\theta,P,\rho)})_{\theta\in[0,\pi/4]}.
\end{split}
\end{equation}
Now by Theorem 5.8 in \cite{lesch}, if $P^+=P^+(L)\in {\rm Gr}(\tilde A)$ is the APS-boundary condition associated to a Lagngagian $L\subset {\rm ker}(\tilde A)$
\[
\frac{d}{d\theta}(\eta(D_{P(\theta,P^+), M_{cut}})=0,
\]
hence by the remark below Theorem \ref{scottwoj}, the function $s\mapsto {\rm det}(\Phi(P(\theta,P^+)\Phi(P(0,P^+))^*)$ is identically zero and hence
\[
\begin{split}
s\mapsto &{\rm det}(\Phi(P(\theta,P^+, \rho)\Phi(P(0,P^+,\rho))^*)={\rm det}(\rho_{cut}^*\Phi(P(\theta,P^+)(\rho_{cut}^*)^{-1}\rho_{cut}^* \Phi(P(0,P^+))^*(\rho_{cut}^*)^{-1})\\
&={\rm det}(\rho^*\Phi(P(\theta,P^+)\Phi(P(0,P^+))^*(\rho^*)^{-1})={\rm det}(\Phi(P(\theta,P^+)\Phi(P(0,P^+))^*)
\end{split}
\]
vanishes identically (here we have used that $\Phi(P(\theta,P^+, \rho)=\rho_{cut}^*\Phi(P(\theta,P^+)(\rho_{cut}^*)^{-1}$ since $\rho_{cut}^*$ commutes with $\tilde \gamma$ and hence preserves its $\pm i$-eigenspaces). Hence the variation of eta-invariant-term in (\ref{blaeta}) vanishes. Then completely analogously to the proof of Theorem 5.9 in \cite{lesch} one deduces from (\ref{blaeta}) using the homotopy invariance of spectral flow that for any $P\in {\rm Gr}(A)$ we have
\[
\begin{split}
\etab(D,M_\rho)=&\etab(D_{P},M^+)+\etab(D_{I-P_\rho},M^-)+\\
&+{\rm SF}(D_{P(\theta,P,\rho)})_{\theta\in[0,\pi/4]}.
\end{split}
\]
Now for $P=P_{M^+}$ we claim that
\[
{\rm SF}(D_{P(\theta,P_{M^+},\rho)})_{\theta\in[0,\pi/4]}=0.
\]
Let $P_{M_{cut}}\in {\rm Gr}(\tilde A)$ be the Calderon projector of $M_{cut}$, then 
\[
{\rm ker}(D_{P(\theta,P_{M^+},\rho)})={\rm ker}\ P(\theta,P_{M^+},\rho)\cap {\rm ker}\ P_{M_{cut}}.
\]
Let $\psi=(\psi_+, \psi_-) \in {\rm ker}\ P(\theta,P_{M^+},\rho)\cap {\rm ker}\ P_{M_{cut}}$ (w.r.t. the splitting of $L^2(E_{|\partial M_{cut}})$ induced by the decomposition of $\partial M_{cut}$. By adopting \cite{lesch} Prop. 5.5 to our situation this is equivalent to
\[
\begin{split}
cos(\theta)P_{M^+}\psi_+&=sin(\theta)P_{M^+, \rho} \psi_-\\
0=sin(\theta)(I-P_{M^+})\psi_+&=cos(\theta)(I-P_{M^+, \rho}) \psi_-.
\end{split}
\]
Since $cos(\theta)\neq 0$, we infer $\psi_- \in {\rm im }\ P_{M^+, \rho}\cap {\rm im }\ P_{M^-}$. On the other hand, if $\psi_- \in {\rm im }\ P_{M^+, \rho}\cap {\rm im }\ P_{M^-}$, then put $\psi_+=tan(\theta)(\rho^*)^{-1})\psi_+$ and $(\psi_+, \psi_-)$ will lie in ${\rm ker}\ P(\theta,P_{M^+},\rho)\cap {\rm ker}\ P_{M_{cut}}$, so the latter is independent of $\theta$. So we arrive at
\[
\etab(D, M_\rho)=\etab(D_{P_{M^+}},M^+)+\etab(D_{I-P_{M^+,\rho}},M^-)
\]
Then applying Theorem \ref{invertible} gives 
\[
\etab(D_{I-P_{M^+,\rho}},M^-)=\etab(D_{P_{M^-}},M^-)+\frac{1}{2\pi i}{\rm tr}\ {\rm log}(\Phi(I-P_{M^+,\rho})\Phi(P_{M^-})^*)
\]
and ${\rm tr}\ {\rm log}(\Phi(I-P_{M^+,\rho})\Phi(P_{M^-})^*)={\rm tr}\ {\rm log}(\Phi(I-P_{M^-})\Phi(P_{M^+, \rho})^*)$ gives the assertion.
\end{proof}

\subsection{Signature operator and adiabatic stretching}\label{begriffe5}
Fix, as above,a $(2n-1)$-dimensional Riemannian manifold $(X,g)$ with nonempty boundary
$\partial X$ and assume that $g$ is of product form near the boundary, the
corresponding collar will be written as $\partial X \times [0,\epsilon)$. The
odd signature operator with trivial coefficients on $X$
\[
D:\oplus_p\Omega^{2p}(X, \mathbb{C}) \rightarrow
\oplus_p\Omega^{2p}(X,\mathbb{C})
\]
is defined to be
\[
D(\beta)=i^{n}(-1)^{p-1}(*d-d*)(\beta),\ {\rm for } \beta \in
\Omega^{2p}(X,\mathbb{C}),
\]
where $*:\Omega^{k}(X,\mathbb{C}) \rightarrow \Omega^{2n-1-k}(X,\mathbb{C})$
denotes the Hodge operator determined by the Riemannian metric on $X$. On the
collar, $D$ takes (modulo unitary conjugation) the form
\begin{equation}\label{product}
D=\gamma(\frac{\partial}{\partial x}+A)
\end{equation}
where the de Rham operator $A:\oplus_k \Omega^{k}(\partial X,\mathbb{C})
\rightarrow \oplus_k \Omega^{2p}(\partial X,\mathbb{C})$ is elliptic and
self-adjoint and given by
\[
A(\beta)= \left \{ \begin{matrix}-(D\hat*+\hat *d)\beta,\quad \beta
\in\oplus_k\Omega^{2k}(\partial X,\mathbb{C}), \\ \quad (D\hat*+\hat
*d)\beta,\quad \beta \in\oplus_k\Omega^{2k+1}(\partial
X,\mathbb{C}).\end{matrix}\right.
\]
Here $\hat *$ denotes the Hodge $*$ operator on $\partial X$ induced by
$g|\partial X$ and $\gamma:\oplus_p \Omega^{p}(\partial X,\mathbb{C})
\rightarrow \oplus_p\Omega^{p}(\partial X,\mathbb{C})$ coincides with $\hat *$
up to a constant:
\[
\gamma (\beta)=\left \{ \begin{matrix}i^{n}(-1)^{p-1}\hat * \beta,\quad \beta
\in\oplus_k\Omega^{2p}(\partial X,\mathbb{C}), \\ \ i^k(-1)^{n-1-q} \hat *\beta,
\quad \beta \in\oplus_k\Omega^{2q+1}(\partial X,\mathbb{C}).\end{matrix}\right.
\]
One has $\gamma^2=-I$, $\gamma A=-A\gamma$ and $\gamma$ is unitary with respect
to the $L^2$-inner product on $\Omega^*(\partial X,\mathbb{C})$ defined by
\[
<\beta_1,\beta_2>=\int_{\partial X}\beta_1\wedge \hat *\beta_2.
\]
Now the Hodge Theorem identifies ${\rm ker} \ A$ with the cohomology of the
complex $(\Omega^*(\partial X,\mathbb{C}), d)$ since the kernel of $A$
corresponds to the harmonic forms so
\[
{\rm ker}\  A={\rm ker }\ d \cap {\rm ker}\ d^* \subset {\rm ker}\ d \rightarrow {\rm
ker}\ d/{\rm image}\ d
\]
is an isomorphsim. Defining $\omega(x,y)=<x,\gamma y>$ one gets (up to a
constant) the intersection form
\[
\omega(\beta_1,\beta_2)=i^r\int_{\partial X}\beta_1\wedge\beta_2
\]
for some integer $r$ depending on the degrees of the $\beta_i$.\\
Then, as we already saw above, since $\partial X$ bounds, the signature of $i\gamma$ restricted to
${\rm ker}\ A$ vanishes, that is, the $i$ and $-i$-Eigenspaces of $\gamma$ acting on
${\rm ker}\ A$ have the same dimension. Recall this  implies that there are Lagrangian subspaces
$L \subset {\rm ker}\ A$ satisfying $\gamma(L)=L^\bot\cap {\rm ker} \ A$. Now let $(H,<\ , \ >,\gamma)$ be a
Hermitian symplectic Hilbert space, that is, a seperable complex Hilbert space
with an isomorphism $\gamma:H \rightarrow H$ with $\gamma^2=-I,
\gamma^*=-\gamma$ and such that the $i$ and $-i$-Eigenspaces have the same
dimension (resp. are both infinite-dimensional). For any subspace $W\subset H$, we define its annihilator
\[
W^0=\{x\in H\ |\ <x,\gamma y>=0 \ {\rm for\ all}\  y \in W\}.
\]
Then, a subspace $W\subset H$ is called isotropic, if $W\subset W^0$, coisotropic, if $W^0\subset W$ and the above Lagrangian-condition translates to $W=W^0$, since that $W^0=JW^\perp$. We the have the following result about symplectic reduction, as taken form Nicolaescu's paper \cite{Nicol2}:
\begin{theorem}\label{sympred2}
Consider $\Lambda\subset H$ Lagrangian and $W\subset H$ an isotropic subspace of $H$, with $W^0$ its annihilator. Then, if $(\Lambda,W^0)$ is a Fredholm pair of subspaces (as defined in \ref{fredholmlag}), then $\mathcal{H}_0=W^0/W$ carries an induced hermitian symplectic structure and
\begin{equation}\label{symplred}
\Lambda^W=(\Lambda\cap W^0)/W
\end{equation}
is a Lagrangian subspace in $\mathcal{H}_0=W^0/W$.
\end{theorem}
The self-adjoint operator $A$ induces a spectral decomposition of $L^2(\Omega^*_{\del X})$ in the following sense, let $E_\mu$ denote the $\mu$--eigenspace of $A$, then set
\begin{equation}\label{decomp43}
\begin{split}
E_\nu^+&=\oplus_{0<\mu \leq \nu} E_\mu,\quad E_\nu^-=\oplus_{-\nu \leq \mu  < 0} E_\mu,\\
F_\nu^+&= \oplus_{\mu > \nu} E_\mu, \quad F_\nu^-= \oplus_{\mu < -\nu} E_\mu,
\end{split}
\end{equation}
This gives at once the orthogonal decomposition
\begin{equation}\label{eigenspace}L^2(\Omega^*_{\del X})=
F^-_\nu\oplus E_\nu^-\oplus \ker A\oplus E_\nu^+\oplus
F^+_\nu.\end{equation}
Now, on the other hand one knows that $L^2(\Omega^*_{\del X})$ decomposes in the sense of Hodge as
\begin{equation}\label{hodge21}L^2(\Omega^*_{\del X})=\im d \oplus
\ker A\oplus \im d^*,
\end{equation}
where here, $d^*=-\hat *d\hat *$ is the $L^2$-adjoint of $d$ considered as a linear operator on $L^2(\Omega^*_{\del X})$, with $\hat *$ as above. Combining these two observations with Theorem \ref{sympred2} leads to the following Lemma (cf. \cite{lesch}) whose proof we include for convenience:
\begin{lemma}\label{symred}
Given the above notation, one has an orthogonal decomposition
\begin{equation}\label{bigsum2}
L^2(\Omega^*_{\del X})=
(F^-_\nu\oplus F^+_\nu)\oplus(d(E^+_\nu)\oplus d^*(E^-_\nu))\oplus
(d^*(E^+_\nu)\oplus d(E^-_\nu))\oplus\ker A.
\end{equation}
where here, the subspaces marked by parentheses as well as $E^-_\nu\oplus
\ker A\oplus E^+_\nu$ are Hermitian symplectic subspaces of $ L^2(\Omega^*_{\del X})$. Furthermore, given a Lagrangian subspace $L\subset   L^2(\Omega^*_{\del X})$ so that $(L, F_0^-)$ form a Fredholm pair of subspaces, then
\begin{equation}\label{symplF}
 R_\nu(L):=\frac{L\cap(F^-_\nu\oplus E^-_\nu\oplus \ker A\oplus
E^+_\nu)}{ L\cap F^-_\nu}\subset E^-_\nu\oplus \ker A\oplus E^+_\nu
\end{equation}
is a Lagrangian  subspace, called the symplectic reduction  of $L$ with respect to the isotropic subspace $F_\nu^-$.
\end{lemma}
\begin{proof}
Observe first that if $S\subset L^2(\Omega^*_{\del X})$ is a closed subspace satisfying $\gamma(S)\perp S$, then $S\oplus \gamma(S)$ is a Hermitian symplectic subspace of $ L^2(\Omega^*_{\del X})$, and $S$ is a Lagrangian subspace of $S\oplus\gamma(S)$. This follows since $\gamma$ preserves $S\oplus \gamma(S)$ and $I\mp i\gamma:S\rightarrow K_{\pm i}$ are isomorphisms, where $K_{\pm i}$ denote the $\pm i$ eigenspaces of $\gamma$ acting on $S\oplus \gamma(S)$. Clearly, then  $S$ is a Lagrangian subspace of $S\oplus\gamma(S)$. Now note that $\hat *$  commutes with $A$, so $d$ and $d^*$ anticommute with $A$. We show that either of the summands in parentheses in (\ref{bigsum2}) is a symplectic subspace. For this, setting in $F^-_\nu\oplus F^+_\nu$  $S=F^-_\nu$ gives the assertion for the first summand in the decomposition \ref{bigsum2}. For $d(E^{\pm}_\nu)\oplus d^*(E^\mp_\nu)$, take $S=d(E^{\pm}_\nu)$, then $\gamma(S)=\hat{*}S =\hat{*}d(E^{\pm}_\nu)=\hat{*}d(\hat{*}E^{\mp}_\nu)=d^*(E^{\mp}_\nu) $. That $\ker A$ is symplectic was discussed above, hence the direct sum $E^-_\nu\oplus \ker A\oplus E^+_\nu$ is symplectic. Finally, the decomposition (\ref{bigsum2}) follows from (\ref{eigenspace}) by applying Hodge decomposition (\ref{hodge21}) to the latter symplectic subspace. Applying Theorem \ref{symplred} with $W^0=F_\nu^-\oplus E_\nu^-\oplus \ker A\oplus E_\nu^+$ and  $W=F_\nu^-$ gives the last assertion, note that since $(L,F_0^-)$ form a Fredholm pair and $W^0/F_0^-$ is finite-dimensional, also $(L,W^0)$ is Fredholm.
\end{proof}
In light of Lemma \ref{symred} and Theorem \ref{sympred2} we can now define the symplectic reduction of $L_X={\rm im }(P_X)$, $P_X$ being the Calderon projector of $D$ on $X$, with respect to the isotropic subspace $F_0^-$:
\begin{equation}\label{vee}
\Lambda_X=R_0(L_X)=\frac{L_X\cap (F_0^-\oplus\ker A)}{L_X\cap  F_0^- }
\subset \ker A.\end{equation}
Then, by the above, $\Lambda_X$ is a Lagrangian subspace of $\ker A$ and by definition of $L_X$ and the unique continuation property of $D$ we have:
\begin{equation}\label{vee1}
V_X=\left\{k\  |\ {\rm there \ is}\  \beta\in\Omega^{{\rm even}}_X \ {\rm s.t.}\  D\beta=0\ {\rm and }\ r(\beta)=f_-+k\in F^-_0\oplus \ker A.\right\}
\end{equation}
This subspace is called the space of {\it limiting values of extended
$L^2$ solutions of $D\beta=0$} in the sense of \cite{Atiyah}, a terminology which will be clear by regarding the first assertion of the following Theorem. For this for $r\ge 0$ define
$$X_r=([-r,0]\times \del X)\cup X$$
and
$$X_\infty= ((-\infty,0]\times \del X)\cup X.$$
We will say $X_r$ ($X_\infty$)  is the extension of $X$ by a collar of length $r$ ($\infty$), note that $D$ extends naturally to $X_{r,\infty}$ since it is induced by a metric having a collar along $\partial X$ (as in (\ref{metricol234}).Note for the following that there exists a $\nu\ge 0$ so that
the Cauchy data space $L_X$ of $D$ is transverse to $F^-_\nu$. This follows from the fact that $(L_X, F_0^-)$ form a Fredholm pair of subspaces so $L_X\cap F_0^-$ is finite-dimensional, so as $\nu$ increases,  $L_X\cap F^-_\nu$ decreases to zero. One calls the smallest such $\nu$ the {\it non--resonance level} for $D$ (calling it $\nu_0$ in the following). We have (cf. \cite{lesch}):
\begin{theorem}\label{adiabaticlimit}
Suppose that $\beta\in \Omega^{\text{\rm \tiny even}}_X$ satisfies $D \beta =0$ and $r(\beta)\in F_0^-\oplus\ker A$. Then one has $d\beta=0$, $d(*\beta)=0$, and $d (r(\beta)) =0$, With $\Lambda_X$ as above, this implies that there is an isomorphism
\begin{equation}\label{limitinglag}
V_X=\im i^*:H^*(X;\C)\rightarrow  H^*(\del X;\C).
\end{equation}
Now let $\nu\geq \nu_0$. Then there exists a subspace $W_X \subset d(E_\nu^+)\subset F^-_0$ being isomorphic to $L_X\cap F_0^-$ resp. to the image of
\begin{equation}\label{bla78}
H^{\text{\rm \tiny even}}(X,\del X;\C)\to H^{\text{\rm \tiny
even}}(X;\C),
\end{equation}
($W_X$ corresponds exactly to the $L^2$ solutions of $D$ on $X_\infty$) so that if $W_X^\perp$ denotes the orthogonal
complement of $W_X$ in $d(E^+_\nu)$, then with respect to the decomposition \ref{bigsum2} of $L^2(\Omega^*_{\del X})$, the 'adiabatic limit' of the
Cauchy data spaces on $X_r$ as $r\rightarrow \infty$ exists and decomposes as a direct sum of Lagrangian subspaces:
\begin{equation}\label{adiablim}
\lim_{r\to\infty}L^r_{X} = F^+_\nu\oplus (W_X\oplus \gamma(W_X^\perp))\oplus d(E^-_\nu) \oplus V_X.
\end{equation}
\end{theorem}
To finish this section we state a glueing result (cf.\cite{lesch}) for the signature operator involving intrinsic and {\it finite}-dimensional data describing the difference of eta-invariants under decomposition. For this we define an analogue of the map $\Phi:\Gr(A)\to
\mathcal{U}(\mathcal{E}_{i},\mathcal{E}_{-i})$ (\ref{phimap}) in the finite--dimensional space $\ker A$. To any Lagrangian subspace
$L\subset \ker A$ we assign the unitary map
\begin{equation}\label{smallphimap}\phi(L):(\mathcal{E}_i\cap \ker
A)\to(\mathcal{E}_{-i}\cap \ker A)\end{equation} by the formula
$$L=\left\{x+\phi(L)(x)  \ | \ x\in (\mathcal{E}_i\cap \ker A)\right\}.$$
\newcommand{\even}{\operatorname{even}}
\begin{theorem}\label{invariantform} Let $D$ be the odd signature operator acting on the split
manifold $M=M^+\cup_N M^-$, then we have
\[
\begin{split}
\etab(D,M)&=\etab(D,M^+;V_{M_+}\oplus F_0^+)+
\etab(D,M^-;F_0^-\oplus V_{M_-})\\
&\quad  +\dim ( V_{M_+}\cap V_{M_-})
-\tfrac{1}{2\pi i}\tr\log(-\phi(V_{M_+})\phi(V_{M_-})^*).
\end{split}\]
\end{theorem}
Note that for the proof it is crucial to note that the dimension of the intersection $  (F_0^-\oplus
\gamma(V))\cap L_{M^+}^r$ is independent of $r\in [0,\infty]$ (analogously for $M_-$) for any Lagrangian $V\subset {\rm ker}\  A$, this will also be of use later (see Chapter \ref{etaquas}). To explain this, note that from the definition of $V_{M_+}$ in (\ref{vee}) there is an exact sequence
$$
0\to L_{M^+}\cap F^-_0\to L_{M^+}\cap (F^-_0+\ker A)\to V_{M_+}\to0.
$$
Then it follows  that for any subspace $V\subset \ker A$ there is an exact sequence
\begin{equation}
\label{exact1}0\to L_{M^+}\cap F^-_0\to L_{M^+}\cap (F^-_0+\gamma(V))\to  V_{M_+}\cap \gamma(V) \to 0.
\end{equation}
Now by Theorem \ref{adiabaticlimit}, $L_{M^+}\cap F^-_0$ is isomorphic to \ref{bla78}, i.e. it is independent of the collar length.
On the other hand, by the same Theorem, $ V_{M_+}$ is isomorphic to the image of $H^*(M^+;\C)\to
H^*(\del M^+;\C)$,  so $V_{M_+}\cap \gamma(V)$ is also independent of the collar-length. Consequently since the middle term in the
exact sequence (\ref{exact1}) is isomorphic to $W_{M_+}\oplus (V_{M_+}\cap\gamma(V))$, its dimension is independent of the collar length $r$.\\
We end this section with a notation which will be used in chapter \ref{etaquas}
 \begin{Def}\label{avmaslov}
For  $(H,\langle\ , \ \rangle, \gamma)$ a finite-dimensional Hermitian symplectic space we define (using the above notation)
$$ m_H:\mathcal{L}(H)\times \mathcal{L}(H) \to \R$$ by
\[
\begin{split}m_H(V,W)&=-\tfrac{1}{\pi
i}\tr\log(-\phi(V)\phi(W)^*)+  \dim(V\cap W)\\
&=-\tfrac{1}{\pi i}
\sum_{\begin{array}{c}\lambda\in\spec(-\phi(V)\phi(W)^*)\\
\lambda\neq -1\end{array}} \log \lambda.
  \end{split}
\]
\end{Def}
Note that, by definition. $m_H(\cdot,\cdot)$ is antisymmetric in its entries, additive on direct sums and, as is shown in \cite{lesch}, continously varying under transformations of the kind $t\mapsto m_H(h_t(\cdot),h_t(\cdot))$, where $h_t:H\rightarrow H$ is continous in $t$.

\section{Eta invariants and quasihomogeneous polynomials}\label{etaquas}
\subsection{General quasihomogeneous polynomials}
We first recall the definition of quasihomogeneous polynomials:
\[
\begin{split}
f \in \mathbb{C}[z_0,\dots,z_n] \quad {\it 'quasihomogeneous'}\  :\Leftrightarrow \  &
\exists \beta_0,\dots,\beta_n,\ \beta \geq 0\in \mathbb{N} \quad{\rm so\ that}\quad\\
&f(z^{\beta_0} z_0,\dots,z^{\beta_n} z_n)=z^\beta f(z_0,\dots,z_n),
\end{split}
\]
where $\{\beta_i/\beta\}_{i=0}^n \in \mathbb{Q}$ are called 'the weights' of $f$.
\begin{Def}\label{defmilnor323}
Let $f\in \mathbb{C}[z_0,\dots,z_n]$ be quasihomogeneous with isolated singularity $0 \in f^{-1}(0)$.
Then the locally trivial fibration
\begin{equation}\label{milnor323}
f:  Y=\bigcup_{u \in \delta S^1} X_u:=f^{-1}(u) \cap B^{2n+2}\rightarrow \delta S^1=:S^1_\delta,
\end{equation}
for $\delta >0$ sufficiently small is called the Milnor bundle of $f$ with Milnor fibres $X_u$.
\end{Def}
{\it Remark.} The Milnor fibres are homotopy equivalent to a 'wedge' of $n$-spheres, $X_u \simeq \bigvee_\mu S^n$, so their (reduced) cohomology is concentrated in $H^n(X_u,\mathbb{C})$.\\
The 'boundary fibration' $\partial Y= \bigcup_{u \in \delta S^1_\delta} \partial X_u \rightarrow \delta S^1$ extends smoothly to the unpunctured disk $D$ hence is trivial, more exactly:
\begin{lemma}
The smooth family of manifolds $\partial X=\{X_u\cap S^{2n+1}\}$, $u \in S^1_\delta$ admits a trivialization which is unique up to homotopy.
\end{lemma}
Fix one fibre $X_u,  u\in S^1_\delta$ and denote a diffeomorphism as indicated by the lemma by
\begin{equation}\label{triv}
\Theta:\partial Y \rightarrow -\partial X_u \times S^1, \ \ {\rm s.t.}\ \Theta(\partial X_w) = -\partial X_u \times \{w\}, \ w \in S^1_\delta.
\end{equation}
Set $Y_0=-X_u\times S^1$, where the '$-$'-sign indicates to take the orientation of the fibre of $(Y_0,\partial Y_0)$ is opposite to that of $(X_u, \partial X_u)$, we will think of $\Theta$ as an (orientation-reversing) diffeomorphism of some neighborhood $U$ of $\partial Y$ in $Y$ onto a neighborhood $V$ of $\partial \tilde Y_0$ in $Y_0$, so $\Theta:U \subset Y \rightarrow V \subset Y_0$ is a diffeomorphism. Now $Tf= \{{\rm ker } f_*: TY \rightarrow TS^1_\delta\}$ carries a canonical metric $g^{Tf}$ induced by $\mathbb{C}^{n+1}$,
on the other hand consider the 'Euler vector field' on $\mathbb{C}^{n+1}$,
\[
X_f(z)=\sum_{i=0}^n 2\pi i w_i z_i\frac{\partial}{\partial z_i}, z \in Y \quad {\rm it\ satisfies}\quad (X_f.f)(z)=2\pi i f(z),
\]
so $X_f$ defines a horizontal distribution $H_{X_f} \subset TY$. Using this one can define a metric $g^Y$ on $Y$ as
\begin{equation}\label{metric}
g^Y=g^{Tf} \oplus f^*g^{S^1_\delta} \quad {\rm s.t.\  the \ splitting}\quad TY=Tf \oplus H_{X_f} \quad{\rm is \ orthogonal},
\end{equation}
where as above $g^{Tf}=g^{\mathbb{C}^{n+1}}|_{Tf}$. $g^{S^1}$ is the standard metric on $S^1$. With these definitions, $f:Y\rightarrow S^1_\delta$ becomes a Riemannian submersion. Note also that $L_{X_f}g^{Tf}=0$, i.e. the fibres of $(Y, g^Y)$ are totally geodesic by \cite{vilms}. In what follows, the next lemma will be essential.
\begin{lemma}\label{product123}
For any small $\delta>0$ as in (\ref{milnor323}) there is a Riemannian metric $\tilde g^Y$ on $Y$ and open neighbourhoods $\mathcal{U}\subset \mathcal{W}\subset Y$ of the boundary $\partial Y$, so that the geodesic distance ${\rm dist}(\partial Y, Y\setminus \mathcal{W})$ becomes arbitrarily small as $\delta\rightarrow 0$ and $\tilde g^Y=g^Y$ on $Y\setminus \mathcal{W}$. Further there is a horizontal vector field $\hat X_f$ on $Y$ that coincides with $X_f$ on $Y\setminus \mathcal{W}$ and so that setting $H_{\hat X_f}={\rm span}(\tilde X_f)$ we have that $\tilde g^Y|\mathcal{W}$ is a submersion metric as in (\ref{metric}) w.r.t. $H_{\hat X_f}$ and so that $L_{\hat X_f}\tilde g^Y=0$. Further there is an isometry $\Theta:\mathcal{U} \rightarrow \mathcal{U}_{0}:=(\mathcal{U}\cap X_u) \times S^1$ w.r.t. to $(\mathcal{U}, \tilde g^Y|\mathcal{U})$ and $(\mathcal{U}_{0}, (\tilde g^{Y}|X_u\oplus g^{S^1})|_{\mathcal{U}_0})$ for some fixed fibre $X_u$. Finally, the functions $(g^Y_{ij}-\tilde g^Y_{ij})(\delta):Y\rightarrow \mathbb{C}$, $i,j \in \{1,\dots, n\}$ and  all their derivatives converge locally uniformly to zero on $Y$ for $\delta\rightarrow 0$.
\end{lemma}
\begin{proof}
Fix a cutoff function $\psi_m:[0,1]\rightarrow [0,1]$ with $\psi_m(t^2)=1$ for $t \leq 1-2/m$ and $\psi(t^2)=0$ for $t \geq 1-1/m$ for some $m>2, m \in \mathbb{N}$. Set $\tilde X_z=\{x \in B_1^{2n+2}|f(x) = \psi_m(|x|^2)z\}$ for $z \in D_\delta$ and choose $m>2$ so that $X_z$ is smooth for all $z \in D_\delta$. Set for $0<\delta$ small enough
\begin{equation}\label{milnorfibr5}
\tilde f: \tilde Y\ = \ \bigcup_{ z \in \partial D_\delta=S^1_\delta}\ \tilde X_z \times \{z\} \subset \mathbb{C}^{n+2}\longrightarrow S^1_\delta.
\end{equation}
let $\pi_1:\mathbb{C}^{n}\times \mathbb{C}\rightarrow \mathbb{C}^n$ be the obvious projection and $\tilde X_f=(X_f, 1) \in \Gamma(T\tilde Y)$ so that $X_f=(\pi_1)_*(\tilde X_f)$. For $z \in S^1_\delta$, set
\[
B_z(m)= \{x\in X_x|\psi_{m}(|x|^2)=0\}\subset \tilde X_z, \ \mathcal{U}^0 =\bigcup_{z\in S^1_\delta} B_z(m).
\]
Let $\tilde Y\setminus\mathcal{W}^0 =\bigcup_{z\in S^1_\delta} \tilde A_z(m)$ where $\tilde A_z(m)= \{x\in X_x|\psi_{m}(|x|^2)=1\}\subset X_z$. It follows from Lemma 4.1.2 in \cite{klein1} that $\delta, m$ can be chosen s.t.  ${\rm dist}(\partial Y, Y\setminus \mathcal{W})\rightarrow 0$ as $\delta\rightarrow 0$. Let $g^{T\tilde f}$ be the metric on $T\tilde f={\rm ker} \tilde f_*$ induced by restriction of the Euclidean metric $g_{\mathbb{C}^{n+2}}$ to $T\tilde f$. Then define a metric $g^{\tilde Y}$ on $\tilde Y$ by
\begin{equation}
g^{\tilde Y}=g^{T\tilde f} \oplus \tilde f^*g^{S^1_\delta} \quad {\rm s.t.\  the \ splitting}\quad T\tilde Y=T\tilde f \oplus H_{\tilde X_f} \quad{\rm is \ orthogonal}.
\end{equation}
By Lemma 4.1.2 of \cite{klein1}, there is a family of diffeomorphisms $\Psi_z:X_z\rightarrow \tilde X_z, \ z \in S^1_\delta$ which assemble to a diffeomorphism $\Psi:X\rightarrow \tilde X$. Set $\tilde g^Y=\Psi^*(g^{\tilde Y})$, $\mathcal{U}=\Psi^{-1}(\mathcal{U}^0),\  \mathcal{W}=\Psi^{-1}(\mathcal{W}^0)$ and $\hat X_f=\Psi^{-1}_*(\tilde X_f)$. Over $\mathcal{U}^0$, set $H_{\tilde X_f}(s)={\rm span}(sX_f, 1) \subset T \mathcal{U}^0, \ s \in [0,1]$ and define a locally trivial fibration $\tilde p: \tilde E:=[0,1]\times \mathcal{U}^0\rightarrow [0,1]\times S^1_\delta, \ \tilde p(s,x)= (s, \tilde f(x))$ with a family of horizontal subspaces $\tilde H\subset TE$ by $\tilde H(s, x)=H_r(s)\oplus H_{\tilde X_f}(x,s), \ s \in [0,1], \ x \in \mathcal{U}^0$, where $H_r(s)={\rm span}(\partial/\partial s)$. It is then clear that $\tilde H$ factorizes to a well-defined horizontal distribution $H_E$ on $p:E=\tilde E/(0,z,x)\sim (0,z',x) \rightarrow D_\delta$ where $(0,z,x)\sim (0,z',x), z, z'\in S^1_\delta$ means identifying the fibres of $\{0\}\times \mathcal{U}^0$ by the identity. Then by construction parallel transport along lifts of any curves in $D_\delta$ w.r.t. $H_E$ are fibrewise isometries. Glueing a copy of the bundle $E$ with reversed orientation along the boundary of $D_\delta$, one gets a complete fibration $\hat p:\hat E\rightarrow S^2$ which is a totally geodesic simply connected Riemannian submersion in the sense of Vilms (\cite{vilms}). By \cite{vilms} (Corollary 3.7) it follows that $\hat E$ is a Riemannian product with projection $\hat p$, in particular $\hat p|p^{-1}(S^1_\delta)=\tilde f|\mathcal{U}^0: \mathcal{U}^0\rightarrow S^1_\delta$ itself is a Riemannian product by restriction, denoting the corresponding isometry by $\tilde \Theta$ and composing $\Theta=\tilde \Theta\circ \Psi$ we arrive at the assertion.\end{proof}
Note that in the following we will only consider the metric $\tilde g^Y$ on $Y$ and denote it by $g^Y$. Using the lemma, we can furthermore assume $g^Y$ to be perturbed in a neighbourhood of the boundary of $Y$ so that it has a metric collar thereon that is compatible with (\ref{triv}) in the sense that the metric collar is induced by a collar on the fixed fibre $X_u$ by $\Theta^{-1}$. To be precise, we use the isometry $\Theta^{-1}$ to glue a metric collar to $Y$ which is given explicitly by (note that $Y$ carries a 'left-boundary' as in (\ref{metricol234}), while $Y_0$ carries the opposite collar $(-\epsilon,0]\times \partial X_u$)
\begin{equation}\label{product999}
\left([0,\epsilon)\times \partial X_u\times S^1, g^{[0,\epsilon)} \oplus g^{\partial X_u}\oplus f^*g^{S^1}\right)
\end{equation}
where $g^{[0,\epsilon)}$ denotes the standard metric on $ [0,\epsilon)$ for some fixed fibre $F:=X_u\subset Y$, we will denote the resulting manifold (resp. metric) again by $Y$ (resp. $g^Y$). Note that this also means that the above 'Euler vector field' $\hat X_F$ extends to a horizontal distribution on $Y\cup_\Theta Y_0$ which coincides over $Y_0$ with $TS_\delta^1\times \{0\}\subset TY_0$, i.e. it induces is a fibrewise isometry on $Y\cup_\Theta Y_0$, this will be useful later.\\
Let now be $D$ the signature operator on $Y$ with respect to $g^Y$, let $A$ be its tangential operator on $\partial Y$. Fix monomials
\[
\{z^\alpha : \alpha \in \Lambda \subset
\mathbb{N}^{n+1}, |\Lambda|= \mu\}\ {\rm s. t.}\ \{[z^\alpha]\}_\alpha \ {\rm spans}\
\mathcal{O}_{\mathbb{C}^{n+1}}/\rm{grad}(f)\mathcal{O}_{\mathbb{C}^{n+1}}=:M(f)
\]
where dim $M(f)=\mu$, write $l(\alpha)=\sum_{i=0}^n(\alpha_i+1)w_i$.
We will proof the following, let $\Lambda_Y\subset {\rm ker} A$ be the space of limiting values of extended $L^2$-solutions of $D\beta=0$ on $Y$, $P_{Y_0}$ be the Calderon-Projector of the trivial bundle $Y_0=-F \times S^1$ with respect to product metric, $P_{Y}^\infty$ the adiabatic limit (see above) of the Calderon projector $P_Y$ on $Y$. In the following, we set ${\rm arg}(re^{i\theta})=\theta \in[0,2\pi), r>0$.
\begin{theorem}\label{eta}
Let $D_{I-P_{Y_0}}, D_{P^+(\Lambda_Y)}$ be defined as in section \ref{begriffe1}. Then the eta-invariant $\eta(D_{I-P_{Y_0}})$ for $n=2k$ or $n=2k+1$, $k \in \mathbb{N}$, equals:
\[
\eta(D_{I-P_{Y_0}})  =\sum_{\alpha \in \Lambda, l(\alpha) \notin \mathbb{Z}}(-1)^{[l(\alpha)]+n}\left(1-2\{l(\alpha)\}\right) +\frac{{\rm arg}(-1+\frac{4}{3}i)}{\pi}\sum _{\alpha \in \Lambda, l(\alpha)\in \mathbb{Z}}(-1)^{l(\alpha)+n},
\]
where the first sum represents the contribution of the non-trivial part of the algebraic monodromy $\rho\in {\rm Aut}(H_n(F,\mathbb {C}))$ of $Y$, the second sum amounts to a summation over a basis for the $\lambda=1$-eigenspace of $\rho$. Note that $\{\cdot\}$ denotes the fractional part, $[\cdot]$ is the integer part function. Furthermore the above eta-invariant equals the eta-invariants given by $APS$-boundary conditions induced by $\Lambda_Y$ and relative to the adiabatic limit of the Calderon-projector on $Y$:
\[
\eta(D_{I-P_{Y_0}}) = \eta(D_{P^+(\Lambda_Y)})  = \eta(D_{P_{Y}^{\infty}})
\]
where both equalities are in fact valid in $\mathbb{R}$.
\end{theorem}
{\it Remarks. }It is conjectured that the above eta-invariant actually equals or at least contributes to the adiabatic limit of the eta-invariant of the signature operator relative to the metric induced by restriction of the Euclidean metric of $\mathbb{C}^{n+1}$. Note that the first formula is valid in $\mathbb{R}$ for these specific boundary conditions, general Lagrangians $\Lambda$ produce additional terms (see Lemma \ref{arb} below). Note also that if $b_\lambda, \lambda \in {\rm spec}( \rho)$ is the Eigenvalue decomposition of the (degenerate) homological intersection form $b$ on the Milnor fibre $X_u=F$ with respect to the unitary map $\rho\in {\rm Aut}(H_n(F,\mathbb {C}))$, we have by (\ref{weightvar}) and (\ref{nemweight}) below:
\begin{equation}\label{bdecomp}
\eta(D_{P^+(\Lambda_Y)})= \sum_{\lambda \in {\rm spec}(\rho), \lambda
\neq 1} {\rm\ sign}(b_\lambda)(1-2c)+\frac{{\rm arg}(-1+\frac{4}{3}i)}{\pi}\sum _{\alpha \in \Lambda, l(\alpha)\in \mathbb{Z}}(-1)^{l(\alpha)+n},
\end{equation}
when $\lambda = e^{2\pi i c},\ c \in (0,1]$, the second sum amounts to a summation
over a basis for the $\lambda=1$-eigenspace of $\rho$.\\
We can also give the eta-invariant on the Milnor bundle relative to {\it arbitrary} $APS$-boundary conditions, so that all expressions involve finite-dimensional terms:
\begin{theorem}\label{eta2}
Set $P^+(\Lambda) \in Gr_\infty(A)$ s.t. ${\rm im}(P^+(\Lambda))=F_0^+\oplus \Lambda$ for $\Lambda \subset{\rm ker\ A}$ arbitrary Lagrangian and $n=2k$ or $n=2k+1$, then:
\[
\begin{split}
\tilde \eta(D_{P^+(\Lambda)})= &\frac{1}{2}\eta(D_{I-P_{Y_0}})+\frac{1}{2\pi i}{\rm tr \log}(\phi(\Lambda)\phi(\Lambda_Y)^*)\\
&+\frac{1}{2}{\rm dim}\ (im(H^{even}(Y,\partial Y,\mathbb{C})\rightarrow H^{even}(Y,\mathbb{C})),
\end{split}
\]
here, $\tilde \eta(D_P)=(\eta(D_P) + {\rm dim\ ker }(D_{P}))/2$ for $P\in Gr(A)$ denotes the reduced eta-invariant and the image of $\phi:Gr({\rm ker}(A)) \rightarrow \mathcal{U}({\rm ker}(A)\cap \mathcal{E}_i,{\rm ker}(A)\cap \mathcal{E}_{-i})$ is a finite dimensional unitary mapping.
\end{theorem}
The proof of the Theorems will be devided in a series of Lemmata. Fixing any fibre $X_u=:F$ we have the above described diffeomorphism
\[
\Theta:\partial Y  \rightarrow -\partial F \times S^1_\delta,
\]
furthermore the metric splits orthogonally $g^Y|_{\partial Y}= g^{\partial Y}\oplus f^*g^{S^1}$. As mentioned above, we define the metric collar of $Y$ so that it extends this trivialization to an orthogonal splitting on a collar of $\partial Y$.
\[
g^Y_{[0,\epsilon)\times \partial Y}=(g^{[0,\epsilon)}\oplus g^{\partial Y})\oplus f^*g^{S^1_\delta} .
\]
Now define the {\it double} $Y^e$ of $Y$ by
\begin{equation}\label{double}
Y^e=Y \cup_\Theta (- F \times S^1_\delta), \quad  \pi^e: Y^e
\rightarrow S^1_\delta, \quad (\pi^e)^{-1}(u)=F\cup (-F)=:F^e,
\end{equation}
and set $g^{-F\times S^1_\delta}=g^F\oplus f^*g^{S^1_\delta}$, this fits together with the above to define a smooth metric  $g^{Y^e}$ on $Y^e$. The flow of $\hat X_f$ on $Y\setminus \mathcal{W}$ is given by
\[
\Phi_f(t,z_0, \dots, z_n)=(e^{2\pi i t w_0},\dots, e^{2\pi i tw_n})
\]
hence $(\Phi_f(t))^*g^{Tf}=g^{Tf}$ for all $t \in [0,1]$. For the following Lemma we assume that $Y$ resp. $H_{X_f}$ is modified in a neighbourhood of $Y$ in the sense of the discussion below Lemma \ref{product123}, i.e. the trivial horizontal distribution on $S^1 \times F$ and $H_{\hat X_f}$ glue smoothly in (\ref{double}). Then
\begin{lemma}\label{isometry}
Let $X_c(u)=c'(u), u \in S^1_\delta$ and let $X^e_f$ be its horizontal lift to $H_{X_f}$ in $Y^e$, $\Phi^e_f(t,z), \ t \in [0,1], \ z \in \tilde Y^e$ be its
flow. Then $Y^e$ is isometric to  $[0,1] \times Y_e^{u}/(0,z) \sim (1,\Phi_f^e(1,z))$, where the latter is furnished with the factorized product metric on $[0,1]\times Y^e$.
\end{lemma}
\begin{proof}
That $\Phi^e_f(t,\cdot)$ is a fibrewise isometry on $Y\subset Y^e$ for the restricted metric follows directly from its explicit form, note that $H_{\hat X_f}$ extends naturally to the collar as constructed above and the associated flow is a fibrewise isometry thereon by definition of the metric as in (\ref{product999}). By \cite{vilms} it follows that $Y^e$ is a totally geodesic fibration, hence again by \cite{vilms}, it is locally over the base a metric product. Explicitly, note that we have a homomorphism $\kappa: \pi_1(S^1) \simeq \mathbb{Z}\rightarrow {\rm Diff}(\tilde Y^e_u)$ so that $\kappa(1)=\Phi^e_f(1,\cdot)$. This defines a free $\mathbb{Z}$-left action on $\mathbb{R}\times Y^e_u$ by $\tilde \kappa (z,t,x)=(t+z,\kappa(z)x)$, if $z \in \mathbb{Z}, t \in \mathbb{R}$, $x \in Y^e$. Consider now the universal covering $\tilde p:\tilde Y^e\rightarrow Y^e$ with the induced metric. Then $\Phi^e_f(1,\cdot):Y^e\rightarrow  Y^e$ lifts to $\Phi^e_f(1,\cdot):\tilde Y^e\rightarrow \tilde Y^e$. Thus $\hat \kappa(z, \tilde x)=(\Phi^e_f)(z,\tilde x)$ gives a free, properly discontinous and isometric $\mathbb{Z}$-action on $\tilde Y^e$ and the two actions $\tilde \kappa, \hat \kappa $ are equivariant w.r.t. the diffeomorphism
\[
\Psi: \tilde Y^e \rightarrow \mathbb{R}\ \times Y^e_u,\quad
\Psi(\tilde x)=((\Phi^e_f)(t+z,x) ,t),
\]
where here, $\tilde x=(\pi_1)_*(z)(\hat x)$, $\pi(\tilde x)=x$, $\pi^e\circ \tilde p(\tilde x)=\delta e^{2\pi i t}u$ and $\hat x \in  \tilde Y^e_u$ where $\tilde Y^e_u\subset \tilde Y^e$ is a fixed lift of $x \in Y^e_u$. Note that $\Psi$ is actually an isometry if its image is equipped with the product metric $dt^2+g_{Y^e_u}$. Then set $Z:=[0,1] \times Y^e_{u}/(0,z) \sim (1,\Phi_f^e(1,z)):= \mathbb{R}\times Y^e_u/\mathbb{Z}$, that is $Z$ is the associated bundle to the principal bundle $\mathbb{R} \rightarrow \mathbb{R}/\mathbb{Z}$ with fibre $\tilde Y^e_u$. It carries the factorized product metric. Since $\Psi$ is an isometry and equivariant, it factorizes to an isometry $\hat \Psi: \tilde Y^e/\mathbb {Z}\rightarrow Z$ and $\tilde Y^e/\mathbb {Z}$ is of course isometric to $Y^e$ with the metric $g^{Y^e}$.
\end{proof}
Let $\Lambda_{Y_0}$ be the space of limiting values of extended $L^2$-solutions of $D\beta=0$ as introduced in (\ref{vee}) resp. (\ref{vee1}) , where $D$ is the signature operator on $Y_0:=-F \times S^1$ w.r.t the product metric, let $P_{Y_0}$ be the Calderon Projector of the signature operator on $Y_0$, $L_{Y_{0}}$ its image.
\begin{lemma}\label{symm}
With the above notations we have
\[
\eta(D_{L_{Y_0}}, Y_0)=0,
\]
on the other hand if $P_{Y_0}^r$ is the Calderon projector of $Y_0$ with an attached collar of length $r$, $L^r_{Y_0}$ its image, $P_{Y_0}^\infty$ resp. $L^\infty_{Y_0}$ their welldefined adiabatic limits (cf. Theorem \ref{adiabaticlimit}), then
\[
\eta(D_{F_0^-\oplus \Lambda_{Y_0}}, Y_0)=\frac{1}{\pi i}{\rm tr \ log}(\Phi(P_{\Lambda_{Y_0}\oplus F_0^-})\Phi(P_{Y_0}^\infty)^*) - {\rm dim \ ker \ }D_{\Lambda_{Y_0}\oplus F_0^-}.
\]
\end{lemma}
{\it Remark. } Here we used the notation of  \cite{lesch} (resp. the discussion above (\ref{fredholmdet})) to refer to define ${\rm log}$ using functional calculus, using the convention
\[
{\rm log} (re^{it})=ln \ r+it,\  r>0, -\pi < t \leq \pi,
\]
for the logarithm.
\begin{proof}
For the first assertion we are done if  for $E=\oplus_p\Omega^{2p}((-F)\times S^1=Y_0,\mathbb{C})$
we can find an isometry
\[
T:L^2(S^1,L^2(E|_{F}))\rightarrow L^2(S^1,L^2(E|_{F})),
\]
which anticommutes with the signature operator $D$ on $H^1(E)$ and maps the
domain of $\tilde D_{\Lambda_{Y_0}}$ isomorphically to itself. Now
\begin{equation}\label{prod}
D=\tilde\gamma(\partial/\partial \theta+\tilde A),\quad\tilde \gamma^2=-1, \ \tilde \gamma \tilde A=-\tilde A \tilde \gamma,
\end{equation}
on $C^\infty_0(E)\simeq C^\infty(S^1,C_0^\infty(E|_{F}))$ for some first order self-adjoint operator operator $\tilde A$ on $C^\infty_0(E|_{F})$ and an endomorphism $\tilde \gamma:E|_{F}\rightarrow E|_{F}$. Now define the  isometry $T:C^\infty(S^1,C_0^\infty(E|_{F}))\rightarrow C^\infty(S^1,C_0^\infty(E|_{F}))$
\[
(Tf)(\theta)=\tilde \gamma f(-\theta), \ \theta \in S^1.
\]
It is a straightforward calculation that $DT=-TD$ on $C^\infty_0(E)$. Now {\it assume} that the bounded extension of $T$ to $L^2(S^1,L^2(E|_{F}))$ preserves the domain $\mathcal{D}_{L_{Y_0}}$ of $D_{L_{Y_0}}$. Then since  $D_{L_{Y_0}}$ is the closure of the operator
\[
D_{L_{Y_0},0}=D|\{\phi \in C_0^{\infty}(E)|\phi \in C_0^\infty(E), P_{Y_0}(\phi|_{\partial Y_0})=0\}
\]
and we have $T^{-1}D_{L_{Y_0},0}T=-D_{L_{Y_0},0}$ on their common domain and since $D\subset D^*$ because of formal self-adjointness of $D$ one has
\[
-D_{L_{Y_0}}=-D_{L_{Y_0},0}^{**}=T^{-1}D_{L_{Y_0},0}^{**}T=T^{-1}D_{L_{Y_0}}T,
\]
where we also used that $T$ is bounded and unitary. So we still have to show that $T$ maps
\[
\mathcal{D}_{L_{Y_0}}=\{\phi \in L^2(E)|\phi \in H^1(E), P_{Y_0}(\phi|_{\partial Y_0})=0\}\subset L^2(E)
\]
onto itself. Now
\[
{\rm im}(P_{Y_0})=r(K:={\rm ker}\ D:H_{1/2}(E)\rightarrow H_{-1/2}(E)) \subset L^2(E|{\partial Y_0}),
\]
where $r$ is the restriction to the boundary, due to unique continuation, this
is injective. Let $\phi \in \mathcal{D}_{L_{Y_0}}$. Then  $\phi \notin K$, assume now $T\phi
\in K$, then $DT\phi=-TD\phi=0$, since $T$ is injective $\phi$ would be in $K$,
so $T\phi \notin K$, since $T$ preserves $H^1(E)$, $T\phi \in \mathcal{D}_{P_{Y_0}}$,
since $T^2=-id$ it is thus surjective.\\
For the second assertion, note that following Kirk and Lesch (\cite{lesch}), Theorem 4.2 (see Theorem \ref{invertible} in Chapter \ref{chapter1}) the difference of reduced eta-invariants of a Dirac-Operator $D$ relative to an arbitrary Lagrangian projections $P_1 \in Gr_{\infty}(A)$ and the Calderon projection $P_X \in Gr_{\infty}(A)$ on a manifold $X$ is given by
\[
\tilde \eta(D_{P_1},X)-\tilde \eta(D_{P_X},X)=\frac{1}{2\pi i}{\rm tr \log}(\Phi(P_1)\Phi(P_X)^*),
\]
where the 'reduced' eta-invariant is defined as $\tilde \eta(D_P)=(\eta(D_P)+{\rm dim \ ker}D_P)/2$ for any projection $P \in Gr_{\infty}$. Now set $X=Y_0$ and substitute for $P_1$ the projection onto $\Lambda_{Y_0}\oplus F_0^-$, where $\Lambda_{Y_0}$ as above. Stretching the collar of $Y_0$ to the length $r$ and noting that following the first part of this lemma, whose argumentation remains valid for any collar length $r$, $\tilde \eta(D_{P^r_{Y_0}},Y^r_0)=0$ for all $r>0$, one has
\[
\tilde \eta(D_{F_0^-\oplus \Lambda_{Y_0^r}}, Y^r_0)=\frac{1}{2\pi i}{\rm tr \ log}(\Phi(P_{\Lambda_{Y^r_0}\oplus F_0^-})\Phi(P_{Y_0}^r)^*),\ r >0.
\]
Now clearly $\Lambda_{Y_0^r}= \Lambda_{Y_0}, r >0$ since the space of limiting values of extended $L^2$-solutions to $D\beta=0$ on $Y_0$ is independent of its collar length. Furthermore we assert that the reduced eta invariant $\tilde \eta(D_{F_0^-\oplus \Lambda_{Y_0}}, Y^r_0)$ is independent of $r>0$. To see this, consider Theorem \ref{invariantform}, applied to the case of the closed 'asymmetric' double $M^r=Y_0^r \cup_ -Y_0$, where the glueing is along the boundaries, which asserts in this case
\[
\tilde \eta(D,M^r)-\tilde \eta(D_{F_0^+\oplus \gamma (\Lambda_{Y_0})}, -Y_0)=\tilde \eta(D_{F_0^-\oplus \Lambda_{Y_0^r}}, Y^r_0).
\]
Since $M^r=Y_0^{r/2}\cup -Y_0^{r/2}$, there is an orientation reversing isometry on $M^r$, so that $\eta(D,M^r)$ is zero, the kernel of $D$ on $M^r$ is a topological invariant and hence independent of $r$ so the left hand side of the above equation is in fact independent of $r$ which gives the assertion. So we conclude
\[
\eta(D_{F_0^-\oplus \Lambda_{Y_0}}, Y_0)=\frac{1}{\pi i}{\rm tr \ log}(\Phi(P_{\Lambda_{Y_0}\oplus F_0^-})\Phi(P_{Y_0}^r)^*) - {\rm dim \ ker \ }D_{\Lambda_{Y_0}\oplus F_0^-},\ r >0,
\]
where $D$ is considered to be defined on $Y_0$. Since the path of projections $r \mapsto P_{Y_0}^r$ is continous for $r \in [0,\infty]$ (Theorem \ref{adiabaticlimit}) and has a well-defined limit, we arrive at the assertion of the lemma.
\end{proof}
Now one derives from Bruening/Lesch (\cite{brle}) resp. Theorem \ref{invertiblecase}, if $P_0$ is the Calderon Projector of the signature operator on the trivial bundle $S^1 \times (-\tilde F)$ with respect to the product metric the equality
\begin{equation}\label{bla}
\eta(D_{\gamma(L_{Y_0})},Y)=\eta(D, Y^e) +\eta(D_{L_{Y_0}}, Y_0),
\end{equation}
but the latter is zero by the lemma. On the other hand (see Definition \ref{avmaslov} explaining the function $m_{H^*(\partial Y,\mathbb{C})}(\cdot,\cdot)$):
\begin{lemma}\label{limiting}
The spaces of limiting values of extended $L^2$-solutions of the signature operators on $Y$ resp. $Y_0$ are equal, that is $\Lambda_Y=\Lambda_{Y_0}$, i.e. $m_{H^*(\partial Y,\mathbb{C})}(\Lambda_Y,\Lambda_{Y_0})=0$.
\end{lemma}
\begin{proof}
$\Lambda_Y$ resp. $\Lambda_{Y_0}$ are isomorphic to (harmonic representatives of) ${\rm im}H^*(Y;\mathbb{C})\rightarrow H^*(\partial Y,\mathbb{C})$ resp. ${\rm im}H^*(Y_0;\mathbb{C})\rightarrow H^*(\partial Y,\mathbb{C})$. Considering  the cohomological Wang exact sequence ($h$ being the monodromy on $Y$, acting on a given fibre $F=Y_u$)
\[
0 \rightarrow H^n(Y, \mathbb{C}) \xrightarrow {{\rm restr}}H^n (F,\mathbb{C})\xrightarrow {\rm h -id}H^n(F,\mathbb{C}) \rightarrow H^{n+1}(Y,\mathbb{C}),
\]
hence the elements of $H^n(Y;\mathbb{C})$ restrict fibrewise to exactly those classes in $H^n(F,\mathbb{C})$ on which the monodromy acts trivially. Invoking Corollary 5.3 in \cite{klein3}, we see that $H^n(Y,\mathbb{C})$ can be represented by forms $\alpha\otimes \gamma$, where, using the notation in loc. cit., $\alpha  \in \Gamma_0(\mathcal{K}_F)$ and additionally $\alpha|Y_u\in {\rm ker(h-id)}$ and $\gamma\in H^*(S^1,\mathbb{C})$. On the other hand, by comparing with the long exact sequence of $(F,\partial F)$, the classes ${\rm ker(h-id)}$ restrict to the elements of $H^n(\partial F,\mathbb{C})$. Consequently, since (the space of harmonic representatives of) $H^{0,n}(\partial F,\mathbb{C})\otimes H^*(S^1,\mathbb{C})$ is Lagrangian in $H^*(\partial Y,\mathbb{C})$ and is contained in the image of the restriction of $H^*(Y,\mathbb{C})$, it equals this image. On the other hand the elements of the restriction of $H^*(Y_0,\mathbb{C})$ to the boundary are clearly given by $H^{0,n}(\partial F,\mathbb{C})\otimes H^*(S^1,\mathbb{C})$, so both restrictions are equal. Now using the antisymmetry of $m_{H^*(\partial Y,\mathbb{C})}(\cdot,\cdot)$ (see the remark below Definition.\ref{avmaslov}) the assertion follows.
\end{proof}
Now by Theorem \ref{invariantform} resp. Definition \ref{avmaslov} (see also \cite{lesch}) one has with $Y^e= Y \cup (Y_0)$ as above:
\[
\eta(D, Y^e)=\eta(D, Y, F_0^+ \oplus \Lambda_Y) +\eta(D,Y_0, F_0^{-}\oplus \Lambda_{Y_0}) + m_{H^*(\partial Y,\mathbb{C})}(\Lambda_Y,\Lambda_{Y_0}),
\]
but the latter summand is zero by the previous lemma.
\begin{folg}
Combining the former two Lemmas, one has on one hand
\[
\eta(D_{\gamma(L_{Y_0})},Y)=\eta(D,Y^e),
\]
and on the other
\begin{equation}\label{eta23}
\eta(D, Y, F_0^+ \oplus \Lambda_Y) =\eta(D, Y^e)- \frac{1}{\pi i}{\rm tr \ log}(\Phi(\Lambda_{Y_0}\oplus F_0^-)\Phi(L_{Y_0}^\infty)^*) + {\rm dim \ ker \ }D_{\Lambda_{Y_0}\oplus F_0^-}.
\end{equation}
\end{folg}
{\it Remark.} So in the latter case, the calculation of the eta-invariant on $Y$ with respect to 'topological' APS-boundary conditions is reduced to the calculation of the eta-invariant on a {\it closed} manifold and the calculation of the 'Maslov-type' term on the right hand side of (\ref{eta23}), which turns out to be zero, as is the result of the next lemma.
\begin{lemma}\label{maslov23}
With the above notation,
\[
\frac{1}{\pi i}{\rm tr \ log}(\Phi(\Lambda_{Y_0}\oplus F_0^-)\Phi(L_{Y_0}^\infty)^*) - {\rm dim \ ker \ }D_{\Lambda_{Y_0}\oplus F_0^-}=0.
\]
\end{lemma}
\begin{proof}
Remember that with respect to the decomposition
\[
L^2(\Omega^*_{\partial X})=(F^-_\nu\oplus F_\nu^+) \oplus(d(E_\nu^+)\oplus
d^*(E_\nu^-))\oplus (d^*(E_\nu^+)\oplus d(E_\nu^-))\oplus \rm{ker}\ A.
\]
one has with the notation from chapter one
\begin{equation}
\lim_{r \rightarrow \infty}L^r_{Y_0}=F_\nu^- \oplus d(E^+_\nu)\oplus
(\gamma(W_{Y_0}^\perp) \oplus W_{Y_0})\oplus \Lambda_{Y_0}.
\end{equation}
where $W_{Y_0} \subset d(E_\nu^-)\subset F_0^-$ s.t. $W_{Y_0}\simeq (im(H^{even}(Y_0,\partial Y_0,\mathbb{C})\rightarrow H^{even}(Y_0,\mathbb{C}))$, $\Lambda_{Y_0}\subset {\rm  ker\ }A$ is the space of limiting values of $L^2$ solutions of $D$ on $Y_{0,\infty}$ as above. On the other hand
\begin{equation}
F_0^- \oplus \Lambda_{Y_0}=F_\nu^-\oplus d^*(E_\nu^+)\oplus d(E_\nu^+)\oplus \Lambda_{Y_0},
\end{equation}
so we see at once that we can reduce to finite dimensions and drop the summand in ${\rm ker}(A)$ since equal summands in both Lagrangians do not contribute by definition:
\[
{\rm log \ tr}(\Phi(\Lambda_{Y_0}\oplus F_0^-)\Phi(L_{Y_0}^\infty)^*)
={\rm log\ tr}(\phi(d^*(E_\nu^+)\oplus d(E_\nu^+))\phi(d(E^+_\nu)\oplus
(\gamma(W_{Y_0}^\perp) \oplus W_{Y_0}))^*),
\]
where $\phi:\mathcal{E}_+\cap (E_\nu^+\oplus E_\nu^-) \rightarrow \mathcal{E}_-\cap(E^+_\nu \oplus E_\nu^-)$ is the corresponding finite dimensional unitary map.
Since
\[
\gamma(d^*(E_\nu^+)\oplus d(E_\nu^+))=d(E_\nu^-)\oplus d^*(E_\nu^-)
\]
and using $\phi(\gamma(V))=-\phi(V)$ for any Lagrangian $V$ we have
\[
{\rm log \ tr}(\Phi(\Lambda_{Y_0}\oplus F_0^-)\Phi(L_{Y_0}^\infty)^*)=
{\rm log\ tr}(-\phi(d^*(E_\nu^-)\oplus d(E_\nu^-))\phi(d(E^+_\nu)\oplus
(\gamma(W_{Y_0}^\perp) \oplus W_{Y_0}))^*),
\]
Taking an orthonormal basis of $d^*(E_\nu^-)\oplus d(E_\nu^-)$ that is adapted to the summands and to the decomposition $d(E_\nu^-)=W_{Y_0}\oplus W_{Y_0}^\perp$ and noting $\gamma(d(E_\nu^+))=d^*(E_\nu^-)$ one sees that since
\[
(d^*(E_\nu^-)\oplus d(E_\nu^-))\cap (d(E^+_\nu)\oplus
(\gamma(W_{Y_0}^\perp) \oplus W_{Y_0})) =W_{Y_0},
\]
only $W_{Y_0}$ contributes with its dimension (for two Lagrangians $U,V$, $U\cap V$ is equal to the $-1$ eigenspace of $-\phi(V)\phi(W)$), the other subspaces in $d^*(E_\nu^-)\oplus d(E_\nu^-)$ cancel in the logarithm with their $\gamma$-composed counterparts (adding ${\rm log}(1)$ summands), so finally since the logarithm was chosen so that ${\rm log}(-1)=\pi i$ we have
\[
\frac{1}{\pi i}{\rm log \ tr}(\Phi(\Lambda_{Y_0}\oplus F_0^-)\Phi(L_{Y_0}^\infty)^*)={\rm dim}\ W_{Y_0}.
\]
We show now that ${\rm dim}\  W_{Y_0}$ is equal to ${\rm dim \ ker \ }D_{\Lambda_{Y_0}\oplus F_0^-}$. Note that the latter is equal to ${\rm dim}\ \gamma(F_0^-\oplus \Lambda_{Y_0})\cap L_{Y_0}$ by Lemma \ref{kerdp}.
By the definition of $\Lambda_{Y_0}$ as a symplectic reduction there is an exact sequence
\[
0\rightarrow L_{Y_0}\cap F_0^+\rightarrow L_{Y_0}\cap(F_0^+ \oplus \gamma(\Lambda_{Y_0}))\rightarrow \Lambda_{Y_0}\cap\gamma(\Lambda_{Y_0})\rightarrow 0,
\]
so ${\rm dim \ ker} \ D_{\Lambda_{Y_0}\oplus F_0^-}={\rm dim}(L_{Y_0}\cap F_0^+)$.
Since $L_{Y_0}\cap F_0^+$ is equal to the $L^2$ solutions of $D\beta=0$ on $Y_{0,\infty}$ which is known to be equal to $W_{Y_0}\simeq {\rm im}(H^{even}(Y_0,\partial Y_0,\mathbb{C})\rightarrow H^{even}(Y_0,\mathbb{C}))$ (compare Theorem \ref{adiabaticlimit} and the remark below Lemma \ref{invariantform}), we arrive at the assertion.
\end{proof}
Using this, we can also identify the $\eta$-invariant with respect to the adiabatic limit of the Calderon projector of $Y$, at least its fractional part:
\begin{folg}
With the above notation,
\[
\begin{split}
\eta(D, Y, F_0^+ \oplus \Lambda_Y) =\eta(D, Y^e)= \eta(D,Y,L^\infty_{Y}),
\end{split}
\]
\end{folg}
\begin{proof}
The first equality is immediate from the above considerations and Lemma \ref{maslov23}, for the second note that
\[
\eta(D,Y,F_0^+ \oplus \Lambda_Y)-\eta(D,Y,L^r_{Y})=\frac{1}{\pi i}{\rm tr \log}(\Phi(F_0^+ \oplus \Lambda_Y)\Phi(L^r_{Y})^*) - {\rm dim \ ker \ }D_{F_0^+\oplus \Lambda_{Y}},
\]
for any $r>0$. Now, as in the proof of Lemma \ref{symm}, we see that $\eta(D,Y,F_0^+ \oplus \Lambda_Y)$ is independent of $r$, the same holds for ${\rm ker \ }D_{F_0^+\oplus \Lambda_{Y}}$ since one argues as in the proof of Lemma \ref{maslov23} with $F_0^-$ replaced by $F_0^+$ and $Y_0$, $\Lambda_{Y_0}$ replaced by $Y$, $\Lambda_Y$ to conclude that is isomorphic to $W_Y\simeq{\rm im}(H^{even}(Y,\partial Y,\mathbb{C})\rightarrow H^{even}(Y,\mathbb{C}))$. Taking the limit $r\rightarrow \infty$, we see in the same way that the limit of first term on the right hand side equals the dimension of $W_Y$, hence both terms on the right cancel and we arrive at the assertion.
\end{proof}
To complete the discussion, we give a formula for the eta-invariant relative to an arbitrary boundary projection $P \in Gr_\infty(A)$. For this, note that if $P,Q,R\in Gr_\infty(A)$, then $(P,Q), (Q,R), (P,R)$ are Fredholm and the differences $P-Q,Q-R,P-R$ are smoothing, hence trace-class, implying $\tau_\mu(P,Q,R)\in \mathbb{Z}$, the triple index as defined in (\cite{lesch}, Definition 6.8) is well-defined.
\begin{lemma}\label{arb}
Let $P \in Gr_\infty(A)$, then the corresponding (reduced) eta-invariant $\tilde \eta(D,Y,{\rm im}\ P)$ can be expressed as follows:
\begin{equation}
\begin{split}
\tilde \eta(D,Y,{\rm im}\ P)= &\frac{1}{2}\eta(D, Y^e)+\frac{1}{2\pi i}{\rm tr \log}(\Phi(P)\Phi(F_0^+ \oplus \Lambda_Y)^*) -\tau_\mu({\rm im}\ P,F_0^+ \oplus \Lambda_Y, L_Y)\\
&+\frac{1}{2}{\rm dim}\ ({\rm im}(H^{even}(Y,\partial Y,\mathbb{C})\rightarrow H^{even}(Y,\mathbb{C})).
\end{split}
\end{equation}
\end{lemma}
\begin{proof}
By substracting the two relations (note that the right hand sides are well-defined since $I-\Phi(F_0^+ \oplus \Lambda_Y)\Phi(L_{Y})^*$ resp. $I-\Phi(P)\Phi(L_{Y})^*$ are by definition smoothing, hence trace class)
\[
\begin{split}
\tilde \eta(D,Y,F_0^+ \oplus \Lambda_Y)-\tilde \eta(D,Y,L_{Y})&=\frac{1}{2\pi i}{\rm tr \log}(\Phi(F_0^+ \oplus \Lambda_Y)\Phi(L_{Y})^*),\\
\tilde \eta(D,Y,{\rm im}\ (P))-\tilde \eta(D,Y,L_{Y})&=\frac{1}{2\pi i}{\rm tr \log}(\Phi(P)\Phi(L_{Y})^*),\\
\end{split}
\]
we infer
\begin{equation}
\begin{split}
\tilde \eta(D,Y,F_0^+ \oplus \Lambda_Y)-\tilde \eta(D,Y,{\rm im}\ P)&=\frac{1}{2\pi i}{\rm tr \log}(\Phi(F_0^+ \oplus \Lambda_Y)\Phi(L_{Y})^*)-\frac{1}{2\pi i}{\rm tr \log}(\Phi(P)\Phi(L_{Y})^*)\\
&= \tau_\mu({\rm im}\ P,F_0^+ \oplus \Lambda_Y, L_Y)-\frac{1}{2\pi i}{\rm tr \log}(\Phi(P)\Phi(F_0^+ \oplus \Lambda_Y)^*),\\
\end{split}
\end{equation}
where $\tau_{\mu}({\rm im}\ P,F_0^+ \oplus \Lambda_Y, L_Y)$ is well defined since the difference of the projections onto $L_Y$ resp. $F_0^+\oplus \Lambda_Y$ and $P$ is trace-class. As in Lemma \ref{maslov23} we deduce that ${\rm ker} \ D_{F_0^+\oplus \Lambda_Y}$ is equal to $W_Y\simeq ({\rm im}(H^{even}(Y,\partial Y,\mathbb{C})\rightarrow H^{even}(Y,\mathbb{C}))$,  so since $\eta(D,Y,F_0^+ \oplus \Lambda_Y)=\eta(D, Y^e)$ we arrive at the assertion.
\end{proof}
From this, one can deduce a formula for general $APS$-Lagrangians which is determined only by {\it finite}-dimensional expressions:
\begin{lemma}
Set $P^+(\Lambda) \in Gr_\infty(A)$ s.t. ${\rm im}(P^+(\Lambda))=F_0^+\oplus \Lambda$ for $\Lambda \subset{\rm ker\ A}$ arbitrary and $n=2k$, then:
\[
\begin{split}
\tilde \eta(D,Y,P^+(\Lambda))= &\frac{1}{2}\eta(D, Y^e)+\frac{1}{2\pi i}{\rm tr \log}(\phi(\Lambda)\phi(\Lambda_Y)^*)\\
&+\frac{1}{2}{\rm dim}\ (im(H^{even}(Y,\partial Y,\mathbb{C})\rightarrow H^{even}(Y,\mathbb{C})),
\end{split}
\]
where here, the image of $\phi:Gr({\rm ker} A) \rightarrow \mathcal{U}({\rm ker}(A)\cap \mathcal{E}_i,{\rm ker}(A)\cap \mathcal{E}_{-i})$ is a finite dimensional unitary mapping.
\end{lemma}
\begin{proof}
That
\[
\frac{1}{2\pi i}{\rm tr \log}(\Phi(F_0^+ \oplus \Lambda)\Phi(F_0^+ \oplus \Lambda_Y)^*)=\frac{1}{2\pi i}{\rm tr \log}(\phi(\Lambda)\phi(\Lambda_Y)^*)\\
\]
is clear directly from the definition since $F_0^+$ contributes with $log(1)$-summands. So we are left with showing that
the triple index $\tau_\mu(F_0^+\oplus \Lambda,F_0^+ \oplus \Lambda_Y, L_Y)$ vanishes. For this, note that one has
\[
\tau_\mu(F_0^+\oplus \Lambda,F_0^+ \oplus \Lambda_Y, L_Y)=\tau_\mu(F_0^+\oplus \Lambda,F_0^+ \oplus \Lambda_Y, {\rm lim}_{r \rightarrow \infty} L^r_Y),
\]
since by the exact sequence
\[
0\rightarrow L^r_Y \cap F_0^+\rightarrow L^r_{Y_0}\cap(F_0^+ \oplus \Lambda)\rightarrow \Lambda_{Y}\cap\gamma(\Lambda)\rightarrow 0,
\]
the middle term is isomorphic to $W_Y\oplus (\Lambda \cap \Lambda_Y)$, since $W_Y,\Lambda_Y$ are topological quantities, hence of constant dimension under varying collar length, $L^r_{Y}\cap(F_0^+ \oplus \gamma(\Lambda))$ is independent of $r$, so we can invoke (\cite{lesch}, Lemma 6.10). Finally note that for two triples of projections $(P,Q,R),(P',Q',R')$ in $Gr_\infty(A)$ such that the corresponding triple indices are well-defined one has
\[
\tau_\mu(P\oplus P',Q\oplus Q',R\oplus R')=\tau_\mu(P,Q,R)+\tau_\mu(P',Q',R'),
\]
so using again
\[
\begin{split}
L^\infty_{Y}&=F_\nu^+ \oplus (W_Y\oplus\gamma(W_Y^\perp))\oplus d(E^-_\nu)\oplus \Lambda_Y,\\
F_0^+ \oplus \Lambda_{Y}&=F_\nu^+\oplus d^*(E_\nu^-)\oplus d(E_\nu^-)\oplus \Lambda_{Y},
\end{split}
\]
we have that  $\tau_\mu(F_0^+\oplus \Lambda,F_0^+ \oplus \Lambda_Y, L^\infty_Y)$ equals
\[
\begin{split}
&\tau_\mu(F_\nu^+,F_\nu^+,F_\nu^+)_{F_\nu^-\oplus F_\nu^+} + \tau_\mu(d^*(E_\nu^-), d^*(E_\nu^-), W_Y\oplus\gamma(W_Y^\perp))_{d(E_\nu^+)\oplus d^*(E_\nu^-)}\\
&+\tau_\mu(d(E_\nu^-), d(E_\nu^-),d(E_\nu^-))_{d^*(E_\nu^+)\oplus d(E_\nu^-)}+\tau_\mu(\Lambda, \Lambda_Y,\Lambda_Y)_{d^*(E_\nu^+)\oplus d(E_\nu^-)}=0,\\
\end{split}
\]
using Proposition 6.11 of \cite{lesch}.
\end{proof}
Interpreting $Y^e$ as a mapping cylinder, using Lemma \ref{isometry}, we can find a formula for its monodromy, acting on $H^n(Y^e_u,\mathbb{C})$. For this,fix one fibre $Y^e_u= F^e=F\cup (-F)$ (here and in the following, we denote $-F$ the oppositely oriented pair $(F,\partial F$) for a given oriented manifold with boundary $F$) and consider the maps
\[
\begin{split}
&i_1:F^e \rightarrow F \quad {\rm induced \ by\ identifying} \quad  -F \rightarrow F,\ F \rightarrow F,\\
&i_2:(F^e, \varnothing) \hookrightarrow (F^e,(-F)).
\end{split}
\]
These induce maps on cohomology
\[
\beta_1=i^*_1:H^n(F,\mathbb{C})=:U^* \rightarrow H^n(F^e,\mathbb{C})=: U^e, \quad
\beta_2=i_2^*:H^n(F^e, -F,\mathbb{C})=:U \rightarrow H^n(F^e,\mathbb{C}),
\]
\begin{lemma}\label{betaisom}
$(\beta_1,\beta_2):U^* \oplus U \rightarrow U^e$ is an isomorphism.
Set $\rho^e \in Aut(U^e, b^e)$, where $b^e$ is the nondegenerate intersection
form on $\tilde F^e$ equal to
\[
\rho^e=\begin{pmatrix} I & 0 \\ V & \rho\end{pmatrix},
\]
where $\rho:U \rightarrow U$ is the algebraic monodromy of the Milnor bundle $f:Y \rightarrow
S^1_\delta$, $V:U^*\rightarrow U$ its variation map. Set $(\Phi_f^e)^*(\omega)=\Phi_f^e(1,\cdot)^*(\omega)$ for $\omega \in
H^*(F^e)$, then
\[
(\Phi_f ^e)^*=\rho^e.
\]
\end{lemma}
Note that the variation map $V$ is defined as the mapping $c \in H^*(F,\mathbb{C}) \mapsto [(\rho-id)^*c] \in H^*(F,\partial F,\mathbb{C})$ if $\rho$ is a representative of the geometric monodromy of $Y$, that is $\rho=\Phi_f(1)$.
\begin{proof}
We set $\hat F= -F$, the long exact sequence for the pair $(F^e,\hat F)$ gives a short exact sequence
\[
0 \rightarrow H^{n}(F^e,\hat F, \mathbb{C}) \xrightarrow
{i_2^*} H^n(F^e,\mathbb{C}) \xrightarrow {r} H^{n}(\hat F,\mathbb{C})
\rightarrow 0.
\]
Defining $i_1$ as above (and identifying $H^{n}(\hat F,\mathbb{C})$ with $H^{n}(F,\mathbb{C})$) it satisfies
\[
r \circ i_1^*= id,
\]
so the sequence splits, i.e. the map $\phi=i_1^* \circ r$ is idempotent and
we have $\rm{ker}(\phi)=\rm{ker}(r)=\rm{im}(i_2^*)$ and
\[
H^n(F^e,\mathbb{C})= \rm{ker}(\phi)\oplus\rm{im}(\phi).
\]
Since $\rm{im}(\phi)=\rm{im}(i_1^*)$, we have the first assertion since
$\beta_{1,2}=i_{1,2}^*$.\\
For the second assertion note that there are isotopies $\Phi^e_f (1)|_F \simeq
\rho_g$ and $\Phi^e_f (1)|_{\hat F}\simeq id$, where $\rho_g:F
\rightarrow F$ denotes a representative of the geometric monodromy, viewing
$F, \hat F \subset F^e$. Given $(\theta,\omega) \in
U^*\oplus U=H^n(F,\mathbb{C})\oplus H^n(F^e,\hat F,\mathbb{C})$ we calculate
\[
(\beta_1,\beta_2)^{-1}\circ (\Phi^e_f)^*\circ
(\beta_1,\beta_2)(\theta,0)=(\beta_1,\beta_2)^{-1}((\Phi^e_f)^*(\theta +
\hat \theta))=(\beta_1,\beta_2)^{-1}(\rho_g^*(\theta)+\hat \theta),
\]
where $\theta + \hat \theta$ denotes the cocycle on $F^e$ restricting to $\theta$ on $F\subset F^e$ and to $\hat \theta:=r^*(\theta)$, where $r:-F\rightarrow F, \ r(x)=x$ (see the remark below) on $-F$. Now one sees that
\[
(\beta_1,\beta_2)(\theta, \rho(\theta)-\theta)=(\rho(\theta)+ \hat \theta),
\]
note that $V(\theta)=\rho(\theta)-\theta \in H^n(F^e,\hat F,\mathbb{C})$ since
the latter can be considered as the cocycles vanishing on boundary cycles so we
have
\[
(\Phi^e_f)^*((\theta,0))=(\theta,V(\theta)) \in U^*\oplus U.
\]
On the other hand
\[
\begin{split}
(\beta_1,\beta_2)^{-1}\circ (\Phi^e_f)^*\circ (\beta_1,\beta_2)(0,\omega)&=
(\beta_1,\beta_2)^{-1}((\Phi^e_f)^*(i_2^*(\omega))=(\beta_1,\beta_2)^{-1}(\rho_g^*(i_2^*(\omega))\\
&=(0, \rho(\omega)),
\end{split}
\]
since  $i_2^*(\omega)$ restricts to the zero cocycle on $\hat F\subset F$, note that here, $\rho$ is considered to act on relative cocycles. Writing this again as
\[
(\Phi^e_f)^*((0,\omega))=(0,\rho(\omega)) \in U^*\oplus U,
\]
we arrive at the assertion.
\end{proof}
{\it Remark.} Note that on the closed manifold $F^e$, $H^n(F^e,\mathbb{C})$ is naturally identified with the set of harmonic $n$-forms chracterized (in the closed case) by $d\beta=\delta\beta=0$, where $\delta=(-1)^{2n(n+1)+1}*d*:\Omega^n(F^e,\mathbb {C})\rightarrow \Omega^{n-1}(F^e,\mathbb{C})$ is the codifferential, in the following we will always implicitly use this identification, since ${\rm ker} \tilde A$ consists exactly of the harmonic forms on $F^e$, $\tilde A$ being the signature operator on $F^e$ (see \cite{Atiyah}).  For any $p$-form $\omega$ on $F$ (or $-F$) one has along $\partial F$ the decomposition
\begin{equation}\label{bdec}
\omega=\omega_{tan}+\omega_{norm}
\end{equation}
where $\omega_{tan}$ agrees with $\omega$ on $p$-tuples of vectors being tangent to $\partial F$ and is zero otherwise, $\omega_{norm}$ is defined so that
(\ref{bdec}) holds. Then it is well-known that one has the following identification (see  \cite{capmil})
\[
\begin{split}
H^*(F,\mathbb{C})&\simeq \{\omega \in \Omega^*(F,\mathbb{C}) \ |\ d\omega=0,\ \delta\omega=0,\ \omega_{norm}=0\}\\
H^*(F,\partial F, \mathbb{C})&\simeq \{\omega \in \Omega^*(F,\mathbb{C}) \ |\ d\omega=0,\ \delta\omega=0,\ \omega_{tan}=0\},
\end{split}
\]
Using this, one sees that $\beta_1$ corresponds to the mapping which takes a form $\omega\in H^n(F,\mathbb{C})$ to the harmonic form $\omega + \hat \omega\in H^n(F^e,\mathbb{C})$ whose restriction to $F\subset F^e$ coincides with $\omega$ and so that $(\omega+\hat \omega)|_{\hat F}=\hat \omega$, where $\hat \omega$ means $\omega$ reflected to $-F$ by $r:F\rightarrow -F, r(x)= x)$, the form $\omega+\hat \omega$ is well-defined since $\omega_{norm}=0$ along $\partial F$. $\beta_2(\omega)$ corresponds to the unique solution $\tilde \omega$ of $(d+\delta)\omega=0$ on $F^e$ which restricts to $\omega$ on $F$ (the uniqueness being a consequence of  'unique continuation', see e.g. Theorem \ref{UCP}).\\
Set for the following
\[
M^{\pm}=\pm F^e\times [0,1] \quad {\rm s.t.} \quad Y^e \simeq M^+ \cup_{(\Phi_f^e(1),id) }M^-,
\]
where the latter isomorpism is an isometry following Lemma \ref{isometry}. \\
Let now $N=(F^e \sqcup (-F^e))$ be the boundary of $M^+$ (so $-N =\partial M^-$), $\hat A$ the signature operator w.r.t. $g^Y|_N$ (here, the $-$-sign indicates 'right-boundaries')
\begin{lemma}\label{glue}
Let $\Lambda_0= \subset ker \hat A$ be the subspace of limiting values of extended $L^2$-solutions of $D\beta=0$ for the signature operator in the form $D=\hat \gamma(\frac{\partial}{\partial z_i}+\hat A)$ on the cylinder $Z=[0,1]\times  F^e$ with the above
product metric, then
\[
\eta(D,Y^e)=m_{H^*(N,\mathbb{C})}(\Lambda_0,\hat \rho^e(\Lambda_0)) \in \mathbb{R},
\]
where the latter refers to the so called  'averaged Maslov index' $m_{H^*(N,\mathbb{C})}(L,\hat \rho^e(L))$ (see Definition \ref{avmaslov}, also  \cite {lesch}).
\end{lemma}
Note that here $\hat \rho^e=\rho^e \oplus id$, corresponding to the decomposition $H^*(N,\mathbb{C})=H^*(F^e,\mathbb{C})\oplus H^*(-F^e,\mathbb{C})$.
\begin{proof}
We claim that using Theorem \ref{invariantform}, Theorem \ref{generalglue} resp. Definition \ref{avmaslov}, one has denoting $Y^e_0:=F^e\times S^1=M^+\cup_{id}M^-$ the 'trivially glued' cylinder with the induced product metric that
\begin{equation}\label{cutandpaste}
\begin{split}
\eta(D,Y_0^e)&=\eta(D,M^+,\Lambda_0\oplus F_0^+)+ \eta(D;M^-,F_0^-\oplus \Lambda_0)+m_{H^*(N,\mathbb{C})}(\Lambda_0,\Lambda_0))\\
\eta(D,Y^e)&=\eta(D,M^+,\Lambda_0\oplus F_0^+)+ \eta(D;M^-,F_0^-\oplus\Lambda_0)
+m_{H^*(N,\mathbb{C})}(\Lambda_0,\hat \rho^e(\Lambda_0))
\end{split}
\end{equation}
where the first line follows directly from Theorem \ref{invariantform}. To derive the second line from Theorem \ref{generalglue}, we first claim two facts. For this let $M_r^\pm$ be $M^\pm$ with a metric collar $N \times [0,r]$ of length $r$ glued to $\partial M^\pm$. Then
\begin{enumerate}
\item The Calderon projectors $P_{M^\pm_r}, r>0$ are independent of $r$.
\item The subspaces $W_{M^\pm}$ are zero, that is the non-resonance-levels of $M^\pm$ are zero.
\end{enumerate}
To prove the fist claim, note that $T_r:L^2(\Omega^{even}(N \times [0,1]))\rightarrow L^2(\Omega^{even}(N \times [0,r]))$ given by $\omega(x,s)\mapsto \sqrt{r}\omega(x,(1/r)s)$ is an isometry w.r.t the $L^2$-metric. Also it follows from a direct computation that $T_r(H_{1/2}(\Omega^{even}(N \times [0,1]))\subset H_{1/2}(\Omega^{even}(N \times [0,r]))$.
Hance $T_r$ preserves the $0$-eigenspace of $D$ acting on $H_{1/2}$, hence $r({\rm ker}\ (D\circ T_r^{-1}))$ equals $r({\rm ker}\ D)$. Finally the second claim follows directly from the fact that ${\rm im}(H^*(M^\pm, \partial M^\pm, \mathbb{C})\rightarrow H^*(M^\pm, \mathbb{C}))$ is zero by the Kuenneth formula and the fact that ${\rm im}(H^*([0,1], \{0,1\}, \mathbb{C})\rightarrow H^*([0,1]), \mathbb{C})$ is zero. We thus have, note that $F_0^-\cap L_{M^+}=0$ and $F_0^+\cap L_{M^-}=0$, by the last of the above two facts (cf. \cite{Nicol2}). Thus we can write
\[
\begin{split}
L_{M^+}&=F_0^+ \oplus \Lambda_{M^+}\\
\gamma(L_{M^-})&=\gamma(F_0^- \oplus \Lambda_{M^-})=F_0^+\oplus \gamma(\Lambda_{M^-})
\end{split}
\]
Now note that $\Phi_f^e(1)^*\oplus id =\hat \rho^e$ preserves the identification of harmonic forms with elements of ${\rm ker} \hat A$ while $\hat \rho^e(F_0^+)=F_0^+$, furthermore since ${\rm ker}\ D$ on $Y^e$ can be identified via harmonic forms with $H^{even}(Y^e, \mathbb{C})$ we deduce by using the Wang exact sequence oy $Y^e$ (see the proof of Lemma \ref{limiting}) that 
\[
{\rm dim}{\rm ker}\ D ={\rm dim\ ker}(\Phi_f^e(1)^*-Id)={\rm dim}(\hat \rho^e(\Lambda_0\cap \Lambda_0)
\]
where for the latter equality we have used Lemma \ref{limiting}. Then using Theorem \ref{generalglue} and Definition \ref{avmaslov} we arrive at the second line in (\ref{cutandpaste}) (noting $\Lambda_0=\Lambda_{M^\pm}$). Substracting the two identities and using the antisymmetry of $m_{H^*(N,\mathbb{C})}(\cdot,\cdot)$ one arrives at
\[
\eta(D,Y^e)-\eta(D,Y_0^e)=m_{H^*(N,\mathbb{C})}(\Lambda_0,\hat \rho^e(\Lambda_0)),
\]
note that this equality is valid in $\mathbb{R}$ only for this special choice of $APS$-Lagrangians. Finally using that $\eta(D,S^1\times F^e)=0$ since there is an isometry which anticommutes with $D$ (as in the proof of Lemma \ref{symm}) one arrives at the assertion.
\end{proof}
So calculating the invariant $m_{H^*(N,\mathbb{C})}(\Lambda_0,\hat \rho^e(\Lambda_0))$ proofs the theorem. For this one calculates the following.
\begin{prop}\label{maslov}
Let $\Lambda_0 \subset {\rm ker}\ \hat A\simeq H^*(F^e  \times (-F^e),\mathbb{C})$ be the space of limiting values of extended $L^2$-solutions of $Df=0$ on $[0,1]\times F^e$ relative to the product metric, then for $n=2k$ or $n=2k+1$:
\[
m_{H^*(N,\mathbb{C})}(\Lambda_0,\rho^e(\Lambda_0))=\sum_{\alpha \in \Lambda, l(\alpha) \notin \mathbb{Z}}(-1)^{[l(\alpha)]+n}\left(1-2\{l(\alpha)\}\right) +\frac{{\rm arg}(-1+\frac{4}{3}i)}{\pi}\sum_{\alpha \in \Lambda, l(\alpha)\in \mathbb{Z}}(-1)^{l(\alpha)+n},
\]
so the first sum represents the contribution of the non-trivial part of $\rho$, the second sum amounts to a summation over a basis for the $\lambda=1$-eigenspace of $\rho$.
\end{prop}
\begin{proof}
In the following, we will consider the case $n=2k, k \in \mathbb{N}$, the case $n=2k+1$ is analogous and the result differs from the fist case by a global sign $(-1)^n$. First consider the exact sequence (we identify $F$ and $\tilde F$)
\begin{equation}\label{seq2}
0 \rightarrow H^{n-1}(\partial F,\mathbb{C}) \xrightarrow {\delta}
H^n(F,\partial F,\mathbb{C})
\xrightarrow {b} H^{n}(F,\partial F,\mathbb{C})^* \xrightarrow {r} H^{n}
(\partial F,\mathbb{C})\rightarrow 0.
\end{equation}
Here, we have identified the natural map  $j:H^n(F,\partial F,\mathbb{C})\rightarrow H^n(F,\mathbb{C})$ induced by ithe inclusion $\hat j:(F,\emptyset)\rightarrow (F,\partial F)$ with $b:H^n(F,\partial F,\mathbb{C})\rightarrow H^n(F,\partial F, \mathbb{C})^*$ , where $b(\alpha,\cdot)=<\alpha, j(\cdot)>$ and $<\ ,\ >:H^n(F,\partial F,\mathbb{C})\times H^n(F,\mathbb{C})\rightarrow \mathbb{C}$ is the Poincare-Lefschetz duality pairing identifying $H^n(F,\mathbb{C})$ with $H^n(F,\partial F,\mathbb{C})^*$. We can decompose $U=H^n(F,\partial F,\mathbb{C})=\rm{ker}(b) \oplus V_1$, $U^*=H^n(F,\mathbb{C})=\rm{coker} (b) \oplus V_2$, so that
$b:V_1\rightarrow V_2$ is an isomorphism. We then have
\begin{lemma}
Let $*: \Omega^n(F,\mathbb{C})=\Omega^n(F,\mathbb{C})$ be the Hodge star
operator associated to the induced metric on $F$. Then $*$ defines isomorphisms
$*:V_{1}\rightarrow V_{2}$ and $*:\rm{ker}(b) \rightarrow \rm{coker}(b)$.
\end{lemma}
\begin{proof}
For any $\theta \in H^n(F,\partial F,\mathbb{C})$, the expression
\[
b(\theta,* \theta)=i^r \int_F\theta\wedge *\theta =c\int_F dv_F=(-1)^n b(*\theta, \theta)
\]
$(c \neq 0)$ is nonzero, so for $\theta \notin \rm{ker}(b)$, $*\theta \in H^n(F,\mathbb{C})$ is in the
image of  $j:H^n(F,\partial F,\mathbb{C})\rightarrow H^n(F,\mathbb{C})$. For if $*\theta \in \rm{coker}(j)$ then reading the above sequence backwards and considering $<\ ,\ >$ on $U^*$ shows $b(*\theta, \theta)=0$. For $\theta \in \rm{ker} (b)$, there is no element $\alpha \in H^n(F,\partial F,\mathbb{C})$ so that $b(\theta)(\alpha)\neq 0$, so especially $*\theta \in \rm{coker}(b)$. Since $*^2=1$, both mappings are injective, hence surjective.
\end{proof}
Since $N=F^e \bigsqcup (-F^e)$, we have $H^n(N,\mathbb{C})=H^n(F^e,\mathbb{C})
\oplus H^n(-F^e,\mathbb{C})$. In the following, for every element $\theta \in
H^n(F^e,\mathbb{C})$, we will denote $\overline \theta$ the corresponding
element in $H^n(-F^e,\mathbb{C})$. Let $\hat D$ be the signature operator on the
cylinder $Z=[0,1]\times F^e$, resp. product metric, it induces an anti-involution
$\hat \gamma$ on $\Omega^*(N,\mathbb{C})$ where $N=\partial Z$ so that
\[
\hat D=\hat \gamma(\partial/\partial x+\hat A).
\]
Now restricting to cohomology we have
\[
\hat \gamma=\begin{pmatrix} \gamma & 0\\0 & -\gamma \end{pmatrix}
\]
relative to the decomposition of $H^n(N,\mathbb{C})$ above.
With these notations, if for elements $\theta^{\pm}\in H^n(F^e,\mathbb{C})$
$\hat \gamma\theta^{\pm}= \pm i\theta^{\pm}$, then $\hat \gamma \overline
\theta^{\pm}= {\mp} i\overline \theta^{\pm}$.
\begin{lemma}\label{hodge}
Under the isomorphism $\beta:=(\beta_1,\beta_2):U^* \oplus U \rightarrow H^n(F^e,\mathbb{C})$ the Hodge star $*_{F^e}$ splits as
\[
(\beta^{-1}\circ *_{F^e}\circ \beta)(\theta, \omega)=(*\omega,*\theta)
\]
if $\theta \in V_{2}, \omega \in V_1$,
\[
(\beta^{-1} \circ *_{F^e}\circ \beta^{-1})(\theta,\omega)=(*\omega,*\theta),
\]
for $\theta \in \rm{coker}(b)$, $\omega \in \rm{ker}(b)$.
\end{lemma}
\begin{proof}
Using the notation from Lemma \ref{betaisom} and its proof, the $*$-operators of  $F$, $\hat F$, resp. $F^e$ together with the long exact sequence of the pair $(F^e,\hat F)$ induce the following commutative diagram. To see this, recall the description of the (relative) cohomology of $F$, $\hat F$ resp. $F^e$ as harmonic forms discussed below Lemma \ref{betaisom}. The rows of the diagram follow from the long exact sequence of the pair $(F^e, \hat F)$ and by using $H^n(\hat F,\mathbb{C})\simeq H^n(F,\mathbb{C})$. From the fact that $*_F$ interchanges Neumann- and Dirichlet-boundary conditions we see that the vertical arrows are well-defined. Note that in the right hand diagram (analogous on the left), we flipped orientation on the upper horizontal and vertical arrow on the right, but suppressed it in the notation (the two signs cancel in all relevant mappings). Now let $\theta \in H^n(F,\mathbb{C})$. Then, since $*$ acts pointwise, $(*_{F^e} \circ i_1^*)(\theta)$ is a harmonic form on $F^e$, whose restriction to $F$ coincides with $*_F \theta$ and hence with $(i_2^* \circ *_F )(\theta)| F$.  Now by unique continuation, both forms coincide on the whole of $F^e$, the left diagram is proven analogously.
\[
\begin{CD}
 0 \rightarrow H^{n}(F^e,\hat F, \mathbb{C})  @>>{i_2^*}> H^n(F^e,\mathbb{C}) @<<i_1^*<H^{n}(\hat F,\mathbb{C}) \rightarrow 0\\
 @VV*_{F}V    @VV*_{F^e}V  @VV*_{F}V  \\
0 \leftarrow H^{n}(\hat F, \mathbb{C})  @>>i_1^* > H^n(F^e,\mathbb{C}) @<<{i_2^*}< H^{n}(F^e,\hat F,\mathbb{C}) \leftarrow 0.\\
\end{CD}
\]
Summarizing, we deduce from the diagram the equalities
\[
*_{F^e}\circ i_1^*=i_2^*\circ *_{F} \quad {\rm resp.}\quad *_{F^e} \circ i_2^*=i_1^*\circ *_{F},
\]
so with $\omega$, $\theta$ as above $(*_{F^e}\circ \beta) (\theta,\omega)=\beta(*_F\omega,*_F\theta)$ which is the assertion.
\end{proof}
Since $\gamma=i^r(-1)^l*_{F^e}$ for some $r,l$ depending on degrees and dimension (see below), these equations continue to hold for $\gamma$ replacing $*_{F^e}$ and defining $\gamma_{1,2}:V_{1,2}\rightarrow V_{1,2}$, $\gamma_F: {\rm ker}(b)\rightarrow {\rm coker}(b)$ by the above equations, we will suppress the
indices of $\gamma$ below.\\
Following well-known results (see for instance \cite{nem2}), one can choose bases $\theta_1,\dots, \theta_\mu \in U$ resp. dually $\omega_1,\dots,\omega_\mu \in U^*$, (i.e. $<\omega_i, \theta_j>=\delta_{ij}$, where $<,>$ is the Poincare-duality pairing), so that the tuple $\mathcal{V}(f):=(U,b,\rho,V)$, where $V:U^*\rightarrow U$ is the variation mapping, for a quasihomogeneous singularity has the following decomposition in terms of the generalized eigenspaces of the monodromy operator $\rho$:
\begin{equation}\label{weightvar}
\mathcal{V}(f)=\bigoplus_{\alpha \in \Lambda}\mathcal{W}_{exp2\pi il(\alpha)}((-1)^{[l(\alpha)]+n})
\end{equation}
where for $|\lambda|=1$, one defines
\begin{equation}\label{nemweight}
\begin{split}
\mathcal{W}_\lambda(\pm 1)&=(\mathbb{C}, \pm i^{-n^2},\lambda,\pm(\lambda-1)i^{n^2}),\quad if \lambda \neq 1,\\
\mathcal{W}_1(\pm 1)&=(\mathbb{C}, 0,1_\mathbb{C},\pm i^{n^2+1}),\quad if \lambda = 1.
\end{split}
\end{equation}
\begin{lemma}\label{unitary}
There is a set of monomials $z^{\alpha_i},\alpha_i \in \Lambda, i=\{1,\dots,\mu\}$ spanning $\mathcal{O}_{\mathbb{C}^{n+1},0}/{\rm grad}(f)\mathcal{O}_{\mathbb{C}^{n+1},0}$ so that the associated elements in $\mathcal{H}^n(f_*\Omega^\cdot_{X/D'_\delta})$ generate the latter as a module over $\mathcal{O}_{D'_\delta}$ and, by restriction to $F$, give rise to a basis in $\theta_i,\dots, \theta_i \in U, i=\{1,\dots,\mu\}$ satisfying the decomposition (\ref{nemweight}). Then, after possibly orthogonalizing elements $\theta_i$ spanning $V_n:={\rm span}\ \{\theta_i:l(\alpha_i) =n \in \mathbb{N}\}$ for some fixed $n$, we can assume that $\omega_i=*\theta_i$ for all $i$ (the $\omega_i$ representing the dual basis in $U^*$), i.e. the set $\{\theta_i\}_{i=1,\dots,\mu}$ is unitary with respect to the $L^2$-inner product on $F$, that is
\[
\int_F\theta_i\wedge \overline *\theta_j=\delta_{ij}, \ i,j \in \{1,\dots, \mu\},
\]
where $\mu$ is the Milnor number of $F$.
\end{lemma}
\begin{proof}
The first assertion, concerning intersection form $b$ and monodromy $\rho$, follows from a classical calculation, see for instance Brieskorn \cite{bries} and \cite{steenbrink}, the asserted form of the variation mapping (\ref{nemweight}) is given in \cite{nem2}.\\
Now for $\theta_i \notin {\rm ker}(j)$ and and since $n=2k, k \in \mathbb{N}$ the set $j(\theta_i)$ diagonalize the intersection form on the image of $j$ with diagonal $\pm 1$. On the other hand since $*$ is unitary with respect to the intersection form, the $\theta_i$ span the $\pm 1$ eigenspaces of $*$, considering $j^{-1}\circ *$  as an endomorphism on $V_1$, that is one has (for the first equality cf. \cite{getzler})
\[
\int_F \theta_i \wedge \overline \theta_j=\pm \int_F\theta_i\wedge \overline *\theta_j= \pm \delta_{ij}\int_F|\theta_i|^2.
\]
So the $\{\theta_i\}_i$ are orthogonal and we can also assume them to be orthonormal.\\
For $\theta_i \in {\rm ker}(j)$, $b(\theta_i;\cdot)$ is degenerate on $U$ and the Hodge star is a map $*:{\rm ker}\ j \rightarrow {\rm coker}\ j$.
It is sufficient to show that for any basis $\{\theta_j\}_j$ of ${\rm ker}\ j$ which is given by monomials $z^{\alpha_j}, j \in\{1,\dots,{\rm dim\ ker}\ b\},\ \alpha_j\subset \Lambda \subset \mathbb{N}^{n+1}$ so that $l(\alpha_j) \in \mathbb{Z}$, where $l(\alpha_j)=\sum_i((\alpha_j)_i+1)w_i$, $*$ maps each subspace $V_p:={\rm span}\ \{\theta_j :l(\alpha_j)=p, p \in\mathbb{Z}\}$ to its corresponding dual $V_p^*\subset {\rm coker}\ j$ and the $V_p$ are orthogonal with respect to the hodge inner product on $F$. That $*:V_p\rightarrow V_p^*$ is clear from the definition of $*$, for the orthogonality of the $V_p$ just consider that since there is a horizontal vector field $\hat X_f$ whose flow induces fibrewise isometries, we have the commuting diagram
\[
\begin{CD}
\Gamma(\mathcal{H}_c^n(f_*\Omega^\cdot_{X/D_\delta}))  @>>\overline *>  \Gamma(\mathcal{H}^n(f_*\Omega^\cdot_{X/D_\delta}))    \\
 @VV\mathcal{L}_{X_f}V    @VV \mathcal{L}_{X_f} V \\
\Gamma(\mathcal{H}_c^n(f_*\Omega^\cdot_{X/D_\delta}))  @>>\overline *> \Gamma( \mathcal{H}^n(f_*\Omega^\cdot_{X/D_\delta})),
\end{CD}
\]
where here $\overline *$ denotes the fibrewise $\overline *$-operation and $\mathcal{H}_c^n(f_*\Omega^\cdot_{X/D_\delta})$ the sheaf of $n$-th cohomologies with compact support. Now if $s_i, s_j \in \Gamma(\mathcal{H}^n(f_*\Omega^\cdot_{X/D_\delta}))$ restrict to $\theta_i \in V_p, \theta_J\in V_q$ respectively for $q\neq p$ then since the flow of $\hat X_f$ preserves the intersection form we have
\[
0=\frac{d}{dt}\Phi_{X_f}(\int_Fs_j\wedge \overline *s_i)=\int_F\left(\mathcal{L}_{\hat X_f}s_j\wedge\overline *s_j+s_j\wedge \overline *\mathcal{L}_{\hat X_f}s_i\right)=4\pi i(q-p)\int_F s_j\wedge \overline *s_i,
\]
 so since $q\neq p$ we arrive at $V_q\perp V_p$. Now if ${\rm dim }\ V_p>1$ for some $p\in \mathbb{N}$, we apply Gram-Schmidt to get an orthogonal basis in $V_p$. This does not affect the form of the $\mathcal{W}_1(\pm 1)$, since $l(\alpha)$ is constant on $V_p$, so we arrive at the assertion.
\end{proof}
Using the last lemma, let $J^+,J^-\subset {1,\dots,\mu}$ be so that $b$ is positive resp. negative definite on the subspaces spanned by $\{\theta_i\}_{i \in J^{\pm}}$, we write $\theta^{\pm}_j$ if $j \in J^{\pm}$. Denote the corresponding elements of the dual base by $\omega_j^{\pm}, j \in J^{\pm}$.
On the other hand, we write $\theta^0_i, i \in J^0$ if $\theta_i \in {\rm ker}(b)$,
$\omega^0 _j, j \in J^0$ for the dual elements (hence if $\omega_j^0 \in {\rm coker} (b)$).
We now decompose ${\rm ker}\ A \cap H^n(F^e,\mathbb{C})=({\rm ker}(\gamma-iI)\oplus{\rm
ker}(\gamma+iI))\cap H^n(F^e,\mathbb{C})=:K^+\oplus K^-$ by
\[
K^+={\rm span}(e_j^+,e_j^-, e_i^0)_{j \in J^{\pm}, i \in J^0},\quad K^-={\rm span}(\tilde f_j^+,\tilde f_j^-, \tilde f_i^0)_{j \in J^\pm, \ i \in J^0}
\]
where (we use $i\gamma=i^{2k+2}(-1)^{k+1}*=*$ for $n=2k$), the use of  $\tilde{}$ will become clear below:
\[
\begin{split}
e_i^{\pm, 0}&:=\theta^{\pm,0}_i-i\gamma\theta^{\pm,0}_i=\theta_i^{\pm,0}-\omega_i^{\pm,0},\\
\tilde f_i^{\pm,0}&:=\theta^{\pm,0}_i+i\gamma\theta^{\pm,0}_i, = \theta_i^{\pm,0}+\omega_i^{\pm,0},\ i \in J^{\pm,0}.
\end{split}
\]
Recall  $N=F^e \bigsqcup (-F^e)$, the same vectors as above, but considered on $-F^e$ have sign-reversed eigenvalues for $\hat \gamma|(-F^e)=-\gamma$, so we have
${\rm ker}\ \hat A \cap H^n(N,\mathbb{C})=({\rm ker}(\hat \gamma-iI)\oplus{\rm
ker}(\hat \gamma+iI))\cap H^n(N,\mathbb{C})=:\hat K^+\oplus \hat K^-=
(K^+\oplus \overline K^-)\oplus (K^- \oplus \overline K^+)$, and
\[
\overline K^-={\rm span}(\tilde e_j^+,\tilde e_j^-, \tilde e_i^0)_{j \in J^{\pm}, i \in
J^0},\quad \overline K^+={\rm span}(f_j^+,f_j^-, f_i^0)_{j \in J^{\pm}, i \in J^0}
\]
where
\[
\begin{split}
\tilde e_i^{\pm,0}&:=\overline \theta^{\pm,0}_i+i\gamma \overline \theta^{\pm,0}_i=\overline \theta_i^{\pm,0}+\overline \omega_i^{\pm,0},\\
f_i^{\pm,0}&:=\overline \theta^{\pm,0}_i-i\gamma\overline \theta^{\pm,0}_i = \overline \theta_i^{\pm,0}-\overline \omega_i^{\pm,0},
\ i \in J^{\pm,0}.
\end{split}
\]
We now define a Lagrangian $L \subset {\rm ker}\ \hat A$ by specifying its associated isometry $P_L:\hat K^+\rightarrow \hat K^-$ (abusing notation in the following by writing $P_L$ for $\Phi(P_L)$, note also that we assume in the following that on the complement of $H^n(N,\mathbb{C})$, $P_L$ is chosen to be in coherence with the following lemma, since the action of $\hat \rho^e$ is trivial outside the middle degree, we omit the details):
\begin{equation}\label{gluelag}
P_L(e_i^{\pm,0}, \tilde e_i^{\pm,0})=(f_i^{\pm,0},\tilde f_i^{\pm,0}).
\end{equation}
Note that by definition of the pairs $(e_i^{\pm,0},f_i^{\pm,0})$ resp. $(\tilde e_i^{\pm,0}, \tilde f_i^{\pm,0})$ which are a basis of the  diagonal $\Delta\subset H^n(N,\mathbb{C})=H^n(F^e,·\mathbb{C})\oplus H^n(-F^e,·\mathbb{C})$, the image of $P_L$ is in fact Lagrangian since
\[
\Delta=\{(f,f)\ | \ f\in H^n(F^e,\mathbb{C})\}\subset H^n(F^e,·\mathbb{C}))\oplus H^n(-F^e,\mathbb{C})
\]
is Lagrangian. In fact:
\begin{enumerate}
\item $\Delta$ is orthogonal to $\hat \gamma(\Delta)$ since
\[
\left( (f,f),\hat \gamma(g,g)\right)_{L^2(N)}=\left( f,\gamma g \right)_{L^2(F^e)} +\left( f,-\gamma
g \right)_{L^2(F^e)}=0,
\]
where $(\cdot,\cdot)_{L^2(N)}$ denotes the $L^2$-innerproduct on $\Omega^*(N,\mathbb{C})$ induced by the metric, $(\cdot,\cdot)_{L^2(F^e)}$ the corresponding one on $F^e$. Furthermore one has
\item $\Delta+\hat \gamma (\Delta)=H^n(F^e,\mathbb{C})\oplus H^n(-F^e,·\mathbb{C})$ since
\[
(f,g)=\frac{1}{2}\left( (f+g,f+g)+\hat \gamma(-\gamma f+\gamma g ,-\gamma f+\gamma
g)\right),
\]
for any pair $(f,g) \in H^n(F^e,·\mathbb{C})\oplus H^n(-F^e,\mathbb{C})$.
\end{enumerate}
In fact we can identify ${\rm im}\ P_L$ very precisely.
\begin{lemma}
$L$ coincides with the space of limiting values $\Lambda_{[0,1]\times F^e}$ of extended $L^2$-solutions of $Df=0$ on $[0,1]\times F^e$ relative to the product metric.
\end{lemma}
\begin{proof}
We already mentioned $i\gamma=*$ for ${\rm dim}\ F=2n=4k$. The space of limiting values of extended $L^2$-solutions is isomorphic to the image of $H^{\rm even}([0,1]\times F^e,\mathbb{C})=H^0([0,1],\mathbb{C})\otimes H^{0,n}(F^e,\mathbb{C})$ under $r:=i^*+i^*(*):H^*([0,1]\times F^e,\mathbb{C})\rightarrow H^*(N,\mathbb{C})$, where $i:N \hookrightarrow [0,1]\times F^e$ is the inclusion.  Since the metric on $[0,1]\times F^e$ is product, $*(H^{\rm even}([0,1]\times F^e,\mathbb{C}))$ is of the form $H^1([0,1],\partial [0,1],\mathbb{C}) \otimes H^*(F^e,\mathbb{C})$, hence its pullback is zero: $i^*(*)=0$ on $*(H^{even}([0,1]\times F^e,\mathbb{C}))$. Furthermore $H^0([0,1],\mathbb{C})$ is represented by a constant function, consequently (with the above notation) ${\rm im}\  r = {\rm span}\{\alpha\oplus \overline \alpha\},\ \alpha \in H^{0,n}(F^e,\mathbb{C})$. Using the isomorphism $\beta=(\beta_1,\beta_2):U^* \oplus U \rightarrow H^n(\tilde F^e,\mathbb{C})$ from Lemma \ref{betaisom} we have thus
\[
\Lambda_{[0,1]\times F^e}={\rm span} \ \{e_0 + \overline e_0,\alpha_i +\overline\alpha_i,\omega_i+\overline \omega_i\}_{i =1,\dots,\mu},
\]
where $e_0$ spans $H^0(F^e,\mathbb{C})$, $\{\alpha_i\}_i$, $\{\omega_i\}_i$ span $U$ (resp. $U^*$ as the dual basis) while $\overline e_0,\overline \alpha_i,\overline \omega_i$ denote the corresponding elements on $-F^e$. We can rewrite the above as
\[
\begin{split}
\Lambda_{[0,1]\times F^e}\cap H^n(N,\mathbb{C})={\rm span} \ &\{(\alpha_i +\omega_i)+(\overline \alpha_i+\overline \omega_i),\\&(\alpha_i -\omega_i)+(\overline \alpha_i-\overline \omega_i)\}_{i =1,\dots,\mu},
\end{split}
\]
but this equals exactly the Lagrangian determined by (\ref{gluelag}).
\end{proof}
Since the action of $\hat \rho^e$ on the basis elements associated to $-(F^e)$
is trivial, we have, denote by ${[z^\alpha(i)]}_{\alpha(i) \in \Lambda}$ a monomial base of
$M(f)$ associated to the ${\theta_i}_{i=1,\dots,\mu}$
\begin{equation}\label{explicit}
\begin{split}
\hat \rho^e(\theta_i^+)=e^{2\pi i l(\alpha_i)}\theta_i^+,& \qquad  \hat
\rho^e(\overline \theta_i^+)=\overline \theta_i^+,\\
\hat \rho^e(\omega_i^+)=(e^{2\pi i l(\alpha_i)}-1)i^{n^2}\theta_i^+ + \omega^+_i,&
\qquad \hat \rho^e(\overline \omega_i^+)=\overline \omega_i^+,\\
\hat \rho^e(\overline \theta_j^-)=\overline \theta_j^-,& \qquad \hat
\rho^e(\theta_j^-)=e^{2\pi i l(\alpha_j)} \theta_j^-,\\
\hat \rho^e(\overline \omega_j^-)=\overline \omega^-_j,&
\qquad \hat \rho^e(\omega_j^-)=-(e^{2\pi i
l(\alpha_j)}-1)i^{n^2}\theta_j^- +\omega^-_j,\\
\hat \rho^e(\theta_k^0)=\theta_k^0,& \qquad  \hat
\rho^e(\overline \theta_k^0)=\overline \theta_k^0,\\
\hat \rho^e(\omega_k^0)=\pm i^{n^2+1}\theta_k^0 + \omega^0_k,&
\qquad \hat \rho^e(\overline \omega_k^0)=\overline \omega_k^0,\\
\end{split}
\end{equation}
where $i \in J^+$, $j\in J^-, k  \in J^0$ and for $k\in J^0$ the signs are determined by $\pm=(-1)^{[l(\alpha_k]+n}$. Note that $L=span\{e_j^{\pm,0}+f_i^{\pm,0}, \tilde e_j^{\pm,0}+\tilde f_j^{\pm,0}\}_{j\in J^{\pm,0}}$, so, using the formulas above, direct calculation leads to the following lemma:
\begin{lemma}\label{bladecomp}
With the above notation and for $n=2k, k\in \mathbb{N}$,
\[
P_{(\rho^e)^*(L)}=\bigoplus_{j \in J+}P^+(j)\oplus \bigoplus_{j \in J^-}P^-(j)\bigoplus_{j \in J^0}\oplus P^0(j),
\]
where
\[
\begin{split}
P^-(j)&=P_{(\rho^e)^*(L)}|_{{\rm span} (e_j^-,\tilde e_j^-)}=\begin{pmatrix}e^{-2\pi il(\alpha_j)}&0 \\ 1-e^{-2\pi i l(\alpha_j)}&1 \end{pmatrix},\quad
P^+(j)=P_{(\rho^e)^*(L)}|_{{\rm span} (e_j^+,\tilde e_j^+)}=\begin{pmatrix}1&1-e^{2\pi il(\alpha_j)} \\ 0&e^{2\pi i l(\alpha_j)} \end{pmatrix},\quad \\
P^0(j)&=P_{(\rho^e)^*(L)}|_{{\rm span} (e_j^0,\tilde e_j^0)}=\begin{pmatrix} \frac{1}{1\mp \frac{i}{2}}&\frac{\mp \frac{i}{2}}{1\mp \frac{i}{2}}\\ \frac{\mp \frac{i}{2}}{1\mp \frac{i}{2}}&\frac{1}{1\mp \frac{i}{2}} \end{pmatrix}.
\end{split}
\]
\end{lemma}
\begin{proof}
Using the introduced notations $e_i^{\pm,0},f_i^{\pm,0},\ i \in\{1,\dots,\mu\}$ for the vectors defined above spanning $\hat K^+$ resp. $\hat K^-$, and writing $\hat K^+\oplus \hat K^-=V^+\oplus V^-\oplus V^0$ for the corresponding decomposition,
we deduce from the above formulas (we will explicitly show the formulas for $P^+$ and $P^0$, $P^-$ is similar to $P^+$) the following. First, let $i \in J^+$, then (note $i^{n^2}=1$ for $n=2k, k\in \mathbb{N}$),
\[
\begin{split}
\hat \rho^e(\tilde f_i^+)&=\hat \rho^e(\theta^+_i+\omega_i^+)\\
&=e^{2\pi i l(\alpha_i)}\theta_i^+ + (e^{2\pi i l(\alpha_i)}-1)i^{n^2}\theta_i^+ + \omega^+_i\\
&=(2e^{2\pi i l(\alpha_i)}-1)\theta_i^+ + \omega_i^+,\\
&=e^{2\pi i l(\alpha_i)}\tilde f_i^+ + (e^{2\pi i l(\alpha_i)}-1)e_i^+,\\
\hat \rho^e(\tilde e_i^+)&=\overline\theta_i^+ +\overline \omega_i^+=\tilde e_i^+\\
\hat \rho^e(f_i^+)&=\overline \theta_i^+ -\overline \omega_i^+=f_i^+\\
\hat \rho^e(e_i^+)&=\hat \rho^e(\theta^+_i-\omega_i^+)\\
&=e^{2\pi i l(\alpha_i)}\theta_i^+ - (e^{2\pi i l(\alpha_i)}-1)i^{n^2}\theta_i^+ - \omega^+_i\\
&= \theta_i^+- \omega_i^+=e_i^+.
\end{split}
\]
Noting $L=span\{e_j^{\pm,0}+f_i^{\pm,0}, \tilde e_j^{\pm,0}+\tilde f_j^{\pm,0}\}_{j\in J^{\pm,0}}$ as above one has since
\[
\hat \rho^e(\tilde e_i^+ +\tilde f_i^+)=e^{2 \pi i l(\alpha_i)}\tilde f_i^+ +(e^{2\pi i l(\alpha_i)}-1)e_i^+ + \tilde e_i^+,\quad \hat\rho^e(e_i^++f_i^+)=e_i^+ +f_i^+,
\]
and by substracting $(e^{2\pi i l(\alpha_i)}-1)$ times the second vector from the first that
\[
\hat \rho^e(L)\cap V^+={\rm span}_i\ \left\{\tilde e_i^+-(e^{2\pi i l(\alpha_i)}-1) f_i^+ + e^{2\pi i l(\alpha_i)}\tilde f_i^+,e_i^+ +f_i^+ \right \},
\]
consequently
\[
P^+(j)=\begin{pmatrix}1&1-e^{2\pi il(\alpha_j)} \\ 0&e^{2\pi i l(\alpha_j)} \end{pmatrix},
\]
which was the assertion. Let now $i \in J^0$, then substituting again from formula (\ref{explicit})
\[
\begin{split}
\hat \rho^e(\tilde f_i^0)&=\hat \rho^e(\theta^0_i+\omega_i^0)\\
&=\theta_i^0 + i^{n^2+1}\theta_i^0 + \omega^0_i\\
&=(1 \pm i)\theta_i^0 + \omega_i^0\\
&=(1 \pm \frac{i}{2})\tilde f_i^0 \pm \frac{i}{2}e_i^0\\
\hat \rho^e(\tilde e_i^0)&=\overline\theta_i^0 +\overline \omega_i^0=\tilde e_i^0\\
\hat \rho^e(f_i^0)&=\overline \theta_i^0 -\overline \omega_i^0=f_i^0\\
\hat \rho^e(e_i^0)&=\hat \rho^e(\theta^0_i- \omega_i^0)\\
&=(1 \mp i)\theta_i^0 - \omega^+_i\\
&= \mp \frac{i}{2}\tilde f_i^0+(1 \mp \frac{i}{2})e_i^0,
\end{split}
\]
so summarizing
\[
\begin{split}
\hat \rho^e(\tilde e_i^0 +\tilde f_i^0)&=(1 \pm \frac{i}{2})\tilde f_i^0 \pm \frac{i}{2}e_i^0+ \tilde e_i^0,\\
\hat \rho^e(e_i^0+f_i^0)&=f_i^0 \mp \frac{i}{2}\tilde f_i^0+(1 \mp \frac{i}{2})e_i^0.
\end{split}
\]
To write the span of these vectors as a graph, we substract from the first vector $\frac{\pm i/2}{1\mp i/2}$ times the second, then one has finally
\[
\begin{split}
\hat \rho^e(L)\cap V^0&=\{ {\rm span}\ \{e_i^0 \mp \frac{\frac{i}{2}}{1\mp \frac{i}{2}}\tilde f_i^0 + \frac{1}{1\mp \frac{i}{2}}f_i^0,\\
&\tilde e_i^0+\frac{1}{1 \mp \frac{i}{2}}\tilde f_i^0\mp \frac{\frac{i}{2}}{1\mp \frac{i}{2}} f_i^0\},
\end{split}
\]
which gives the asserted form of $P^0$.
\end{proof}
Using this lemma, the eigenvalues of $P_{\hat \rho^e(L)}$ can be read off resp. calculated as follows:
\begin{equation}\label{spec}
\begin{split}
{\rm spec}(P_{\hat \rho^e(L)})&=\left\{\{1,{\rm exp}(\pm 2\pi i  l(\alpha_j))\}_{j \in J^{\pm}}, \{1, \frac{3}{5}\mp \frac{4}{5}i\}_{j \in J^0}\right\}\\
&=\left\{\{1,{\rm exp}((-1)^{[l(\alpha)]+n}2 \pi i l(\alpha))\}_{\alpha  \in \Lambda: l(\alpha) \notin \mathbb{Z}},\{1, \frac{3}{5}+(-1)^{l(\alpha)+n+1} \frac{4}{5}i\}_{\alpha \in \Lambda:l(\alpha) \in \mathbb{Z}}\right\}.
\end{split}
\end{equation}
where $[\cdot]$ is the integer part function. Now by taking ${\rm log}(re^{it})=ln \ r+it,\  r>0, -\pi < t \leq \pi$ one has
\begin{equation}\label{finaleta}
\begin{split}
m_{H^*(\tilde F^e)}(L,\rho^e(L))&=-\frac{1}{\pi i}{\rm tr} \ {\rm log} (-P_{L}P^*_{ \rho^e(L)})+dim(L\cap \rho^e(L))\quad \\
&=-\frac{1}{\pi i}\sum_{\lambda \in spec(-\phi(L)\phi(\rho^e(L))^*), \lambda \neq -1}{\rm log} \lambda,	\\	
&=\sum_{\alpha \in \Lambda, l(\alpha) \notin \mathbb{Z}}(-1)^{[l(\alpha)]+n}(1-2\{l(\alpha)\}) +\frac{{\rm arg}(-1+\frac{4}{3}i)}{\pi}\sum _{\alpha \in \Lambda, l(\alpha)\in \mathbb{Z}}(-1)^{l(\alpha)+n},
\end{split}
\end{equation}
where ${\rm arg}(re^{i\theta})=\theta \in[0,2\pi), r>0$.
\end{proof}
The above formula proves Proposition \ref{maslov}, which in turn proves Theorem \ref{eta}. Note that the 'non-degenerate part' of the expression for the eta-invariant equals the algebraic eta-invariant given by Nemethi (\cite{nem1},\cite{nem2}) for so-called '$(-1)^n$-hermitian variation structures', specialized for the case of weighted homogeneous polynomials.

\subsection{Eta invariants and spectral flow}
We close this section by relating the above expression for an eta-invariant on $Y$ with our results of \cite{klein3} (see also \cite{klein1}, Chapter 3). There we observed that the variation structure of a quasihomogeneous polynomial $f$ is completely determined by a set of $\mu$ spectral flows ${\rm sf}(\alpha)$ on its Milnor bundle $Y$ associated to a set of monomials $z^\alpha, \alpha \in \mathbb{Z}^{n+1}$, where $\mu$ is the Milnor number of $f$ and $\{z^\alpha\}$ span its Milnor algebra $M(f)$ as a module over $\mathbb{C}$. On the other hand, the sum of the ${\rm sf}(\alpha)/\beta$ is, modulo a Maslov-type number, determined by a difference of eta-invariants. This difference is, by Lemma 3.2.2 of \cite{klein1}, determined by the restriction of the weighted circle action $\sigma$ of $f$ to the boundary of $Y$ and the Cauchy data space of the trivial bundle $\tilde Y_0:=Y_u\times S^1_\delta$ (for some $u\in S^1_\delta$ fixed) resp. its adiabatic limit in $\partial \tilde Y_0\simeq \partial Y$, alternatively one can replace $\tilde Y_0$ by the $\beta$-fold cyclic covering of $Y$. \\
Assume thus we have chosen a diffeomorphism $\Theta: \partial Y\rightarrow \partial \tilde Y_0$ which is an isometry w.r.t. the product metric $g_0$ on $\tilde Y_0$ and that $Y$ as well as $\tilde Y_0$ are equipped with metric collars as in (\ref{product999}) (note that we assume $\tilde Y_0$ having 'left-boundary'). Let $D_{\tilde Y_0}$ be the signature operator on $\tilde Y_0$ associated to $g_0$ (with tangential operator $A$), $L_{\tilde Y_0}$ its Cauchy data space and recall that there is a limit
\[
\lim_{r \rightarrow \infty}L^r_{\tilde Y_0}=F_\nu^+ \oplus W_{\tilde Y_0}\oplus\gamma(W_{\tilde Y_0}^\perp)\oplus d(E^-_\nu)\oplus V_{\tilde Y_0}
\]
where $L^r_{\tilde Y_0}$ is defined by the 'elongation' $\tilde Y_0\cup_{\partial \tilde Y_0}\partial \tilde Y_0\times (-r,0]$, $\nu\geq\nu_0 \in \mathbb{N}$, where $\nu_0$ is the non-resonance-level for $D_{\tilde Y_0}$ and $W_{\tilde Y_0}\subset d(E_\nu^+)\oplus d^*(E_\nu^-)\subset L^2(\Omega^*_{\partial \tilde Y_0})$, so $W_{\tilde Y_0}=r(\mathcal{K}_{\tilde Y_0})$, using notation from \cite{klein3}. Assume now the link $L:= \partial Y_u$, is a rational homology sphere which implies ${\rm ker}(b)=0$, since ${\rm coker}\  i^*:(H^n(Y_u, \partial Y_u,\mathbb{C})\rightarrow H^n(Y_u,\mathbb{C}))\simeq H^n(\partial Y_u,\mathbb{C})=0$. Set $\rho_{\tilde Y_0}=\sigma_{1/\beta}|Y_u\times id:\tilde Y_0\rightarrow \tilde Y_0$ and note that the $\pm 1$-eigenspaces of $b$ on $H^*(Y_u,\mathbb{C})$ induce a splitting $\mathcal{K}_{\tilde Y_0}=\mathcal{K}^{+}_{\tilde Y_0}\oplus \mathcal{K}^{-}_{\tilde Y_0}\oplus \mathcal{K}_0$, where $\mathcal{K}_0$ represents the $0$-form part (applying the results of Appendix B of \cite{klein3} resp. \cite{klein1}), this induces a splitting $r(\mathcal{K}^{\pm}_{\tilde Y_0}\oplus \mathcal{K}_0)=W_{\tilde Y_0}^\pm\oplus V_0= W_{\tilde Y_0}\oplus V_{\tilde Y_0}$, where we absorbed $V_{Y_0}$ in $V_0$, the zero-form part.of $r(\mathcal{K}_{\tilde Y_0})$. Then set
\begin{equation}
\rho_{\tilde Y_0,b}^*|(W_{\tilde Y_0})=\rho_{\tilde Y_0}^*|_{\partial \tilde Y_0}(W_{\tilde Y_0}^+) \oplus (\rho_{\tilde Y_0}^*|_{\partial \tilde Y_0})^t(W_{\tilde Y_0}^-)
\end{equation}
Note that $(\cdot)^t$ means taking the adjoint of an element $\mathcal{B}(L^2(\Omega^*_{\partial \tilde Y_0}))$, the map is well-defined since $\rho_{\tilde Y_0}^*$ preserves the splitting on $\mathcal{K}_{\tilde Y_0}$, representing fibrewise the algebraic monodromy of $f$. Restricting the action of $\rho_{\tilde Y_0,b}$ to $P_-(W_{\tilde Y_0})$, define
\begin{equation}\label{globalproj2}
\rho_{\tilde Y_0,b}^{\pm}:=P_+ +P_-\circ \rho_{\tilde Y_0,b}^* | \in {\rm End}(W_{\tilde Y_0}, W_{\tilde Y_0}\oplus \gamma W_{\tilde Y_0}),
\end{equation}
where $P_\pm=1/\sqrt{2}(Id\mp i\gamma)$ and finally set
\[
L_{\rho^*_b,\tilde Y_0^\infty}= F_\nu^+ \oplus \rho_{\tilde Y_0,b}^{\pm}(W_{\tilde Y_0})\oplus\gamma(W_{\tilde Y_0}^\perp)\oplus d(E^-_\nu)\oplus V_{\tilde Y_0}  \in {\rm Gr}_\infty(A).
\]
For the following, consider for any isotropic subspace $W\subset L^2(\Omega^*_{\partial \tilde Y_0})$ its associated symplectic subspace $W\oplus\gamma W\subset L^2(\Omega^*_{\partial \tilde Y_0})$ and consider a Lagrangian in this symplectic subspace $L\subset W\oplus \gamma W$. Then let $\phi_W(L):P_-(W\oplus \gamma W) \rightarrow P_+(W\oplus \gamma W)$ be the associated isometry. Define furthermore
\[
\tau(f,b)=\sum_{\alpha \in \Lambda:\frac{1}{2}<\{l(\alpha)\}<1} (-1)^{[l(\alpha)]+n}\in \mathbb{Z},
\]
We can then state
\begin{folg}\label{spectraletainv}
Let $\eta(D_{P^+(\Lambda_Y)})$ be the eta-invariant of the odd signature operator $D$ on $Y$ as calculated in Theorem \ref{eta}. Let ${\rm sf}(\alpha)=-\frac{1}{2}{\rm SF}(\alpha)$, where $\{{\rm SF}(\alpha), \ \alpha \in \Lambda \subset \mathbb{N}^{n+1}\}$ is the set of spectral flows introduced in \cite{klein3}. We then have
\begin{equation}\label{spectraleta}
\eta(D_{P^+(\Lambda_Y)})=\sum_{\alpha \in \Lambda, \frac{{\rm sf}(\alpha)}{\beta} \notin \mathbb{Z}}(-1)^{[\frac{{\rm sf}(\alpha)}{\beta}]+n+1}\left(1-2\{\frac{{\rm sf}(\alpha)}{\beta}\}\right) +\frac{{\rm arg}(-1+\frac{4}{3}i)}{\pi}\sum_{\alpha \in \Lambda, \frac{{\rm sf}(\alpha)}{\beta}\in \mathbb{Z}}(-1)^{\frac{{\rm sf}(\alpha)}{\beta}+n+1},
\end{equation}
where again, $[\cdot]$ denotes the integer part function. Assume now that ${\rm ker}(b)=0$. Then one has
\begin{equation}\label{etadecomp}
\eta(D_{P^+(\Lambda_Y)})={\rm  sign}(b)-\frac{1}{\pi i}{\rm tr\ log}(\Phi(L_{\rho^*_b,\tilde Y_0^\infty})\Phi(L_ {\tilde Y_0^\infty})^*)-2\tau(f,b)
\end{equation}
where we have that
\[
\frac{1}{\pi i} {\rm tr\ log}(\Phi(L_{\rho^*_b,\tilde Y_0^\infty})\Phi(L_ {\tilde Y_0^\infty})^*)=
\frac{1}{\pi i} {\rm tr\ log}(\phi_{W_{\tilde Y_0}}(\rho_{\tilde Y_0,b}^{\pm}(W_{\tilde Y_0}))\phi_{W_{\tilde Y_0}}(W_{\tilde Y_0})^*),
\]
so is determined by finite-dimensional expresssions. Note that, using the isometry $\Theta: \partial Y\rightarrow \partial \tilde Y_0$, the ${\rm tr\ log}$-terms in (\ref{etadecomp}) can be regarded as being defined on $L^2(\Omega^*_{\partial Y})$.
\end{folg}
\begin{proof}
The first assertion (\ref{spectraleta}) follows from Theorem \ref{eta} and using
\begin {equation}\label{bla56789}
l(\alpha)=\sum_{i=0}^n(\alpha_i+1)w_i={\rm deg}(z^\alpha)+\sum_{i=0}^n w_i=\frac{{\rm sf}(\alpha)}{\beta}+1
\end{equation}
by the definition of the weighted degree, ${\rm deg}(z^\alpha)$ and ${\rm sf}(\alpha), \alpha \in \Lambda$ in \cite{klein3} resp. Chapter 3 of \cite{klein1}. To prove formula (\ref{etadecomp}), we observe that
\begin{equation}
\begin{split}
\frac{1}{2\pi i} {\rm tr\ log}(\Phi(L_{\rho^*_b,\tilde Y_0^\infty})\Phi(L_ {\tilde Y_0^\infty})^*) &= \sum_{\alpha\in \Lambda:(-1)^{[l(\alpha)]+n}=1}\{2({\rm deg}(z^\alpha(j)) +\sum^\mu_{i=1} w_i)\}\\
&-\sum_{\alpha\in \Lambda:(-1)^{[l(\alpha)]+n}=-1}\{2({\rm deg}(z^\alpha(j)) +\sum^\mu_{i=1} w_i)\}+\tau(b,f).
\end{split}
\end{equation}
Substituting this into the formula in Theorem \ref{eta} using (\ref{bla56789}) gives the assertion.
\end{proof}
{\it Remark.} We conclude  that, modulo the integer $\tau(f,b)$, and by Corollary 5.3 of Appendix A in \cite{klein3} applied to $\tilde Y_0$, (\ref{etadecomp}) determines $\eta(D_{P^+(\Lambda_Y)})$ for the case ${\rm ker}(b)=0$ by topological resp. spectral-invariants of the fibre, namely its signature and its space of $L^2$-harmonic sections {\it and} the geometry of the 'boundary fibration' $\partial Y$, represented by the $1/\beta$-evaluation of the restricted circle action $\sigma$ on the image of
$\mathcal{K}_{\tilde Y_0}$ under $r$, in this sense, modulo the integers, the 'interior' fibration structure of $Y$ is not needed to calculate $\eta(D_{P^+(\Lambda_Y)})$. On the other hand, (\ref{spectraleta}) encodes a certain 'rigidity' of $\eta(D_{P^+(\Lambda_Y)})$, namely, let $Y_\tau, \tau \in [0,1]$ be a smooth family of bundles $f_\tau:Y_\tau\rightarrow S^1$ so that $Y_1=Y$, $\partial Y_\tau$ is isometric to $\partial Y$ for any $\tau \in [0,1]$ and there is a set of global sections $\mathcal{S}_t:=(s_1(t),\dots,s_\mu(t))$ of $\mathcal{H}^n(f_*\Omega^\cdot_{Y_\tau/S^1})$ (using notation from Section 4.2 in \cite{klein1}) so that in analogy to Proposition 4.2.3 in \cite{klein1} there is a set of spectral flows ${\rm SF}(D_\tau, \mathcal{S}_\tau),\tau \in [0,1]$ on $Y$ so that
\begin{equation}\label{rigidity}
{\rm SF}(D_\tau, \mathcal{S}_\tau)={\rm SF}(D_1, \mathcal{S}_1)\ {\rm for\ any}\  \tau\in [0,1].
\end{equation}
On the other hand, assuming that each $Y_\tau$ has totally geodesic fibres diffeomorphic to the Milnor fibre of $f$, it should be possible to derive a similar formula as (\ref{spectraleta}) for any $\tau \in [0,1]$, so that (\ref{rigidity}) implied equality of the corresponding eta-invariants, we leave the details to a further investigation.

\subsection{Brieskorn polynomials}\label{brieskorn}
Consider a Brieskorn singularity which is given by a polynomial $f:(\mathbb{C}^{n+1},0)\rightarrow (\mathbb{C},0)$ as
\begin{equation}\label{bries1}
f= \sum_{i=1}^{n+1}z_i^{a_i},
\end{equation}
where $a_i\in \mathbb{N}_+$ and assuming as before $n=2k$. Consider its Milnor fibration $f:Y\rightarrow S^1_\delta$ defined in (\ref{milnor323}) with Milnor fibre $F$ and the submersion metric $g$ as defined in (\ref{metric}).Consider the intersection form $b:U:=H^n(F,\partial F, \mathbb {C})\rightarrow U^*=H^n(F,\mathbb{C})$ of $F$, note that $b$ is symmetric, the variation mapping $V:U\rightarrow U^*$ and the monodromy $h:U\rightarrow U$ of $f$. We will now follow Nemethi \cite{nem2} and express $b$, $h$ and $V$ in terms of the $\{a_i\}$. For this start with the singularity $z\mapsto z^a$ for $a \in \mathbb{N}_+$, i.e. $n=0$. Then
\begin{equation}\label{briesvar1}
\mathcal{V}(z^a)=\oplus_{k=1}^{a-1}\mathcal{W}_{exp(2\pi i k/a)}(+1)_{n=0},
\end{equation}
where we use the notation for variation structures from \cite{nem2}. Since $\mathcal{W}_\zeta(+1)_{n=0}=(\mathbb{C};1,\zeta, \zeta-1)$, this is equivalent to
\begin{equation}\label{briesvar2}
\mathcal{V}(z^a)=\oplus_{k=1}^{a-1}(\mathbb{C};1,e^{2\pi i k/a}, e^{2\pi ik/a}-1).
\end{equation}
In this situation one has the following (see \cite{nem2}).
\begin{theorem}\label{bries3}
The variation map $V(f)$ and the monodromy $h(f)$ of $f=\sum_i z_i^{a_i}$ satisfy
\[
(V(f),h(f))=\oplus'(V_{\bf k},h_{\bf k}),
\]
where  $\oplus'=\oplus_{k_1=1}^{a_k-1}\dots \oplus _{k_{n+1}=1}^{a_{n+1}-1}, {\bf k}=(k_1,\dots,k_{n+1})$, and
\[
V_{\bf k}=(-1)^{n(n+1)/2}(e^{2\pi i k_1/a_1}-1)\cdots (e^{2\pi i k_{n+1}/a_{n+1}}-1),\ {\rm and}\ h_{\bf k}=e^{2\pi i \sum_{j=1}^{n+1}k_j/a_j}.
\]
Furthermore, since $b=(h-1)V^{-1}$, one has $b(f)=\oplus'b_{\bf k}$, where $b_{\bf k}=(h_{\bf k}-1)/V_{\bf k}$, explicitly
\[
b_{\bf k}=\frac{sin(\pi \sum_jk_j/a_j)}{2^n\prod_j sin(\pi k_j/a_j)}.
\]
\end{theorem}
\begin{proof}
For isolated singularities $g:(\mathbb{C}^{n+1},0)\rightarrow (\mathbb{C},0)$ resp. $h:(\mathbb{C}^{m+1},0)\rightarrow (\mathbb{C},0)$
let $f:(\mathbb{C}^{n+1}\times \mathbb{C}^{m+1},0)\rightarrow (\mathbb{C},0)$ be given by $f(x,y)=g(x)+h(y)$. Then the Sebastiani-Thom-Theorem (\cite{thom}) states if $F_g, F_h, F_f$ are the Milnor fibres of $f,g,h$ respectively, that
\[
H^{n+m+1}(F_f,\mathbb{C})=H^{n}(F_g,\mathbb{C})\otimes H^{m}(F_h,\mathbb{C}),
\]
and that $h_f=h_g\otimes h_h$ if $h_f,h_g,h_h$ denotes the respective monodromies. Furthermore, the Deligne-Sakamoto-Theorem (\cite{delsak}) states that the corresponding Seifert-forms satisfy
\[
S_f=(-1)^{(m+1)(n+1)}S_g\otimes S_h.
\]
Then the first two claims follow directly when one considers that $S(\cdot,\cdot)=<V^{-1}\cdot,\cdot>$, where $<,>$ is the perfect pairing $<,>:U\otimes U^*\rightarrow \mathbb{C}$, $(\alpha,\beta) \mapsto \int_F\alpha\wedge\beta$ on a given Milnor fibre $F$. The formula for $b$ then follows by direct calculation.
\end{proof}
We will now use the arguments of section \ref{etaquas} to prove the analogue of Theorem \ref{eta}.For that, let $P_{Y_0}$ be the Calderon-Projector of the trivial bundle $Y_0=F \times S^1$ as in the last section with respect to product metric. In the following, we set ${\rm arg}(re^{i\theta})=\theta \in[0,2\pi),\ r>0$. For the following theorem set
\[
\Lambda:=\left\{{\bf k} \in \mathbb{Z}^{n+1}|1\leq k_j \leq a_j-1\right \},\quad \Lambda_0:= \left\{{\bf k}\ \in \Lambda| \sum_{j=1}^{n+1}k_j/a_j \in \mathbb{Z}\right\}.
\]
Then write for any subset $\Lambda' \subset \Lambda$ the symbol $\sum_{\Lambda'}$ as the sum over all $n+1$-tuples ${\bf k}\subset \mathbb{Z}^{n+1}$ so that ${\bf k} \in \Lambda'$. Then one has the following.
\begin{theorem}\label{etab}
Let $D_{I-P_{Y_0}}$ be defined as in section \ref{begriffe1}, relative to the signature operator $D$ with respect to the submersion metric $g$ on the Milnor bundle $Y$ of the Brieskorn polynomial $f$ as given in (\ref{bries1}). Then the eta-invariant $\eta(D_{I-P_{Y_0}})$ for $n=2k$ or $n=2k+1, \ k\in \mathbb{N}$ equals:
\[
\begin{split}
\eta(D_{I-P_{Y_0}})  &=(-1)^n\sum_{\Lambda\setminus\Lambda_0} {\rm sign}\ \left({\rm sin}(\pi\sum_{j=1}^{n+1}k_j/a_j)\right)\cdot \left(1-2\{\sum_{j=1}^{n+1}k_j/a_j\}\right) \\
&-\frac{{\rm arg}(-1+\frac{4}{3}i)}{\pi}\sum _{\Lambda_0}(-1)^{{\sum_{j=1}^{n+1}k_j/a_j}+n},
\end{split}
\]
where the first sum represents the contribution of the non-trivial part of the algebraic monodromy $\rho$ of $Y$, the second sum amounts to a summation over a basis for the $\lambda=1$-eigenspace of $\rho$. Note that $\{\cdot\}$ denotes the fractional part. All the other results from Theorems \ref{eta} and \ref{eta2} hold in complete analogy, with the above expression for $\eta(D_{I-P_{Y_0}})$ replacing the general formula in the last section.
\end{theorem}
\begin{proof}
We will examine the case $n=2k$, the odd case is similar (note the total change of sign). Following the argumentation for the case of general quasihomogeneous polynomials in the last section, we have to determine the quantity
\[
m_{H^*(N,\mathbb{C})}(\Lambda_0,\rho^e(\Lambda_0))=-\frac{1}{\pi i}\sum_{\lambda \in {\rm spec}(-\phi(\Lambda_0)\phi(\rho^e(\Lambda_0))^*), \lambda \neq -1}{\rm log}\ \lambda,
\]
where $\Lambda_0 \subset {\rm ker}\ \hat A\simeq H^*(F^e  \times (-F^e),\mathbb{C})$ is the space of limiting values of extended $L^2$-solutions of $Df=0$ on $[0,1]\times F^e$ relative to the product metric.
Analaogous to the last section, let $J^+,J^-\subset {1,\dots,\mu}$ and let ${\bf k}_i$ be an index family on $\Lambda$ so that $b_{{\bf k}_i}$ is positive resp. negative definite, let the corresponding subspaces be spanned by the elements $\{\theta_i\}_{i \in J^{\pm}} \in U$, we write $\theta^{\pm}_j$ if $j \in J^{\pm}$. We now denote the corresponding elements of $U^*$ so that $b, V, h$ have the (diagonal) form from Theorem \ref{bries3} by $\omega_j^{\pm}\in U^*, j \in J^{\pm}$. On the other hand, we write $\theta^0_i \in U, i \in J^0$ if $\theta_i \in {\rm ker}(b)$, $\omega^0 _j \in U^*, j \in J^0$ for the dual base in ${\rm coker} (b)$. Note that in the following, we will replace $\hat \gamma$ by $\tilde \gamma=-\hat \gamma$, which amounts in switching the $\pm i$-eigenspaces of $\hat \gamma$, the total sign change will be taken account for at the end of the caclulation. We then decompose ${\rm ker}\ \hat A \cap H^n(N,\mathbb{C})=({\rm ker}(\tilde \gamma-iI)\oplus{\rm
ker}(\tilde \gamma+iI))\cap H^n(N,\mathbb{C})=:\hat K^+\oplus \hat K^-= (K^+\oplus \overline K^-)\oplus (K^- \oplus \overline K^+)=V^+\oplus V^-\oplus V^0$ as above by
\[
K^+={\rm span}(e_j^+,e_j^-, e_i^0)_{j \in J^{\pm}, i \in J^0},\quad K^-={\rm span}(\tilde f_j^+,\tilde f_j^-, \tilde f_i^0)_{j \in J^\pm, \ i \in J^0}
\]
where (note again $i\gamma=i^{2k+2}(-1)^{k+1}*=*$ for $n=2k$), set now $|b_{{\bf k}_i}|:=({\rm sign} \ b_{{\bf k}_i})b_{{\bf k}_i}$ if $i\in J^\pm$, $|b_{{\bf k}_i}|=1$, if $i \in J^0$,
\[
\begin{split}
e_i^{\pm, 0}&:=\theta^{\pm,0}_i+i\gamma\theta^{\pm,0}_i=\theta_i^{\pm,0}+|b_{{\bf k}_i}|\omega_i^{\pm,0},\\
\tilde f_i^{\pm,0}&:=\theta^{\pm,0}_i-i\gamma\theta^{\pm,0}_i, = \theta_i^{\pm,0}-|b_{{\bf k}_i}|\omega_i^{\pm,0},\ i \in J^{\pm,0}\\
\end{split}
\]
and we use the corresponding basis elements $\overline \theta_i,\overline \omega_i$ on $H^n(-F^e,\partial F^e,\mathbb{C})$ resp. $H^n(-F^e,\mathbb{C})$
\[
\overline K^-={\rm span}(\tilde e_j^+,\tilde e_j^-, \tilde e_i^0)_{j \in J^{\pm}, i \in
J^0},\quad \overline K^+={\rm span}(f_j^+,f_j^-, f_i^0)_{j \in J^{\pm}, i \in J^0}
\]
where
\[
\begin{split}
\tilde e_i^{\pm,0}&:=\overline \theta^{\pm,0}_i-i\gamma \overline \theta^{\pm,0}_i=\overline \theta_i^{\pm,0}-|b_{{\bf k}_i}|\overline \omega_i^{\pm,0},\\
f_i^{\pm,0}&:=\overline \theta^{\pm,0}_i+i\gamma\overline \theta^{\pm,0}_i = \overline \theta_i^{\pm,0}+|b_{{\bf k}_i}|\overline \omega_i^{\pm,0},
\ i \in J^{\pm,0}.
\end{split}
\]
As before we define the Lagrangian $L \subset {\rm ker}\ \hat A$ by specifying
its associated isometry $P_L:\hat K^+\rightarrow \hat K^-$ as
\begin{equation}\label{gluelag2}
P_L(e_i^{\pm,0}, \tilde e_i^{\pm,0})=(f_i^{\pm,0},\tilde f_i^{\pm,0}).
\end{equation}
Then, for the "$+$"-case (suppressing the suffix in the following) using the notation from Theorem \ref{bries3} one gets
\[
\begin{split}
\hat \rho^e(e_i)&=(h_{{\bf k}_i}+|b_{{\bf k}_ i}|V_{{\bf k}_i})\theta_i+|b_{{\bf k}_ i}|\omega_i\\
\hat \rho^e(\tilde f_i)&=(h_{{\bf k}_i}-|b_{{\bf k}_i}|V_{{\bf k}_i})\theta_i-|b_{{\bf k}_i}|\omega_i\\
\hat \rho^e(\tilde e_i)&=\tilde e_i,\quad \hat \rho^e(f_i)=f_i.
\end{split}
\]
so using $b_{\bf k}=(h_{\bf k}-1)/V_{\bf k}$ one gets
\[
\begin{split}
\hat \rho^e(e_i)&=(2h_{{\bf k}_i}-1)\theta_i+|b_{{\bf k}_i}|\omega_i=h_{{\bf k}_i}e_i+(h_{{\bf k}_i}-1)\tilde f_i\\
\hat\rho^e(\tilde f_i)&=\tilde f_i,\quad \hat \rho^e(\tilde e_i)=\tilde e_i,\quad \hat \rho^e(f_i)=f_i.
\end{split}
\]
and consequently
\[
\hat \rho^e(L)\cap V^+= {\rm span}\ \left\{e_i^++\frac{h_{{\bf k}_i}-1}{h_{{\bf k}_i}}\tilde f_i^+ + \frac{1}{ h_{{\bf k}_i}} f_i^+,\tilde e_i^++\tilde f_i^+ \right \},
\]
so writing again for $P_{(\rho^e)^*(L)}:\hat K^+\rightarrow  \hat K^-$ the decomposition
\[
P_{(\rho^e)^*(L)}=\bigoplus_{j \in J+}P^+(j)\oplus \bigoplus_{j \in J^-}P^-(j)\bigoplus_{j \in J^0}\oplus P^0(j),
\]
one gets for the first summand
\[
P^+(i)=P_{(\rho^e)^*(L)}|_{{\rm span} (e_j^+,\tilde e_j^+)}=\begin{pmatrix}\frac{1}{ h_{{\bf k}_i}}&0 \\ \frac{h_{{\bf k}_i}-1}{h_{{\bf k}_i}}&1 \end{pmatrix}.
\]
An analogous computation for the "-"-case gives
\[
\hat \rho^e(L)\cap V^-= {\rm span}\ \left\{\tilde e_i^- -(h_{{\bf k}_i}-1)f_i^- +  h_{{\bf k}_i} \tilde f_i^-,e_i^- + f_i^- \right \},
\]
which implies
\[
P^-(i)=P_{(\rho^e)^*(L)}|_{{\rm span} (e_j^-,\tilde e_j^-)}=\begin{pmatrix} 1&-(h_{{\bf k}_i}-1)\\0& h_{{\bf k}_i} \end{pmatrix}.
\]
For the "0"-case just go back to the formulas in Theorem \ref{bries3}, from which it follows that $\sum_{j=1}^{n+1}k_j/a_j \in \mathbb{Z}$ since $h_{{\bf k}}=1$. Then using the following fomula for $V_{\bf k}$ (see \cite{nem1}) which follows directly from the one given in Theorem \ref{bries3}, one infers that the sign of $V_{\bf k}/(-i)$ is given by $(-1)^{\sum_{j=1}^{n+1}k_j/a_j}$:
\[
V_{\bf k}=(-1)^{n(n+1)/2}(-2i)^{n+1}\oplus_{k_1=1}^{a_1-1}\dots\oplus_{k_n=1}^{a_n-1}e^{\pi i \sum_{j=1}^nk_j/a_j} \cdot \prod_{j=1}^n{\rm sin}\frac{\pi k_j}{a_j},
\]
note that for $n=2k$, $(-1)^{n(n+1)/2}(-i)^{n+1}=-i$. So comparing this with the formula (\ref{nemweight}) which is valid for an appropriate choice of basis in $U$ and its corresponding dual basis with respect to $<\cdot,\cdot>$ and using the calculations in the proof of Lemma \ref{bladecomp} for the "0"-case, one arrives at
\[
P^0(j)=P_{(\rho^e)^*(L)}|_{{\rm span} (e_j^0,\tilde e_j^0)}=\begin{pmatrix} \frac{1}{1\pm \frac{i}{2}}&\frac{\pm \frac{i}{2}}{1\pm \frac{i}{2}}\\ \frac{\pm \frac{i}{2}}{1\pm \frac{i}{2}}&\frac{1}{1\pm \frac{i}{2}} \end{pmatrix},
\]
where the signs are determined by $\pm=-(-1)^{\sum_{j=1}^{n+1}k_j/a_j}$.
\end{proof}

\section{Appendix A}\label{relcohom}
Let $U \subset \mathbb{C}^{n+1}$ be an open set and let
$f:U\rightarrow \mathbb{C}$ be a holomorphic map so that $x \in
\mathbb{C}^{n+1}$ is an isolated singularity, that is $f$ outside $x$ is a
submersion, assume $f(x)=0$. Let $\epsilon$ and $\delta $ be positive real
numbers and  $S=\{u \in \mathbb{C} \mid |u| < \delta \}$, $X= \{z \in
\mathbb{C}^{n+1} \mid |z| < \epsilon,\ f(z) \in S \}$, $X_0=\{z \in X \mid
f(z)=0\}$ so that with $X' =X - X_0$, $S' =S -{0}$ one gets a locally trivial
$C^\infty$-fibration $f: X' \rightarrow S'$. Let $(\Omega^\cdot_{X'},d)$ be the
sheaf complex of holomorphic differential forms on $X'$, then with
$\Omega^i_{X'/S'}=\Omega^i_{X'} /df \wedge \Omega^{i-1}_{X'}$ we get the sheaf
complex of relative differential forms $(\Omega^\cdot_{X'/S'},d)$ on $X'$. By
the Lemma of Poincare and the regularity of $f|X'$ one has a resolution of
$f^{-1}\mathcal{O}_{S'}$ in the category of $(f^{-1}\mathcal{O}_{S'})$-modules
by
\begin{equation}\label{res}
0 \rightarrow f^{-1}\mathcal{O}_{S'} \rightarrow \Omega^0_{X'/S'} \rightarrow
\Omega^1_{X'/S'} \rightarrow \dots \ .
\end{equation}
Here, $f^{-1}\mathcal{O}_{S'}$ is the topological preimage of the sheaf
$\mathcal{O}'$ of holomorphic functions on $S'$. On the other hand, we observe
that the vector spaces $H^i(X_u,\mathbb{C})$, where $X_u$ are the fibres of
$f:X' \rightarrow S'$, are the fibres of the etale space of the sheaf $R^i
f_{*}\mathbb{C}_{X'}$, where for an abelian sheaf $\mathcal{F}$ on $X$ and a mapping
$f:X \rightarrow S$ $R^if_*\mathcal{F}$ is the sheaf on $S$ associated to $V
\subset S, V \mapsto
H^p(f^{-1}(V),\mathcal{F})$ ($R^if_*$ is identical to the right derived functor
of the direct image functor $f_*$ and is calculated by
injective or $f_*$-acyclic resolutions of $\mathcal{F}$, for details see Hartshorne
\cite{hartshorne}). We have the following isomorphism, refer to Looijenga \cite{loo}. Note that the sections of the vectorbundle accociated to
$R^if_{*}\mathbb{C}_{X}$ (with fibres $H^i(X_u,\mathbb{C})$ over $S'$)
constitute the sheaf $R^i f_{*}\mathbb{C}_{X}\otimes _{\mathbb{C}_{S}}\mathcal{O}_{S}$.
\begin{lemma}\label{sect}
With notation as above, the natural map
\[
(R^i f_{*}\mathbb{C}_{X}) \otimes _{\mathbb{C}_{S}}\mathcal{O}_{S}
\longrightarrow R^i f_{*}(f^{-1}\mathcal{O}_{S})
\]
is an isomorphism.
\end{lemma}
Now consider the complex of direct image sheafs $f_*\Omega^\cdot_{X'/S'}$, this
is a complex of $\mathcal{O}_{S'}$-modules, its cohomology sheafs will be
denoted by $\mathcal{H}^p(f_*\Omega^\cdot_{X'/S'})$ for all $p$. The following
result identifies these with the space of sections in the bundle of fibrewise
cohomology groups, for details we refer to Looijenga \cite{loo}.
\begin{prop}\label{deRham}
In the above situation, the fibrewise de Rham evaluation maps
\[
DR_u: \mathcal{H}^p(f_*\Omega^\cdot_{X'/S'})_u \longrightarrow
H^i(X_u,\mathbb{C})
\]
given by integration over the fibre $f^{-1}(u), u \in S'$ are isomorphisms.
Furthermore, they fit together to define a sheaf isomorphism
\[
DR: \mathcal{H}^p(f_*\Omega^\cdot_{X'/S'})_u \longrightarrow (R^i
f_{*}\mathbb{C}_{X'}) \otimes _{\mathbb{C}_{S'}}\mathcal{O}_{S'}.
\]
\end{prop}
\begin{proof}
We will briefly describe the arguments. Note first, that, taking the canocial
soft resolution for the complex $f_*\Omega^\cdot_{X'/S'}$, one has two spectral
sequences with $E_2$-terms
\[
'E_2^{p,q}  = \mathcal{H}^p(R^qf_*\Omega^\cdot_{X'/S'}), \quad ''E_2^{p,q}=R^p
f_*(\mathcal{H}^q(\Omega^\cdot_{X'/S'})),
\]
both converging to the cohomology of the full complex,
$\mathbb{R}^\cdot(\Omega^\cdot_{X'/S'})$.
Here, $R^qf_*\Omega^\cdot_{X'/S'}$ denotes the complex
$R^qf_*(\Omega^i_{X'/S'})_{i \in \mathbb{Z}}$. Since $f$ is Stein, the first
spectral sequence degenerates which gives
\[
\mathcal{H}^p(f_*\Omega^\cdot_{X'/S'})\simeq
\mathbb{R}^p(\Omega^\cdot_{X'/S'}).
\]
On the other hand, considering the resolution (\ref{res}), we also have
$\mathcal{H}^p(\Omega^\cdot_{X'/S'})=0,\ p> 0$, that is, the second spectral
sequence degenerates, giving
\[
R^pf_*(f^*\mathcal{O}_{S'})\simeq  \mathbb{R}^p(\Omega^\cdot_{X'/S'}).
\]
Putting this together and using Lemma \ref{sect}, we arrive at the assertion.
\end{proof}
Note that in the above, we worked outside the 'critical set', that is, over
$S'$, which implied that $\mathcal{H}^p(\Omega^\cdot_{X'/S'})=0,\ p> 0$.  Now
note that the sheaf complex of relative differential forms is equally well
defined on $X$ over $S$, so for further use we state the following refinement of
Lemma \ref{deRham}, we will only sketch its proof, for details see
Looijenga \cite{loo}, Greuel \cite{greu} or Brieskorn \cite{bries}.
\begin{lemma}\label{germs}
Let $f:X\rightarrow S$ be a good Stein representative of a smoothing of an
isolated singularity as described above. Then, after possibly shrinking $S$
we have $\mathcal{H}^p(f_*\Omega^\cdot_{X/S})=0,\ n >p> 0$ and
$\mathcal{H}^n(f_*\Omega^\cdot_{X/S})$ is a free $\mathcal{O}_S$-module
of rank $\mu$, where $\mu$ is the $n$-th Betti number of a Milnor fibre. The
former is fitting in the exact sequence
\[
0 \rightarrow R^i f_{*}\mathbb{C}_{X} \otimes
_{\mathbb{C}}\mathcal{O}_{S}\xrightarrow
{\alpha^n}\mathcal{H}^n(f_*\Omega^\cdot_{X/S})
\xrightarrow {\beta^n}f_*\mathcal{H}^n(\Omega^\cdot_{X/S})\rightarrow 0.
\]
whgich implies lemma \ref{deRham}. Furthermore, in $0 \in S$, there is a
canonical isomorphism
\begin{equation}\label{stalk2}
\beta^n:\mathcal{H}^p(f_*\Omega^\cdot_{X/S})_0 \longrightarrow
f_*\mathcal{H}^p(\Omega^\cdot_{X/S})_0=H^p(\Omega^\cdot_{X/S,x})
\end{equation}
for $p > 0$.
\end{lemma}
\begin{proof}
The first thing to prove (\cite{loo}, Prop. 8.5) is the long exact sequence for
$p>0$
\[
.. \rightarrow R^p f_{*}\mathcal{H}^0(\Omega^\cdot_{X/S})\xrightarrow
{\alpha^p}\mathcal{H}^p(f_*\Omega^\cdot_{X/S})
\xrightarrow {\beta^p}f_*\mathcal{H}^p(\Omega^\cdot_{X/S})\rightarrow  R^p
f_{*}\mathcal{H}^0(\Omega^\cdot_{X/S})\rightarrow ..
\]
Lemma \ref{deRham} then follows from
$f_*\mathcal{H}^p(\Omega^\cdot_{X'/S'})=0$,
$\mathcal{H}^0(\Omega^\cdot_{X'/S'})=f^{-1}\mathcal{O}_{S'}$ and Lemma
\ref{sect}.
Then one proves that if $X''\subset X$ so that $f|X''$ is also a Stein
representative then the restriction homomorphism between the corresponding exact
sequences is an isomorphism, taking direct limits, one infers that $\beta^p$ is
an isomorphism. Now in \cite{loo}, Prop 8.20, one proves that
one has an {\it exact} sequence of stalks
\[
0 \rightarrow \mathcal{O}_{S,0} \rightarrow \mathcal{O}_{X,x} \rightarrow \Omega^0_{X/S,x}
\rightarrow
\Omega^1_{X/S,x} \rightarrow \dots \rightarrow \Omega^n_{X/S,x}
\]
and $\Omega^n_{X/S,x}/d\Omega^{n-1}_{X/S,x}$ is free of rank $\mu$ (as a
$\mathcal{O}_{S,0}$-module). Then from (\ref{stalk2}), the exact sequence
\begin{equation}\label{stalk3}
\mathcal{H}^p(\Omega^\cdot_{X/S,x})\rightarrow
\Omega^n_{X/S,x}/d\Omega^{n-1}_{X/S,x}\xrightarrow {d}\Omega^{n+1}_{X/S,x}
\end{equation}
and the fact that $\mathcal{H}^p(f_*\Omega^\cdot_{X/S})$ is coherent, it already
follows that for sufficiently small $S$, $\mathcal{H}^p(f_*\Omega^\cdot_{X/S})$
is a free $\mathcal{O}_S$-module of rank $\mu$ for $p=n$ and is trivial for
$0<p<n$. Note again that for $S$ small enough one has
$\mathcal{H}^0(\Omega^\cdot_{X/S})=f^{-1}\mathcal{O}_{S}$, so $R^p
f_{*}\mathcal{H}^0(\Omega^\cdot_{X/S})$ may be identified with
$R^i f_{*}\mathbb{C}_{X} \otimes_{\mathbb{C}}\mathcal{O}_{S}$.
\end{proof}
Note that from the above Lemma and the sequence (\ref{stalk3}) it follows that a
basis spanning the $\mathcal{O}_{S,0}$-module
$\Omega^n_{X/S,x}/d\Omega^{n-1}_{X/S,x}$ will already span the coherent
$\mathcal{O}_S$-module $\mathcal{H}^p(f_*\Omega^\cdot_{X/S})$, provided $S$ is
small enough. However, $df: \Omega^n_{X/S}\rightarrow \Omega^{n+1}_X\simeq
\mathcal{O}_X$, is only an isomorphism outside $\{x\}$, so if $j:X'\rightarrow
X$ denotes the inclusion, we have in our case (see \cite{loo})
$\omega_f:=j_*j^{-1}\Omega^n_{X/S}\simeq \Omega^{n+1}_X$ and one has the
sequence:
\[
0\rightarrow \Omega^n_{X/S,x}\rightarrow \omega_{f,x}\rightarrow
\omega_{f,x}\otimes\mathcal{O}_{\{x\},x}
\]
i.e. $\omega_f$ and $\Omega^n_{X/S}$ coincide outside of $\{x\}$. One then has
the exact sequence
\begin{equation}\label{dualizing}
0\rightarrow \Omega^n_{X/S,x}/d\Omega^{n-1}_{X/S,x}\rightarrow
\omega_{f,x}/d\Omega^{n-1}_{X/S,x}\rightarrow
\omega_{f,x}\otimes\mathcal{O}_{\{x\},x}
\end{equation}
so $\omega_{f,x}/d\Omega^n_{X/S,x}$ is also a free (\cite{loo}, Prop. 8.20)
$\mathcal{O}_{S,0}$-module of rank $\mu$ (note that $\mathcal{O}_{\{x\},x}\simeq \mathcal{O}_{\mathbb{C}^{n+1},x}/(\frac{\partial f}{\partial z_0},\dots,\frac{\partial f}{\partial
z_n})\mathcal{O}_{\mathbb{C}^{n+1},x}$). Then identifying $\omega_{f,x}$ with
$\mathcal{O}_{\mathbb{C}^{n+1},0}$ by means of $\alpha \mapsto df\wedge \alpha$
there is a correspondence of $d\Omega^{n-1}_{X/S,x}$ with a certain $\mathbb{C}\{f\}$ submodule of
$\mathcal{O}_{\mathbb{C}^{n+1},x}$ which we denote by $\hat M(f)$. For $f$
quasihomogeneous, that is, there are positive integers
$\beta_0,\dots,\beta_n,\beta$ so that $f$ is a $\mathbb{C}$-linear combination
of monomials $z_0^{i_0}\dots z_n^{i_n}$ so that $i_0\beta_0+\dots
+i_n\beta_n=\beta$ one deduces that $\mathcal{O}_{\mathbb{C}^{n+1},x}/M(f)$
coincides with $\mathcal{O}_{\mathbb{C}^{n+1},x}/(\frac{\partial
f}{\partial z_0},\dots,\frac{\partial f}{\partial
z_n})\mathcal{O}_{\mathbb{C}^{n+1},x}$ and
a basis fort the latter module can be chosen to consists of monomials
$\alpha_1,\dots,\alpha_r$, so that for every $\alpha_j$ there is a number $d_j$
such that $\alpha_j=z_0^{i_0}\cdot \dots\cdot z_n^{i_n}$ with
$i_0w_0+\dots+i_nw_n=d_j$ where $w_i=\beta_i/\beta$ ($d_j$ will be called the
degreee of $\alpha_j$). Summing up, we have (\cite{loo})
\begin{lemma}\label{poly}
For $f:X\rightarrow S$ quasihomogeneous with $0 \in \mathbb{C}^{n+1}$ an
isolated singularity there are global sections $\phi_1,\dots,\phi_\mu$ of $\mathcal{H}^i(f_*\Omega^\cdot_{X/S})$
which represent a basis of $H^n(X_u,\mathbb{C})$ for any $u \in S'$ that can be represented by monomials
$\alpha_1,\dots,\alpha_\mu \in \mathbb{C}[z_0\dots,z_{n}]$ by the correspondence $\phi\mapsto$ [coefficient of $df\wedge \phi_i$]. Here, $\mu$ is the Milnor number of $f$. These monomials project onto a $\mathbb{C}$-basis of $\mathcal{O}_{\mathbb{C}^{n+1},0}/(\frac{\partial f}{\partial z_0},\dots,\frac{\partial f}{\partial z_n})\mathcal{O}_{\mathbb{C}^{n+1},0}$.
\end{lemma}
We finally note that $R^i f_{*}\mathbb{C}_{X'} \otimes_{\mathbb{C}}\mathcal{O}_{S'}$ carries a canonical flat connection, the Gauss-Manin connection. Using the correspondence describes in lemma \ref{germs} one can extend this to sections of the sheaf $\mathcal{H}^p(f_*\Omega^\cdot_{X/S})$. For this, one sets over $S'$ if $\omega \in \mathcal{H}^p(f_*\Omega^\cdot_{X/S})$, so $d\omega=df\wedge \tilde \omega$ for a certain $\tilde \omega \in f_*\Omega^n_{X}$,
\[
\begin{split}
\nabla_\psi\omega:=\mathcal{L}_\psi(\omega)&=i_\psi d\omega\ {\rm mod}(df_*\Omega^{n-1}_X)\\
&=\tilde \omega \ {\rm mod}(df\wedge f_*\Omega^ {n-1}_X+df_*\Omega^{n-1}_X),
\end{split}
\]
where $\psi$ lifts $\frac{\partial}{\partial z}$, so $\nabla_\psi\omega=dz\otimes [\tilde \omega]$. It is then well-known (\cite{loo}) that $\nabla$ extends 'regular-singular' along $S$ and that $\nabla$ maps each of the the modules $ \mathcal{H}^p(\Omega^\cdot_{X/S,x})\subset  \Omega^n_{X/S,x}/d\Omega^{n-1}_{X/S,x}\subset \omega_{f,x}/d\Omega^{n-1}_{X/S,x}$ into the next in the chain of inclusions.\\
{\it Remark.} Instead of working with the sheaf $\omega_{f}$ whose quotient by $d\Omega^{n-1}_{X/S}$ at $x$ fits into the short exact sequence (\ref{dualizing}) and which is isomorphic to $\Omega^n_{X/S}$ outside of $\{x\}$ (this approach goes back to Looijenga \cite{loo}) we will in the following also frequently refer to a more common definition of the Brieskorn lattice $\mathcal{H}''$ which is equivalent to the above for our case of an isolated singularity. $\mathcal{H}''$, understood as a sheaf over $S$, fits into the exact sequence (see \cite{bries})
\[
 0 \rightarrow f_*\Omega^n_{X/S}/d(f_*\Omega^{n-1}_{X/S}) \xrightarrow{df \wedge} \mathcal{H}'':=f_*\Omega^{n+1}_{X}/df\wedge d(f_*\Omega^{n-1}_{X/S}) \rightarrow  f_*\Omega^{n+1}_{X/S}  \rightarrow 0
\]
while its stalk at $s=0$ is isomorphic to $\mathcal{H}''_0=\Omega^{n+1}_{X,x}/df\wedge d(\Omega^{n-1}_{X/S,x})$ and, by the above sequence, $\mathcal{H}''$ coincides with $\mathcal{H}^n(f_*\Omega^\cdot_{X/S})$ outside of $0$.\\
Let now $\omega$ be a section of $\mathcal{H}''$ over a neighbourhood $S\subset \mathbb{C}$ around $s=0$ and consider this as a section of $\mathcal{H}^n(f_*\Omega^\cdot_{X/S})$ on $S'=S\setminus \{0\}$, those sections are called by Varchenko \cite{varchenko2} 'geometric sections'. Consider now over $S'$ the locally constant sheaf $\underline H^n=R^i f_{*}\mathbb{C}_{X'}$ and its dual $\underline H_n={\rm Hom}(\underline H^n, \mathbb{C})$ as the sheaf of homomorphisms from $\underline H^n$ to $\mathbb{C}$. For any $s \in S'$ we have $\underline H_n(s)\simeq H_n(X_s, \mathbb{C})$ and there is a natural isomorphism $T: \underline H^n(\gamma(0))\rightarrow \underline H^n(\gamma(1))$ for any smooth path $\gamma:[0,1]\rightarrow S'$ induced by the fibre bundle structure of $X'$. I.e. we obtain a morphism
\[
M:\pi_1(S', s)\rightarrow {\rm Aut}(\underline H^n(s))
\]
whose evaluation $M(1)$ at the generator $1 \in \pi_1(S', s)$ we will call the monodromy $M$ of $f$. We can 'sheafify' these topological constructions and the result coincides with he Gauss-Manin connection restricted to $\mathcal{H}^p(f_*\Omega^\cdot_{X'/S'})$ described above. Dualizing the above topological notion of parallel transport to isomorphisms $T^*: \underline H_n(\gamma(0))\rightarrow \underline H_n(\gamma(1))$ for smooth paths $\gamma$ as above, we can consider a covariant constant (multivalued) section $\delta$ of $\underline H^n$ over $S'$. Let $s(\omega)$ be the section of $\underline H^n$ over $S'$ represented by $\omega$, then by a Theorem of Malgrange (\cite{malg2}) one has for the dual pairing of $\delta$ and $s(\omega)$ over $S'$:
\begin{theorem}\label{malgrangeth}
The series
\begin{equation}
(s(\omega), \delta)(t)= \sum_{\alpha}\sum_{k=0}^n\frac{1}{k!}a_{k,\alpha}t^\alpha ({\rm ln}\ t)^k
\end{equation}
where $\alpha >-1$, $e^{-2\pi i\alpha}$ is an eigenvalue of $M$, converges in each sector $a< {\rm arg} t< b$ if $0<|t|$ is sufficiently small in $S'$.
\end{theorem}
Furthermore, the coefficients $a_{k,\alpha}$ depend linearly on the section $\delta$, which implies (cf. Varchenko \cite{varchenko2}) there is a set of covariantly constant sections $A^\omega_{k,\alpha}(t)$ of $\underline H^n(t)$ over (a eventually smaller) $S'$ so that
\[
s(\omega)(t)=\sum_{\alpha}\sum_{k=0}^n\frac{1}{k!}A^\omega_{k,\alpha}(t)t^\alpha ({\rm ln}\ t)^k
\]
and by \cite{varchenko3} for any $t, k, \alpha$ the $A^\omega_{k,\alpha}(t)$ belong to the generalized eigenspace of $M$ associated to $e^{-2\pi i\alpha}$. Then one calls the weight $\alpha(\omega)$ of $\omega$ the number $\alpha(\omega):=\{{\rm  min}(\alpha)|{\rm at \ least \ one \ of\ the\ sections \ } A^\omega_{0,\alpha}(t),\dots, A^\omega_{n,\alpha}(t)\neq 0\}$. Then the {\it principal part} of $s(\omega)$ is defined as
\[
s_{max}(\omega)(t)=\sum_{k=0}^n\frac{1}{k!}A^\omega_{k,\alpha(\omega)}(t)t^{\alpha(\omega)} ({\rm ln}\ t)^k,
\]
and the prinicipal parts of geometric sections of one weight are linearly independent at all points $t\in S'$ if they are at one point $t$. Then one calls the Hodge filtration of each $\underline H^n(t)$ the sequence of subspaces $\{F^p\}$ in $\underline H^n(t)$ generated by the principal parts of all geometric sections $\omega$ of $f$, evaluated at $t$, so that $\alpha(\omega)\leq n-p$. Note that since $s(f\omega)=ts(\omega)$, we have $F^{p+1}\subset F^{p}$ and the $F^p$ form a subbundle of $\underline H^n$ (cf. \cite{varchenko3}). We can now define the {\it spectrum} of a singularity due to Varchenko \cite{varchenko2}:
\begin{Def}\label{spectrumdef}
Let the principal parts of sections $\omega^p_1, \dots, \omega^p_{j(p)} \in \mathcal{H}''$ be a basis of $F^p/F^{p+1}$, i.e. their weights satisfy $\alpha(\omega^p_j)\in (n-p-1,n-p]$. Then the union of all such weights $\alpha(\omega^p_j)$ for all geometric sections $\omega^p_j$ and $(p,j)$ satisfying the above is called the spectrum of $f$.
\end{Def}
Note that by \cite{varchenko3}, at each point $t \in S'$, $F^p$ is left invariant by the semismple part $M_s$ of $M$. So the spectrum of $f$ is just the union over all $p$ of the set of numbers $n-l_p(\lambda)$ being asscociated to each eigenvalue $\lambda$ of the action of $M_s$ on $F^p/F^{p+1}$ that satisfy $exp(2\pi i l_p(\lambda))=\lambda$ and the normalization condition $p\leq l_p(\lambda)< p+1$. It is an unordered collection of $\mu$ numbers, $\mu$ being the Milnor number of $f$. Note that for an isolated quasihomogeneous singularity, the spectrum can be expressed in terms of a monomial basis $z^{\alpha(i)}$ of $M(f)$ as $\gamma_i= l(\alpha(i))-1$, hence written as a 'divisor', ${\rm sp}(f)=\sum_{\alpha(i) \in A} \left(l(\alpha(i))-1\right) \in \mathbb{Z}^{(\mathbb{Q})}$ (cf. \cite{kulikov}). We now have the following celebrated theorem due to Varchenko \cite{varchenko2}.
\begin{theorem}\label{definvspec}
The spectrum of $f$ having an isolated singularity at the origin does not change under a deformation (depending holomorphically on the deformation parameters) of $f$ leaving its Milnor number unchanged (these deformations we will refer to as $\mu$-constant deformations).
\end{theorem}
Note that the spectrum of an isolated holomorphic singularity $f:\mathbb{C}^{n+1}\rightarrow \mathbb{C}$ is a topological invariant for $n\leq 2$, while the (integer) Seifert form (\ref{seifertform}) determines and is determined by the topological type of $f$ for $n\geq 3$ (cf. Saeki \cite{saeki}). However, as Saeki shows, the former result remains true for $n=3$ if $f$ is quasihomogeneous, moreover we have (cf. \cite{saeki}, \cite{varchenko2}):
\begin{theorem}\label{foureq}
Let $f$ and $g$ be quasihomogeneous polynomials with an isolated singularity at the origin in $\mathbb{C}^{n+1}$ for $n\geq 1$. Then the following four are equivalent:
\begin{enumerate}
\item  $f$ and $g$ are connected by a $\mu$-constant deformation.
\item  $f$ and $g$ are connected by a topologically constant deformation.
\item $f$ and $g$ have the same weights.
\item $f$ and $g$ have the same spectrum.
\end{enumerate}
\end{theorem}
To close this section we briefly discuss the term $\epsilon$-hermitian variation structure, introduced by Nemethi \cite{nem1}, \cite{nem2}. For this let $U$ be a complex vector space with a complex conjugation $\overline \cdot$. $U^*={\rm Hom}_{\mathbb{C}}(U,\mathbb{C})$ will denote its dual. There is a natural isomorphism $\vartheta:U\rightarrow U^{**}$ given by $\theta(u)(\phi)=\phi(u)$. If $\phi \in {\rm Hom}_{\mathbb{C}}(U,U')$ then $\overline \phi$ is defined as $\phi(x)=\overline \phi(\overline x)$. The dual of $\phi$ is as usual defined as $\phi^*:U'^*\rightarrow U^*$ by $\phi^*(\psi)=\psi\circ \phi$. Understanding this, we define:
\begin{Def}\label{variation}
An $\epsilon$-hermitian variation structure is a quadrupel $(U,b,h,V)$ so that $U$ is as above, $\epsilon \in \{1,-1\}$ and one has
\begin{enumerate}
\item $b:U \rightarrow U^*$ is a $\mathbb{C}$-linear endomorphism so that $\overline {b^*\circ \theta}=\epsilon b$,
\item $h$ is a $b$-orthogonal automorphism of $U$ that is $\overline {h^*}\circ b\circ h= b$,
\item $V:U^*\rightarrow U$ is a $\mathbb{C}$-linear endomorphism and $\overline {\theta^{-1}\circ V^*}=-\epsilon V\circ\overline h^*$ and $V\circ b=h-id$.
\end{enumerate}
\end{Def}
Note that if $b$ is an isomorphism, then $V=(h-id)b^{-1}$, so the variation structure is determined by $(U,b,h)$ alone. If $V$ is an isomorphism, the structure is called {\it simple}. Now, if $f:(\mathbb{C}^{n+1},0)\rightarrow (\mathbb{C},0)$ is an isolated hypersurface singularity, then setting $U=H^n(F,\partial F,\mathbb{C})$, where $F$ is the Milnor fibre, $b$ the complex $(-1)^n$-symmetric intersection form, $h$ the complexified monodromy and $V$ the complexified variation mapping ($V$ is defined as $V:U^*\rightarrow U$ by $V(\omega)=[h(\omega)-\omega]\in U$, where $h: F\rightarrow F$ is a smooth representative of the geometric monodromy of $f$ fixing the boundary pointwise) furnish a $(-1)^n$-hermitian variation structure, denote it by $\mathcal{V}(f)$. This variation structure is always simple, as follows from the Gysin sequence of $f$ by applying the $5$-Lemma. Any basis$\{e_i\}\in U$ defines a dual base $\{e_i^*\}_i$ in $U^*$, that is $e_i^*(e_j)=1$ if $i=j$ and $0$ otherwise. Using the matrix representation with respect to such a basis, we can give the following example of simple $(-1)^n$-hermitian variation structures, we assume $|\lambda|=1$:
\begin{equation}\label{nemweight2}
\begin{split}
\mathcal{W}_\lambda(\pm 1)&=(\mathbb{C}, \pm i^{-n^2},\lambda,\pm(\lambda-1)i^{n^2}),\quad if \lambda \neq 1,\\
\mathcal{W}_1(\pm 1)&=(\mathbb{C}, 0,1_\mathbb{C},\pm i^{n^2+1}),\quad if \lambda = 1.
\end{split}
\end{equation}
It is well-known (see \cite{nem2}) that any simple variation structure with diagonalizable monodromy $h$ is a direct sum of indecomposable structures $\mathcal{W}_\lambda(\pm 1),\  \lambda \in S^1$. Now let $S$ be the Seifert form of an isolated singularity $f$ as given by
\begin{equation}\label{seifertform}
S(a,b)=<V^{-1}a, b>, \quad a,b \in U,
\end{equation}
and let $K_f$ be its link. Then by a result of Durfee (\cite{durfee}) we have the following.
\begin{theorem}
Let $n\geq 3$. Then the Seifert form $S$ of $f$ is determined and determines the isotopy class of $K_f\subset S^{2n+1}$ as a fibred knot. Furthermore if $b$, $h$ are intersection form and monodromy automorphism of $f$ as introduced above, one has $b=S+(-1)^nS^t$ and $h=(-1)^{n-1}S^{-1} S^t$, i.e. the variation structure $(U,b,h,V)$ of $f$ is determined by the Seifert form.
\end{theorem}
Since for $n\geq 3$ the Seifert form is by the above determined by and determines the topological type of $f$, $V_f=f^{-1}(0)$ being toplogically a cone over its link, it follows that in these dimensions the variation structure of $f$ is determined by its topological type, that is the homeomorphism type of the pair $(\mathbb{C}^{n+1},V_f)$.

\end{document}